\documentclass[10pt,reqno]{amsart}
\usepackage{amsthm,amssymb,mathrsfs,color,mathtools,empheq,esint}
\usepackage{pinlabel}
\usepackage[shortlabels]{enumitem}
\mathtoolsset{showonlyrefs}

\usepackage{amsthm,amssymb,mathrsfs,mathtools,empheq,esint}

\usepackage{comment}

\usepackage{tensor}

\usepackage{graphicx}
\usepackage{xcolor} 
\usepackage{tensor}
\usepackage{slashed}
\usepackage{cite}

\newcommand{\cl}{\mathcal}

\newtheorem{theorem}{Theorem}[section]
\newtheorem*{theorem*}{Theorem}
\newtheorem{corollary}[theorem]{Corollary}

\newtheorem{lemma}[theorem]{Lemma}
\newtheorem{proposition}[theorem]{Proposition}
\newtheorem*{proposition*}{Proposition}
\newtheorem{conjecture}[theorem]{Conjecture}
\newtheorem*{conjecture*}{Conjecture}

\theoremstyle{definition}

\newtheorem{remark}[theorem]{Remark}

\newcommand{\rar}{\rightarrow}
\numberwithin{equation}{section}

\def\bR {\mathbb{R}}
\def\bS {\mathbb{S}}

\def\cH {\mathcal{H}}

\def\la {\langle}
\def\ra {\rangle}

\newcommand{\tx}[1]{\mathrm{#1}}

\newcommand{\conj}[1]{\overline{#1}}

\newcommand{\lam}{\lambda}

\newcommand{\ud}{\mathrm{\,d}}
\newcommand{\vd}{\mathrm{d}}
\newcommand{\udr}{\,r\vd r}
\newcommand{\vD}{\mathrm{D}}

\newcommand{\uln}[1]{{\underline{ #1 }}}

\definecolor{deepgreen}{cmyk}{1,0,1,0.5}

\newcommand{\A}{\mathcal{A}}

\newcommand{\E}{\mathcal{E}}

\newcommand{\HH}{\mathcal{H}}

\newcommand{\LL}{\mathcal{L}}

\newcommand{\R}{\mathbb{R}}
\newcommand{\Sp}{\mathbb{S}}
\newcommand{\Z}{\mathbb{Z}}

\newcommand{\bs}[1]{\boldsymbol{#1}}

\newcommand{\al}{\alpha}

\newcommand{\de}{\delta}

\newcommand{\te}{\theta}

\newcommand{\s}{\sigma}

\newcommand{\z}{\zeta}

\newcommand{\De}{\Delta}

\newcommand{\p}{\partial}
\newcommand{\na}{\nabla}

\newcommand{\supp}{\operatorname{supp}}

\newcommand{\ext}{\operatorname{ext}}

\makeatletter

\newcommand{\Rmnum}[1]{\expandafter\@slowromancap\romannumeral #1@}
\makeatother

\newcommand{\ti}{\widetilde}

\newcommand{\U}{\underline}

\newcommand{\ang}[1]{\left\langle{#1}\right\rangle}
\newcommand{\abs}[1]{\left\lvert{#1}\right\rvert}

\newcommand{\EQ}[1]{\begin{equation}\begin{split} #1 \end{split}\end{equation}}
\newcommand{\pmat}[1]{\begin{pmatrix} #1 \end{pmatrix}}

\newcommand{\Del}[1]{}

\numberwithin{equation}{section}

\newcommand{\mif}{{\ \ \text{if} \ \ }}

\newcommand{\mas}{{\ \ \text{as} \ \ }}

\definecolor{green}{rgb}{0,0.8,0} 

\newcommand{\bbR}{\mathbb R}

\newcommand{\calZ}{\mathcal Z}

\newcommand{\Lam}{\Lambda}

\begin{document}

\title[Bubbling wave maps with prescribed radiation]{Dynamics of bubbling wave maps with prescribed radiation}
\author{Jacek Jendrej}
\author{Andrew Lawrie}
\author{Casey Rodriguez}

\begin{abstract}
We study energy critical one-equivariant wave maps taking values in the two-sphere. It is known that any finite energy wave map that develops a singularity does so by concentrating the energy of (possibly) several copies of the ground state harmonic map at the origin. If only a single bubble of energy is concentrated, the solution decomposes into a dynamically rescaled harmonic map plus a term that accounts for the energy that radiates away from the singularity. In this paper, we construct blow up solutions by prescribing the radiative component of the map. In addition, we give a sharp classification of the dynamical blow up rate for every solution with this prescribed radiation. 



\end{abstract}

\thanks{J. Jendrej was supported by  ANR-18-CE40-0028 project ESSED.  A. Lawrie was supported by NSF grant DMS-1700127 and a Sloan Research Fellowship. C. Rodriguez was supported by an NSF Postdoctoral Research Fellowship DMS-1703180. }

\maketitle

\section{Introduction}
\label{sec:intro}

We consider wave maps  from (1+2)-dimensional Minkowski space $\R^{1+2}_{t, x}$  to the $2$-sphere, $\bS^2$, with 1-equivariant symmetry. In this setting, the objects under consideration  are formal critical points of the Lagrangian action, 
\EQ{
\cl A(U) = \frac{1}{2} \int_{\R^{1+2}_{t, x}} \Big( {-} \abs{\p_t U(t, x)}^2 + \abs{\na U(t, x)}^2 \Big) \, \ud x  \ud t
}
where  $U: \R^{1+2}_{t, x} \to \Sp^2 \subset \R^3$ are maps that take the restricted form, 
\EQ{ \label{eq:fullmap} 
U(t, r, \te) = ( \sin u(t, r) \cos \te, \sin u(t, r) \sin \te, \cos u(t, r)) \in \Sp^2\subset \R^3
}
Here $(r, \te)$ are polar coordinates on $\R^2$, and $u(t, r)$ is a radially symmetric function. 
The Cauchy problem for $1$-equivariant wave maps reduces to a scalar semilinear wave equation for the polar angle $u$:
\EQ{\label{eq:wmk}
      &\partial_t^2 u - \p_r^2 u - \frac{1}{r} \p_r u + 
      \frac{\sin 2 u}{2 r^2} =0, \quad (t,r) \in \bbR \times (0,\infty),\\
      &(u(t_0), \partial_t u(t_0)) = (u_0, \dot u_0), \qquad  t_0 \in \bR. 
}

Wave maps are referred to as nonlinear $\s$-models in the high energy physics literature, see e.g.,~\cite{MS, GeGr17}. From the mathematical point of view, they are a canonical example of a geometric wave equation as they simultaneously generalize the free scalar wave equation to manifold valued maps and the classical harmonic maps equation to Lorentzian domains. The case considered here is of particular interest, as the static solutions given by finite energy harmonic maps are amongst the simplest examples of topological solitons;  other examples include kinks in scalar field equations,  vortices in Ginzburg-Landau equations, Dirac monopoles, Skyrmions, and Yang-Mills instantons; see~\cite{MS}. 
The symmetry reduced equation~\eqref{eq:wmk} is a much studied problem, since it admits intriguing features from the point of view of dynamics, e.g., bubbling harmonic maps, multi-soliton solutions, etc.,  in the relatively simple setting of a geometrically natural semilinear wave equation. 
%
For a more thorough presentation of the physical or geometric content of this equation, see e.g.,~\cite{MS, ShSt00, GeGr17}.

Equation \eqref{eq:wmk} can be expressed as a Hamiltonian system, 
\begin{align} \label{eq:wmk-ham}
\partial_t \bs u(t) = J\circ\vD \cl E(\bs u(t)),
\end{align}
where $\bs u(t)$ is our notation for the vector $\bs u(t) := (u(t), \p_t u(t))$. The energy functional $\cl E$ is defined for $\bs u_0 = (u_0, \dot u_0)$ by the formula
\begin{align}
  \label{eq:energy-wmap}
  \cl E(\bs u_0) := \pi \int_0^{\infty}\Bigl(|\dot u_0|^2 + |\partial_r u_0|^2 + \frac{\sin^2 u_0}{r^2}\Bigr)\udr.
\end{align}
Note that initial data $\bs u_0 = (u_0, \dot u_0)$ of finite energy forces $u_0(r)  \to m \pi$ as $r \to 0$ and $u_0(r) \to n \pi$ as $r \to \infty$ for $m, n \in \Z$. We will fix $m=0$, and $n=1$ in our analysis, but we could just as well consider states of finite energy with arbitrary endpoint in $  \pi \Z$.  Thus the finite energy maps we study connect the north and south poles of $\Sp^2$ and have topological degree one, i.e., they are members of the space
\EQ{
\HH_1 := \{ \bs u_0 \mid \E( \bs u_0) < \infty, \lim_{r \to 0} u_0(r) = 0, \quad \lim_{r \to \infty} u_0(r)  = \pi\} 
}
The family of stationary solutions 
\begin{align}
Q_\lambda(r) := 2\arctan\big(\frac{r}{\lambda}\big), \quad \lam >0
\end{align} 
plays a fundamental role in the study of \eqref{eq:wmk}.
We will write $Q(r) := Q_1(r)$. We note $\bs Q:= (Q, 0)\in \HH_1$ since $Q(r) \rar \pi$ as $r \rar \infty$. In fact, $Q$, which is the polar angle of a degree one \emph{harmonic map},  has minimal energy in $\HH_1$; see~\eqref{eq:bog} below. 

We will often work with vectors in the energy space $\mathcal H := H \times L^2$, where the space $H$ is the completion of $C_0^\infty((0, \infty))$ for the norm
\begin{equation}
  \label{eq:H-norm}
  \|v\|_H^2 := \pi \int_0^{\infty}\Big(|\partial_r v(r)|^2 + \frac{|v(r)|^2}{r^2} \Big)\udr.
\end{equation}
Note that $v \in H$ forces $\lim_{r\to +\infty}v (r) = 0$ and that $\bs u_0 \in \HH_1$ means that $u_0 - Q_{\lam} \in H$ for any fixed $\lam>0$.  It is classical that~\eqref{eq:wmk} is locally well-posed for initial data in $\HH$; see~\cite{ST94}. 

\subsection{Main results}
The breakthrough works of Krieger, Schlag, Tataru~\cite{KST}, Rodnianski, Sterbenz~ \cite{RS}, and Rapha{\"e}l, Rodnianski~\cite{RR} proved that wave maps can develop singularities in finite time by concentrating energy at a point in space. 
The main goal of this paper is to directly tie the precise blow-up dynamics of concentrating wave maps to the part of the solution that radiates away from the singularity. 

We start by clarifying what is meant by the \emph{radiation field} of a singular wave map. By equivariance and energy-criticality a solution to~\eqref{eq:wmk} can only become singular by concentrating energy at $r=0$. In~\cite{Struwe}, Struwe proved that such energy concentration can only occur via the bubbling of at least one harmonic map, at least along a sequence of times.  It was later shown in~\cite[Theorem 1.3]{CKLS1} that any $1$-equivariant wave map $u(t) \in \HH_1$ with energy $\E(\bs u) < 3 \E(\bs Q)$ that blows up, say as time $t \rar 0^+$,  admits a decomposition of the form 
\EQ{ \label{eq:ckls1} 
&\bs u (t) = \bs Q_{\lam(t)} + \bs u^*_0 + \bs g(t), \\
&\lam(t) \to 0 \mas t \to 0^+, \\
&\| \bs g(t) \|_{\HH} \to 0 \mas t \to 0^+,
}
where $\bs u^*_0 \in \HH$ is uniquely determined by $\bs u(t)$, and $\lam(t)$ is a continuous function. In fact, 
letting $\bs u^*(t)$ denote the wave map evolution of the initial data $\bs u^*(0) = \bs u^*_0$, we have  $\E( \bs u^*) = \E( \bs u) - \E(\bs Q)$ and, 
\EQ{
\bs u(t, r)  &=  (\pi,0) + \bs u^*(t, r),  \quad \forall  r \ge t, 
}
i.e., $\bs u^*(t)$ accounts for the part of $\bs u(t)$ that radiates out of the light cone as we approach the singular time. We will refer to $\bs u^*_0\in \HH$ or the associated flow $\bs u^*$ as the \emph{radiation field} of the singular wave map $\bs u$.

In the case of a wave map that blows up by bubbling off a \emph{single} harmonic map,  the cap on the energy $\E( \bs u )< 3 \E(\bs Q)$  (which ensures  that there can only be one blow up bubble) was removed by C{\^o}te~\cite{Cote}, and Jia, Kenig~\cite{JK} generalized the result to $k$-equivariant maps.  These works also provided a further generalization of~\eqref{eq:ckls1} that allows for the possibility of more concentrating bubbles along a sequence of times. Later, Duyckaerts, Jia, Kenig, and Merle~\cite{DJKM2} established a one-bubble decomposition as in~\eqref{eq:ckls1} for general wave maps with energy slightly above the ground state harmonic map.

Given the qualitative decomposition~\eqref{eq:ckls1}, it is natural to ask what information does the radiation field $\bs u^*_0$ carry about the blow up dynamics. In this paper, our approach to answering  this question is to fix a mapping $\bs u_0^* \in \HH$ as a candidate for a radiation field, and ask: 
\begin{itemize} 
\item Do solutions that become singular as in~\eqref{eq:ckls1} with radiation $\bs u^*_0$ exist? 
\item If yes, can we characterize \emph{all singular solutions} with radiation $\bs u^*_0$? 
\end{itemize} 
This perspective is natural. Indeed, the only solution as in~\eqref{eq:ckls1} with $\bs u^*_0= (0, 0)$ is the stationary solution $Q_{\lam_0}$ for fixed $\lam_0>0$. Thus, the dynamical parameter $\lam(t)$ in~\eqref{eq:ckls1} can only tend to zero in the presence of nontrivial $\bs u^*_0$ and naively one can expect that nonlinear interactions between the wave map evolution $\bs u^*(t)$ and $\bs Q_{\lam(t)}$ to drive the concentration. In this paper, we show that for sufficiently regular $\bs u^*_0 \in \HH$ this naive intuition is correct, and that the answer to both questions above is yes. In fact,  the dynamical parameter $\lam(t)$ is precisely determined by the rate of decay of $\bs u^*_0(r)$ as $r \to 0$. 

We demonstrate this for a natural class of radiation fields,  i.e. those with polynomial vanishing as $r \to 0$. Let $\nu > \frac{9}{2}$ and $q \in \R \backslash \{0\}$, and consider a radiation field $\bs u_0^* = (u_0^*, u_1^*) \in \HH$ such that either 
\EQ{ \label{eq:u*} 
u^*_0(r) &=  q r^{\nu} + o(r^{\nu}) \mas r \to 0, \\
 \quad u_1^*(r) &= 0, \\
}
or 
\EQ{ \label{eq:u*dot} 
u^*_0(r) &= 0,  \\
 \quad u_1^*(r) &= q r^{\nu-1} + o(r^{\nu-1}) \mas r \to 0.
}

The main result is the following theorem.

\begin{theorem}[Main Theorem]\label{t:main1}
Let $\bs u^*_0 \in \HH$ be as in~\eqref{eq:u*} or~\eqref{eq:u*dot} with $\nu > \frac{9}{2}$. 
Then, 
\begin{itemize} 
\item[\textbf{(a)}] {\textbf{Construction:}} If $q<0$ in~\eqref{eq:u*}, there exists $T_+>0$ and a solution $\bs u_c \in C((0,T_+)); \cl H_1)$ to \eqref{eq:wmk}  blowing up backwards-in-time at $T_- =0$ such that 
\begin{align}
\bs u_c(t) = \bs Q_{\lam_c(t)} + \bs u^*_0 + \bs o_{\cl H}(1), \quad \mbox{as } t \rar 0^+, 
\end{align} 
with
\begin{align} \label{eq:rate1} 
\lam_c(t) = \frac{p|q|}{\nu^2(\nu+1)} \frac{t^{\nu+1}}{|\log t|}, 
\end{align}
where 
\EQ{ \label{eq:pdef} 
p = p(\nu) := \frac{\nu (\nu + 2)\sqrt{\pi}\Gamma \Bigl ( \frac{3+\nu}{2}\Bigr )}{4 \Gamma \Bigl ( \frac{4+\nu}{2} \Bigr )}.
}
If $q <0$ in~\eqref{eq:u*dot} then the exact same result holds with the explicit constant 
\begin{align}\label{eq:tipdef} 
p(\nu) = \frac{(\nu + 1)\sqrt{\pi}\Gamma \Bigl ( \frac{2+\nu}{2}\Bigr )}{4 \Gamma \Bigl ( \frac{3+\nu}{2} \Bigr )}.
\end{align}
\item[\textbf{(b)}] \textbf{Classification:}   \label{t:main2} Let $\bs u(t) \in \HH_1$ be \emph{any}  finite energy solution to~\eqref{eq:wmk} that blows up by concentrating \emph{one} bubble backwards-in-time at $t = 0$ while radiating $\bs u^*_0$ as in~\eqref{eq:u*}, i.e., $\bs u(t)$ admits a 
decomposition
\begin{align}\label{eq:decomp1} 
\bs u(t) = \bs Q_{\lam(t)} + \bs u^*_0 + \bs o_{\cl H}(1), \quad \mbox{as } t \rar 0^+,
\end{align} 
with $\lam(t) = o(t)$ as $t \to 0^+$. Then, 
\EQ{ \label{eq:q} 
q <0, 
}
and the rate $\lam(t)$ satisfies, 
\EQ{
 \lam(t) = \left (\frac{p \abs{q}}{\nu^2(\nu+1)}  +o(1)\right )  \frac{t^{\nu+1}}{|\log t|} \mas t \to 0^+, \label{eq:gen_blowup_rate}
}
where $p(\nu)$ is as in~\eqref{eq:pdef}.  

If instead the radiation takes the form~\eqref{eq:u*dot} then the same result holds with $p(\nu)$ as in~\eqref{eq:tipdef}. 
\end{itemize} 
\end{theorem}

\begin{remark} 
In part $(b)$ of Theorem~\ref{t:main1} we assume that the blow up solution $\bs u(t)$ admits a decomposition~\eqref{eq:decomp1} with one bubble, but this is a mild hypothesis given the existing literature.  By \cite[Theorem 1.3]{CKLS1} the decomposition~\eqref{eq:decomp1} holds for any solution $\bs u(t) \in \HH_1$ such that $\E( \bs u) < 3 \E(\bs Q)$. For larger energies, it was proved by C{\^o}te~\cite{Cote} that such a decomposition holds as long as we assume that there is only one concentrating bubble. The qualitative classification results for energies larger than $3 \E(\bs Q)$ in ~\cite{Cote, JK} allow for solutions that simultaneously concentrate the energy of multiple bubbles in finite time (along a sequence of times), which necessitates our inclusion of the hypothesis that only one bubble concentrates. However,  it is quite possible that finite time multiple bubble trees do not exist.  
\end{remark} 

\begin{remark} 
In Appendix~\ref{a:formal}  we show that the blow up rate~\eqref{eq:rate1} can be formally derived from the leading order behavior of the linear evolution of the radiation. This is borne out by our analysis -- the explicit constants $p(\nu)$ in~\eqref{eq:pdef},~\eqref{eq:tipdef} arise as the leading order coefficients of the evolution of the data $\bs u_0^*(r)/r$ via the free $4d$ wave equation at the point $(t, 0)$; see Section~\ref{s:radiation} and Appendix~\ref{a:u*dot}. 
\end{remark} 



\begin{remark}[Forward-in-time blow up and the sign of the bubble]  \label{r:q} 
Suppose the radiation takes the form $\bs u^*(0)  = (u_0^*, 0)$ as  in~\eqref{eq:u*} with $q < 0$. Then due to the time reversal symmetry of \eqref{eq:wmk},
a wave map $\bs u^+$ satisfies \eqref{eq:decomp1} if and only if $u^-(t,r) :=  u^+(-t,r)$ satisfies  
\begin{align}
\bs u^-(t) = \bs Q_{\lam(|t|)} + \bs u_0^* + \bs o_{\cl H}(1), \quad \mbox{as } t \rar 0^-.  
\end{align}
Thus, the conclusions of Theorem \ref{t:main1} are equally valid forward-in-time.
However, if we assume the radiation takes the form $\bs u^*(0)  = (0, u_1^*)$ as in~\eqref{eq:u*dot} with $q < 0$, then
a wave map $\bs u^+$ satisfies \eqref{eq:decomp1} if and only if $u^-(t,r) :=  -u^+(-t,r)$ satisfies  
\begin{align}
\bs u^-(t) = -\bs Q_{\lam(|t|)} + \bs u_0^* + \bs o_{\cl H}(1), \quad \mbox{as } t \rar 0^-.  
\end{align} 
In particular, the conclusions of Theorem \ref{t:main1} hold forward-in-time with the \emph{sign of the bubble reversed}. 

In all of the analysis we could have just as easily considered 
$\bs u_0^* = (u_0^*, u_1^*)$ with  $u_0^*,  u^*_1$ both nontrivial and simply determined which factor contributed the leading  order behavior of the linear flow; see Section~\ref{s:radiation} and Appendix~\ref{a:u*dot}. For example, if the radiation takes the form 
\begin{align}
\bs u_0^* = (q_1 r^\nu + o(r^\nu), q_2 r^{\nu-1} + o(r^{\nu-1})) 
\end{align}
with 
\begin{gather}
q_1 \frac{\nu (\nu + 2)\sqrt{\pi}\Gamma \Bigl ( \frac{3+\nu}{2}\Bigr )}{4 \Gamma \Bigl ( \frac{4+\nu}{2} \Bigr )} + 
q_2 \frac{(\nu + 1)\sqrt{\pi}\Gamma \Bigl ( \frac{2+\nu}{2}\Bigr )}{4 \Gamma \Bigl ( \frac{3+\nu}{2} \Bigr )} < 0, \label{eq:sign_1}\\
q_1 \frac{\nu (\nu + 2)\sqrt{\pi}\Gamma \Bigl ( \frac{3+\nu}{2}\Bigr )}{4 \Gamma \Bigl ( \frac{4+\nu}{2} \Bigr )} - 
q_2 \frac{(\nu + 1)\sqrt{\pi}\Gamma \Bigl ( \frac{2+\nu}{2}\Bigr )}{4 \Gamma \Bigl ( \frac{3+\nu}{2} \Bigr )} < 0, \label{eq:sign_2}
\end{gather}
then there exist bubbling solutions $\bs u_c^\pm$ and scaling parameters $\lam^{\pm}$ defined for $\pm t > 0$ near 0 satisfying $\lam^{\pm}(t) \simeq |t|^{\nu+1} |\log |t||^{-1}$ such that 
\begin{align}
\bs u_c^{\pm}(t) = \bs Q_{\lam^{\pm}(t)} + \bs u_0^* + \bs o_{\cl H}(1), 
\quad \mbox{as } t \rar 0^{\pm}. 
\end{align}
Moreover, if the sign of \eqref{eq:sign_2} is reversed then there exist bubbling solutions $\bs u_c^\pm$ and scaling parameters $\lam^{\pm}$ defined for $\pm t > 0$ near 0 satisfying $\lam^{\pm}(t) \simeq |t|^{\nu+1} |\log |t||^{-1}$ such that 
\begin{align}
\bs u_c^{\pm}(t) = \pm \bs Q_{\lam^{\pm}(t)} + \bs u_0^* + \bs o_{\cl H}(1), 
\quad \mbox{as } t \rar 0^{\pm}. 
\end{align}
In particular, the sign of the bubble is reversed for a solution bubbling forward in time in this situation. In both cases, classification results hold with different explicit constants appearing in the concentration rates $\lam^{\pm}(t)$. 
However, 
we will stick to the forms~\eqref{eq:u*} and~\eqref{eq:u*dot} for simplicity. 
\end{remark} 

\begin{remark}[Straightforward generalizations] \label{r:variations} 
The restriction $u^*_0(r) \to 0$ as $ r \to \infty$ in~\eqref{eq:u*} does not any significant play a role in the analysis and we could just as easily consider $u_0^*(r)$ such that $u_0^*(r) \to n \pi$ as $r \to \infty$ for any integer $n \in \Z$.   
\end{remark}

\begin{remark}[Different blow up dynamics]  \label{r:other} 
The statement of Theorem~\ref{t:main1} makes clear that the order of vanishing of $\bs u^*_0$ in~\eqref{eq:u*} or ~\eqref{eq:u*dot}  as $r \to 0$ determines the blow up rate~\eqref{eq:rate1}. We considered the polynomial decay, e.g.,  $r^ \nu$ as in~\eqref{eq:u*}, first and foremost because of its simplicity, but also because it yielded the previously unknown blow up dynamics $\lam(t) \simeq t^{\nu+1}/ \abs{\log t}$.  

The restriction $\nu>9/2$ in Theorem~\ref{t:main1} is technical -- it allow us to use the simple pointwise estimates in Lemma~\ref{l:linear_app} to pick out the leading order of the evolution $\bs u^*(t)$ near $r =0$. The odd integer values of $1<\nu< 9/2$, i.e., $\nu= 3$, could easily be included in our analysis, since function in~\eqref{eq:phidef} is manifestly smooth in this case.  The case $\nu=1$ in~\eqref{eq:u*} is a special case that could also be treated by our methods and there we would see the same blow up dynamics from~\cite{R19}.  The number $9/2$ could possibly be lowered with a more careful analysis, but we do not pursue this here.
 
The techniques in this paper should also allow Theorem~\ref{t:main1} to be extended to radiation fields with non-pure polynomial decay. In Appendix~\ref{a:formal} we outline how to formally guess the rate of $\lam(t)$ for a few different choices for $\bs u^*(r)$. For example, the methods here should readily cover radiation fields of the form, 
\EQ{
u_0^*(r)  = - r^{\nu} \abs{\log r}^\mu + o( r^{\nu} \abs{\log r}^\mu)  \mas r \to 0, \quad \nu>\frac{9}{2}, \mu \in \R
} 
which would yield the blow-up rates, 
\EQ{
\lam(t) \simeq  t^{\nu+1} \abs{\log t}^{\mu-1} \mas t \to 0.
}
Of particular interest is the case $\mu =1$, which yields the pure power rates $\lam(t) \simeq t^{\nu+1}$ discovered by Krieger, Schlag, Tataru~\cite{KST} using a different method. 

The stable blow up regime discovered by Rapha{\"e}l, Rodnianski~\cite{RR} is an open set of initial data (including smooth data) in $H^2$ that leads to blows up with the rate $\lam(t) \simeq t e^{-\sqrt{ \abs{\log t}}}$, i.e., just barely faster than the self-similar rate $\lam(t) = t$, which is forbidden for finite energy solutions to~\eqref{eq:wmk}.  The radiation associated to these blow up solutions is no better than energy class, and thus it is not clear if specific radiation profiles for these solutions can be characterized \emph{a priori} as is the case for the solutions in Theorem~\ref{t:main1}. 
\end{remark} 

\begin{remark} 
We note that for integer values of $\nu>9/2$, the radiation $\bs u^*(t)$ can be a $C^\infty$ function. We do not pursue here questions concerning additional regularity of the blow up solution $\bs u_c(t)$ in part $(a)$. 
\end{remark} 

\subsection{Discussion of Theorem~\ref{t:main1}} 

A basic tenet of the approach in this paper is that certain bubbling phenomena can be readily accessed by viewing the solution backwards from the final time, rather than from the point of view of initial data. This ``backwards''  point of view opens the door to classification statements as in part $(b)$ of Theorem~\ref{t:main1} and to the uniqueness conjecture posed below.  Our method can be motivated via an analogy with the scattering problem for nonlinear waves. 

\subsubsection{Analogy with the scattering problem} There are  two ways to think about the scattering problem. On the one hand, start with initial data for a nonlinear wave $u_{{NL}}(t)$ and \emph{show that it scatters} by finding a linear wave $u_{{L}}(t)$ such that 
\EQ{ \label{eq:scattering} 
\| u_{{NL}}(t) - u_{{L}}(t) \|_\HH  \to 0 \mas t \to \infty
}
This first type of scattering question is typically called \emph{completeness of wave operators}. On the other hand, one may start with a free evolution $u_{{L}}(t)$ and attempt to find a nonlinear wave $u_{{NL}}(t)$ such that~\eqref{eq:scattering} holds. This latter type of scattering question is called \emph{existence of wave operators} and is typically easier to address than the completeness question. As one may imagine, the completeness question requires precise estimates for the nonlinear flow, which are harder to obtain than the corresponding information about the linear flow that serves as input for the existence question. 

In the setting of singularity formation, one can view the approach of ``finding sets of initial data that lead to blow-up'' taken for example in~\cite{KST, RS, RR} as somewhat analogous to the \emph{completeness} question in scattering theory.  The approach in this paper, in which we prescribe the radiation field $\bs u^*_0$ and look for blow-up solutions radiating $\bs u^*_0$ can be  viewed as analogous to the \emph{existence} question in scattering theory. One of the main advantages of the ``backwards'' perspective taken here is that it gives evidence for a natural \emph{unique} continuation of singular solutions past  the blow up time.

\subsubsection{The Radiative Uniqueness Conjecture and unique continuation} 
Part $(b)$ of Theorem~\ref{t:main1} determines the rate of concentration for any blow up solution radiating $\bs u^*_0$ up to the precise constant in the leading order term in $\lam(t)$. The following corollary is immediate. 
 \begin{corollary} 
Let $\bs u_c(t)$ denote the constructed blow-up solution from Part (a) of Theorem~\ref{t:main1} and let $\bs u(t) \in \HH_1$ be \emph{any} other solution that blows up backwards-in-time at $t =0$ with the same radiation field $\bs u^*_0$. Then,
\EQ{
\| \bs u_c(t) - \bs u(t) \|_{\HH} \to 0 \mas t \to 0^+. 
}
\end{corollary} 
In fact, we expect the radiation to~\emph{uniquely} determine the blow-up solution. 
\begin{conjecture}[Radiative Uniqueness Conjecture] \label{c:uc} 
Let $\bs u_c(t)$ denote the constructed blow-up solution from Part (a) of Theorem~\ref{t:main1} and let $\bs u(t) \in \HH_1$ be \emph{any} other solution that blows up backwards-in-time at $t =0$ with the same radiation field $\bs u^*_0$. Then,
\EQ{
\bs u(t) = \bs u_c(t) 
}
\end{conjecture} 

The significance of this conjecture is that its proof would yield the \emph{unique continuation} of  blow-up wave maps past the singularity preserving the energy and the topological class of the solution. To be more concrete, suppose $\bs u_0^*$ is as in $\eqref{eq:u*}$ with $q < 0$. Then by Theorem \ref{t:main1}, there exists a finite energy solution $\bs u^+$ defined for small positive time such that 
\begin{align}
\bs u^+(t) = \bs Q_{\lam^+(t)} - (\pi,0) + \bs u^*_0 + \bs o_{\cl H}(1),
\quad \mbox{as } t \rar 0^+. 
\end{align} 
The shift by $-\pi$ above is included for convenience so that we now have $\bs u^+(t,r) = \bs u^*(t,r)$ for all $r \geq t$. The wave map $\bs u^+$ can be continued past $t = 0$ by attaching the forward-in-time bubbling solution $\bs u^-$ that also radiates $\bs u^*_0$: 
\begin{align}
\bs u^-(t) = \bs Q_{\lam^-(t)} - (\pi,0) + \bs u^*_0 + \bs o_{\cl H}(1), \quad 
\mbox{as } t \rar 0^-. 
\end{align} 
As discussed in Remark \ref{r:q}, $u^-(t,r) := u^+(-t,r)$, $t < 0$, in this case. The resulting weak solution is now defined on an open interval containing $t = 0$ and is continuous except at $(t,r) = (0,0)$.  Viewed in forward time, the solutions consists of a concentrating bubble radiating $\bs u^*(t)$ for $t<0$ and an expanding bubble for $t>0$. 

 In the case of radiation of the form $\bs u^*_0 = (0, u_1^*)$ as in~\eqref{eq:u*dot} with $q < 0$, the continuation past $t = 0$ is obtained by attaching the forward-in-time anti-bubbling solution $\bs u^-$ that also radiates $\bs u^*_0$: 
\begin{align}
\bs u^-(t) = -\bigl (\bs Q_{\lam^-(t)} - (\pi,0) \bigr) + \bs u^*_0 + \bs o_{\cl H}(1), \quad 
\mbox{as } t \rar 0^-. 
\end{align} 
Again, as discussed in Remark \ref{r:q}, $u^-(t,r) := -u^-(-t,r)$, $t < 0$, in this case. Viewed in forward time, the solutions consists of a concentrating anti-bubble radiating $\bs u^*(t)$ for $t<0$ and an expanding bubble for $t>0$.
 A proof of Conjecture~\ref{c:uc} ensures that the continuation process discussed in each case previously is the only meaningful one -- of course one could also ``extend'' the blow up solution past the singular time by evolving only the radiation $\bs u_0^*$, but this extension does not conserve energy ($\E(\bs Q)$ is lost) and does not preserve the topological class of the solution. 
 
These weak solutions provide an intriguing connection 
to the work of Topping~\cite{Top-IMRN}, which gave a continuation for the harmonic map heat flow in which an expanding bubble is reattached after a rotation at the blow-up time. Indeed, the continuation proposed here for solutions with radiation of the form~\eqref{eq:u*dot} comes with a $180$-degree rotation of the bubble when viewed in the context of the full map~\eqref{eq:fullmap}. This raises interesting questions about the nature of the flow outside of equivariant symmetry, e.g., do these continuations of bubbling solutions yield information about ``nearby'' continuous-in-time solutions?  It seems that not much is known in this direction in the case of wave maps, but one can draw parallels to formal observations 
made by van den Berg, Williams~\cite{vdBW} for the Gilbert-Landau-Lifshitz flow that describe ``near'' blow up solutions in which the inner scale quickly rotates over an angle of $180$-degrees.  In the context of Schr\"odinger maps, Merle, Rapha{\"e}l, Rodnianski~\cite{MeRaRo13} observed an \emph{instability mechanism} given by the rotational freedom in that model, and conjectured that the blow up regime for wave maps in~\cite{RR} is similarly unstable under non-corotational perturbations; see~\cite[Comment 2 after Theorem 1.1]{MeRaRo13}.

The authors are unaware of any uniqueness results as in Conjecture~\ref{c:uc} for concentrating nonlinear waves, but we mention here the pioneering proof of existence and uniqueness of minimal mass blow up for the mass critical NLS  by Merle~\cite{Merle93} and the analogous result (using a different method) by Rapha{\"e}l, Szeftel ~\cite{RS} for the mass critical NLS with an inhomogeneous nonlinearity.  A significant distinction is that in these works the minimal mass blow up emits no radiation.  The question of extending singular solutions past blow up was also addressed in~\cite{BDS}, but in the different context of supercritical self-similar wave maps, and with an extension that comes with a change of topological charge.  



\subsubsection{Relationship with previous work:} In the context of energy critical singular nonlinear waves, the ``backwards'' approach taken here was initiated in a series of papers by the first author on the blow-up problem for the focusing energy critical nonlinear wave equation (NLW). In~\cite{moi15}, the first author provided an upper bound on the blow up speed $\lam(t)$ for the NLW in dimensions $d=3, 4, 5$ in the case where the radiation lies in $H^{s+1} \times H^{s}$ for $s > \frac{d-2}{2}$ and $s \ge1$, and showed that blow up is impossible in the case of regular $\bs u^*_0$ with  $u^*_0(0) < 0$. In~\cite{moi17-jfa}, the first author gave constructions of blow up solutions for the $5d$ NLW for regular radiation analogous to the result in part $(a)$ of Theorem~\ref{t:main1}, via a related approach. In this context, the significant new element of the present work is the \emph{sharp classification of the rate} in part $(b)$ of Theorem~\ref{t:main1}. To the authors' knowledge, this is the first paper to obtain such a result.  In earlier work ~\cite{JL1, R19}, the authors considered the special case of threshold dynamics. As part of the main theorem in~\cite{JL1}, the first two authors determined the rate of concentration for  global-in-time pure $2$-bubble solutions in equivariance classes $k \ge2$. In part of~\cite{R19}, the third author characterized the rate for any minimal topologically trivial $1$-equivariant blow-up solution. In fact, the latter paper~\cite{R19} can be compared to the special case of $\nu=1$ in~\eqref{eq:u*}, but where the radiation is given precisely by $u^*_0(r) = - Q(r)$. See also work of the first author~\cite{JJ-KdV}, which characterized the dynamical behavior for pure $2$-solitons with the same limit speed for the KdV equation in the unstable regime. 

In a broader sense, the heuristic principle that the size of the nonlinear interaction between a blow up bubble and the rest of the solution influences, or even determines the blow up speed is well documented in the literature. For the mass critical NLS, Bourgain, Wang~\cite{BoWa97} produced examples of blow up solutions with regular $u^*$ where the blow up speed is given by that of the explicit solution $S(t)$, which equals the psuedo-conformal transform of the ground state solitary wave (which has $u^* = 0$). This was later clarified in the classification result of Merle, Rapha{\"e}l~\cite{MeRa05}, which showed how the regularity properties of $u^*$ distinguish the so-called $\log\log$-regime from the $S(t)$-type blow up regime. Later, a remarkable sequence of works by Martel, Merle, Rapha{\"e}l,~\cite{MMR14-1, MMR15-2, MMR15-3} on the mass critical gKdV equation showed that blow up solutions  close to the soliton have a fixed rate if the initial  data have sufficient  decay properties, and that exotic blow up regimes exist if these decay properties are relaxed.  In the case of energy critical equations, Merle, Rapha{\"e}l, Rodnianski conjecture in~\cite[Comment 4 after Theorem 1.1]{MeRaRo13} a deep relationship between blow up dynamics and the regularity of $u^*$ in the context of the Schr\"odinger maps equation.  

As mentioned in Remark~\ref{r:other} the pure polynomial blow up regime $\lam(t) \simeq t^{1+\nu}, \nu> \frac{1}{2}$ discovered by  Krieger, Schlag, Tataru in~\cite{KST} can be recovered via the methods in this paper, as long as $\nu$ is sufficiently large. Since the landmark work~\cite{KST} there have been several subsequent developments, see e.g.,~\cite{GK} which provided the optimal range of pure polynomial blow up, $\nu>0$.  Recently, Krieger, Miao~\cite{KMiao} proved that these blow up solutions are stable under a sufficiently regular perturbation. The methods from~\cite{KST} do not track the solution outside the light cone, and thus it is unclear from the construction what form the asymptotic radiation takes. From Remark~\ref{r:other} we see that the radiation corresponding to pure polynomial blow up rates vanishes like $r^\nu \abs{\log r}$ and thus the evolution of this radiation is singular on the light cone even for integer values of $\nu>0$, which is consistent with~\cite{KST}.  

\subsection{Outline of the proof}
\label{ssec:outline}

The proof of Theorem~\ref{t:main1} can be summarized as follows. Let $\bs u^*(t)$ be the nonlinear evolution of the asymptotic radiation $\bs u^*_0$ and denote by $\bs u^*_L(t)$ the linear evolution. At time $t>0$ we study the flow near the set $\{ \bs Q_{\lam} + \bs u^*(t)\}$ for $\lam\ll 1$.  Considering a solution $\bs u(t)$ near this set on a time interval  we write  
\EQ{ \label{eq:decomp2} 
\bs u(t) = \bs Q_{\lam(t)} + \bs u^*(t) + \bs g(t)
}
where the pair $(\lam(t), \bs g(t)) := (\lam(t), g(t), \dot g(t))$ are determined uniquely by imposing the orthogonality conditions $0 = \ang{  \calZ_{\lam(t)} \mid g(t)}$, where we take $\calZ(r) = \chi(r) r \p_r Q(r)$ for a smooth cut-off $\chi$. This yields a coupled system for $(\lam(t), \bs g(t))$ under the assumption that $\bs u(t)$ above solves~\eqref{eq:wmk}, i.e., 
\begin{align} 
 \frac{\ud}{\ud t} \pmat{ g  \\ \dot g}  &= \pmat{  \dot  g  + \frac{\lam'}{\lam} (r \p_r Q)_{\lam}  \\ \De g - \frac{1}{r^2}( f( Q_\lam + u^* +g) - f(Q_\lam) - f(u^*)) } \\
 \ang{  \calZ_{\lam(t)} \mid g(t)} &= 0, \label{eq:orthog} 
\end{align} 
where $f(\rho) = \frac{1}{2} \sin 2 \rho$. 
This is the beginning of the classical modulation theoretic approach and a standard argument yields a preliminary estimate $\abs{\lam'(t)} \lesssim \| \dot g(t) \|_{L^2}$, by differentiating~\eqref{eq:orthog} and using the first row in the equation for $\bs g(t)$. The goal is to (a) find a solution such that 
\EQ{ \label{eq:lag} 
\lam(t) \to 0, \quad \| \bs g(t) \|_{\HH} \to 0 
} 
in finite time, and (b) characterize the dynamics of $\lam(t)$ for any such solution. Since~\eqref{eq:wmk} is second order in time, refined information linking the dynamics of $\bs g(t)$ and $\lam(t)$ enters in the study of $\lam''(t)$.  However, a naive approach of twice differentiating~\eqref{eq:orthog} typically does not yield sufficient information to close. 
If the goal is to construct a single solution, a now standard remedy (developed with outstanding success in, e.g.,~\cite{RR, MR1, MR2, MR3, MR4, MR5, MMR1, MMR2, MMR3, MMR-sem, RaSz11}), is to  refine the ansatz, i.e., extract more profiles from $\bs g(t)$ before imposing the orthogonality conditions, improving at each step the equations of motion for the remainder and the dynamical parameters. However, fixing a refined ansatz at the outset is somewhat at odds with proving a general classification result as in part $(b)$ of Theorem~\ref{t:main1}. 

We thus employ a general approach developed by the authors in~\cite{JL1, R19, JJ-KdV}, which was inspired by earlier work of the first author in~\cite{JJ-AJM, JJ-APDE}. We do not refine the ansatz at all, rather  we  \emph{refine the modulation parameters}. For technical reasons we first define a new function $\zeta(t) \sim 4\lam(t) \log(t/ \lam(t))$. We then define a new modulation parameter, $b(t)$, by 
\EQ{
b(t) = \zeta'(t) + \textrm{small correction} 
}
where the correction has a large derivative designed to cancel critical terms of indeterminate sign in the equation for $\z''$. In the study of $b'(t)$ we find that the leading order is given by the nonlinear interaction between the bubble $\bs Q_{\lam(t)}$ and the linear part of the radiation $\bs u_L^*(t)$, and this gives rise to the universal rate $\lam(t) \simeq t^{\nu+1}/ \abs{\log t}$; see Section~\ref{s:mod}.  
There are several difficulties that arise. For one, to make sense of the above, we need control over the size of the error $ \|\bs g(t) \|_{\HH}$ in terms of $t, \lam(t)$ --  even the correction to $\zeta'$ is seen to be small only once we know the size of $\| \bs g(t) \|_{\HH}$. 
 
 We argue as follows. For a solution as in~\eqref{eq:decomp2} such that~\eqref{eq:lag} holds, we have 
\EQ{
\E( \bs u) = \E( \bs Q) + \E( \bs u^*) 
}
On the other hand a Taylor expansion of the energy using~\eqref{eq:decomp2} yields, 
\EQ{
\E( \bs u) = \E( \bs Q_\lam + \bs u^*) + \ang{ D \E( \bs Q_\lam + \bs u^*) \mid \bs g} + \ang{ D^2 \E( \bs Q_\lam + \bs u^*) \bs g \mid \bs g} + O( \| \bs g \|_{\HH}^3) 
}
The orthogonality condition~\eqref{eq:orthog} ensure that the quadratic term in coercive, hence combining the previous two identities yields, 
\EQ{
\| \bs g \|_{\HH}^2 \simeq \Big(\E( \bs Q_\lam + \bs u^*) -  \E( \bs Q) - \E( \bs u^*)\Big)  +  \ang{ D \E( \bs Q_\lam + \bs u^*) \mid \bs g}
}
Since $D\E(\bs Q_\lam) = 0$ the last term above can be replaced with $\ang{ D \E(  \bs u^*) \mid \bs g}$. In the settings considered in earlier work such as~\cite{moi15, JL1, R19}, the latter term is also shown to be negligible, which means that the pure interaction terms in the first grouping above yield the size of $\| \bs g \|_{\HH}^2$. Indeed, using technique analogous to those developed by Struwe~\cite{Struwe99} in a different context,~\cite{moi15} shows that $\ang{D \E(  \bs u^*) \mid \bs g}$ is, to leading order, a conservation law. In~\cite{JL1, R19} the argument is even simpler, since there $\bs u^* = - \bs Q_\mu$ and the linear term is manifestly lower order. Here there is a significant difference. In fact,  $\ang{ D \E(  \bs u^*) \mid \bs g}$ is shown to carry leading order interaction, see Lemma~\ref{l:linear_pairing_estimate}. This novel feature can be attributed to the fact that the underlying system for the dynamical parameters is not autonomous in contrast to~\cite{moi15, JL1, R19}; see e.g., Appendix~\ref{s:rnu}. Section~\ref{s:energy} is devoted to making this precise. In fact,  the sketch above is oversimplified as we have ignored the need to localize due to the slow spatial decay of $r \p_r Q(r) \sim r^{-1}$. This introduces large, but ultimately manageable error terms.  

In Sections~\ref{s:construction} and \ref{s:classification} we combine the modulation theoretic analysis from Section~\ref{s:mod} with the energy estimates for $\bs g(t)$ in Section~\ref{s:energy} to construct and classify bubbling solutions that radiate~$\bs u^*_0$. 

We emphasize that while the arguments in this paper are involved, the technique is elementary.  In particular,  we do not use any sophisticated dispersive analysis aside from the local well-posedness theory of~\eqref{eq:wmk}, which is classical.  




\section{Preliminaries}
In this section we recall elementary facts about solutions to \eqref{eq:wmk} needed in our analysis and establish precise behavior of the radiation field $\bs u^*(t,r)$ inside the light cone $\{r \leq t \}$. 

To simplify the analysis we fix a precise form for the radiation $\bs u^*_0(r)$. We define 
\EQ{
\bs u^*_0(r)  = (u_0^*(r), \dot u_0^*(r)) := \chi(r) (qr^{\nu}, 0) \label{eq:u*1} 
}
where $\chi(r)$ is a smooth cutoff such that $\chi(r)= 1$ for $r \le \frac{1}{2}$, and $\supp(\chi(r)) \subset \{ r \le 1 \}$. We remark that the arguments that follow can be readily adapted to the more general form of the radiation in~\eqref{eq:u*}. See Appendix~\ref{a:u*dot} for a discussion on how to to adapt the proof in the case of radiation as in~\eqref{eq:u*dot}. 

\subsection{Notation}
Given a radial function $f: \bR^d \to \bR$ we will abuse notation and simply write $f = f(r)$, where $r = |x|$. We will also drop the factor $2\pi$ in our notation for the $L^2$ pairing of radial functions on $\bR^2$ 
\begin{equation}
	\la f, g \ra  := \frac{1}{2\pi}\la f, g\ra_{L^2(\bR^2)} = \int_0^\infty \conj{f(r)} g(r) \, r \ud r.
\end{equation}
We define a norm $H$ by 
\begin{equation}
	\| v \|_{H}^2 := \int_0^\infty  \left(( \partial_r v(r))^2 +    \frac{(v(r))^2}{r^2} \right)  r \ud r,
\end{equation}
and for pairs $\bs v = (v, \dot v)$ we write  
\begin{equation}
	\| \bs v \|_{\cH} := \| (v, \dot v)\|_{H \times L^2}.
\end{equation}
The change of variables $x = \log r$ gives us an identification between the radial functions $H(a \leq r \leq b)$ and $H^1(\log a \leq x \leq \log b)$, i.e., 
\begin{align}
	v(r) \in H(a \leq r \leq b) \iff v(e^x) \in H^1(\log a \leq x \leq \log b)
	.
\end{align} By the fundamental theorem of calculus this means
\begin{equation}\label{eq:infty_estimate}
	\| v \|_{L^{\infty}(a \leq r \leq b)} \le C \| v \|_{H(a \leq r \leq b)}.
\end{equation}
The scaling invariance of \eqref{eq:wmk} plays a key role in our analysis.
Given a radial function $v: \bR^2  \to \bR$ we denote the $\dot H^1$ and $L^2$  re-scalings as follows 
\begin{equation} \label{eq:scaledef} 
	v_\lambda(r) = v(r/ \lambda), \quad
	v_{\uln{\lambda}}(r)  = \frac{1}{\lambda} v(r/ \lambda).
\end{equation}
The corresponding infinitesimal generators are given by 
\begin{equation}
	\label{eq:LaLa0}
	\begin{aligned}
		\Lambda v &:= -\frac{\partial}{\partial \lambda}\bigg|_{\lambda = 1} v_\lambda = r \partial_r v  \quad (\dot H^1_{\textrm{rad}}(\bR^2) \,  \textrm{scaling}), \\
		\Lambda_0 v &:= -\frac{\partial}{\partial \lambda}\bigg|_{\lambda = 1} v_{\uln{\lambda}} = (1 + r \partial_r ) v  \quad (L^2_{\textrm{rad}}(\bR^2) \,  \textrm{scaling}).
	\end{aligned}
\end{equation}
Finally, we will often use the notation 
\EQ{ \label{eq:fdef} 
	f( \rho) := \frac{1}{2} \sin 2 \rho 
}
to denote the nonlinearity in~\eqref{eq:wmk}, which means ~\eqref{eq:wmk} becomes 
\EQ{
	\p_t^2 u - \p_r^2 u - \p_r u + \frac{1}{r^2} f( u) = 0. 
}

\subsection{The harmonic map $\bs Q$} 

We recall that the unique (up to scaling and sign change) nontrivial, harmonic map is given by 
\begin{align*}
	\bs Q(r) = (2 \arctan r,0).  
\end{align*}
The harmonic map $\bs Q$ has a variational characterization as follows. For all $\bs u = ( u_0,  u_1) \in \HH_1$ where 
\EQ{
	\HH_1:= \{ (u_0, u_1) \mid \E(\vec \phi)< \infty, \quad u_0(0) = 0, \quad \lim_{r \to \infty} u_0(r) = \pi\}, 
}
we have 
\begin{align}
	\cl E( \bs u) \geq \cl E(\bs Q) = 4\pi
\end{align}
with equality if and only if $\bs u = \bs Q_\lam$.  Indeed, 
for $(u_0, u_1) \in \HH_1$, we have the following 
Bogomol'nyi factorization of the nonlinear energy:
\EQ{ \label{eq:bog}
	\E( \bs u)  &=   \pi \| u_1 \|_{L^2}^2 +  \pi \int_0^{\infty} \left(\p_r u_0  - \frac{\sin(u_0)}{r}\right)^2 \, r \ud r +  2\pi \int_0^{\infty} \sin(u_0) \p_ru_0 \, \ud r\\
	&= \pi \| u_1 \|_{L^2}^2  + \pi \int_0^{\infty} \left(\p_r u_0  - \frac{\sin(u_0)}{r}\right)^2 \, r\ud r + 2 \pi  \int_{u_0(0)}^{u_0(\infty)} \sin(\rho)  \, d\rho \\
	&  =  \pi \| u_1 \|_{L^2}^2  +  \pi \int_0^{\infty} \left(\p_r u_0  - \frac{\sin(u_0)}{r}\right)^2 \, r\ud r  + 4\pi. 
} 
We conclude $\cl E(\bs u) \geq 4\pi$ with equality if and only if $\bs u = \bs Q_\lam$ for some $\lam > 0$. 

The Schr\"odinger operator corresponding to linearizing \eqref{eq:wmk} about the harmonic map $Q_\lam$ is given by 
\EQ{
	\LL_\lam:= - \p_r^2 - \frac{1}{r} \p_r + \frac{1}{r^2} f'(Q_\lam).
} 
For convenience we write $\LL := \LL_1$. Due to the variational characterization of $\bs Q$, one expects the spectrum of $\LL$ to contain no unstable modes. Indeed, differentiating the equation
\begin{align*}
	\p_r^2 Q_\lambda + \frac{1}{r} \p_r Q_\lambda - \frac{1}{r^2} f(Q_\lam)= 0
\end{align*}
with respect to $\lambda$ and setting $\lambda = 1$ implies  
\begin{align}
	\Lam Q(r) = \frac{2r}{1 + r^2}
\end{align}
satisfies 
\EQ{
	\LL \Lam Q = 0, \quad \Lam Q \in L^\infty(\R^2).
}
By Sturm oscillation theory, we conclude $\LL$ has no negative eigenvalues. A useful computation we will use throughout our analysis is:
\EQ{ \label{eq:LaQL2} 
	\int_0^R [ \Lambda Q(r)]^2 \, r\, \ud r  =  -\frac{2 R^2}{1 + R^2} +2 \log(1+ R^2)  = 4 \log R + O(1) \mas R \to \infty.
}
In particular, the previous line implies $\Lam Q$ fails (logarithmically) to be in $L^2(\bbR^2)$, and we say 0 is a resonance for $\LL$.  In general, we have $\LL_\lam \Lam Q_\lam = 0$ by scaling. Finally,  
we note $\Lam_0 \Lam Q$ has an important cancellation which leads to improved decay at $\infty$. Indeed, 
\EQ{ \label{eq:Lam0LamQ} 
	\Lam_0 \Lam Q = \frac{4r}{(1+ r^2)^2} 
}
so $\Lam_0 \Lam Q \in L^1(\bbR^2) \cap L^\infty(\bbR^2)$.

\subsection{Description of the radiation field $\bs u^*$ inside the light cone.} \label{s:radiation}
Let $\chi \in C^\infty(\bR^2)$ be radial such that $\chi(r) = 1$ for $r \leq \frac{1}{2}$ and $\chi(r) = 0$ if $r \geq 1$.
Let 
$$\bs u_0^*(r) = (u_0^*(r), \dot u_0^*(r)) = \chi(r) (qr^{\nu}, 0),$$ as in~\eqref{eq:u*1} and let $\bs u^*(t) \in \cl H_0$ be the unique finite energy solution to \eqref{eq:wmk} with $\bs u^*(0) = \bs u^*_0$, i.e. the radiation.
Let $u^*_L = u^*_L(t,r)$ be the solution to the linear wave equation 
\begin{align}
\begin{split}
		&\partial_t^2 u_L - \partial_r^2 u_L -\frac{1}{r}\partial_r u_L + \frac{1}{r^2} u_L = 0, \quad (t,r) \in \bbR \times (0,\infty), \\
		&\bs u^*_L(0) = \bs u^*_0. 
\end{split}	\label{eq:LW}
\end{align}
Our goal of this section is to obtain a description of $\bs u^*(t,r)$ inside the light cone $\{r \leq t\}$.  To do this, we first describe $\bs u^*_L(t,r)$ for $r \leq t$ using Kirchoff's formula. We then compare $\bs u^*(t,r)$ to $\bs u_L^*(t,r)$ for $r \leq t$ using the standard well-posedness theory for \eqref{eq:wmk}.  It is in carrying out these steps that the technical restriction $\nu > \frac{9}{2}$ is required.

\begin{lemma}\label{l:linear_app}
	Let $\nu > \frac{9}{2}$.  There exists $C = C(\nu,q) > 0$ such that for all $r \leq  |t|$,
	\begin{align}
		|u_L^*(t,r) - q p \abs{t}^{\nu - 1} r | &\leq C r^3 \abs{t}^{\nu - 3}, \label{eq:linearleading} \\
		|\p_t u^*_L(t,r)| &\leq C r \abs{t}^{\nu-2}, \label{eq:lineartest}\\
		|\p_r u^*_L(t,r)| &\leq C \abs{t}^{\nu-1}, \label{eq:linearrest}
	\end{align}
	where 
	\begin{align*}
		p(\nu) =  \frac{\nu (\nu + 2)}{2} \int_0^1 \rho^{\nu + 2 } (1 - \rho^2)^{-1/2} \, \ud \rho = \frac{\nu (\nu + 2)\sqrt{\pi}\Gamma \Bigl ( \frac{3+\nu}{2}\Bigr )}{4 \Gamma \Bigl ( \frac{4+\nu}{2} \Bigr )}.
	\end{align*}
\end{lemma}

\begin{proof}
	For $(t,x) \in \bbR^{1+4}$, define $v(t,x) = u^*_L(t,|x|)/|x|$.  Then $\Box_{\mathbb R^{1+4}} v = 0$ and $\bs v(0) = \frac{1}{r} \bs u_0^*$. For $|x| \leq |t|$, we can express $v(t,x)$ using Kirchoff's formula and a change of variables by 
	\EQ{ \label{eq:kirchoff}
		v(t,x) &= \frac{1}{8} \p_t \Bigl ( \frac{1}{t} \p_t \Bigr ) 
		\left (
		t^3 \abs{t}^{\nu-1} \fint_{|y| \leq 1} q \left | y + \frac{|x|}{t} e_1 \right |^{\nu-1} (1 - |y|^2)^{-1/2} \, \ud y   
		\right ) \\
		&= \p_t \left ( \frac{1}{t} \p_t \right ) \left ( t \abs{t}^{1+\nu} \phi(|x|/t) \right ), 
	}
	where $e_1 = (1,0,0,0)$ and 
	\begin{align}\label{eq:phidef} 
		\phi(z) = \frac{1}{8} \fint_{|y| \leq 1} q \left | y + z e_1 \right |^{\nu-1} (1 - |y|^2)^{-1/2} \, \ud y, \quad z\in [-1,1]. 
	\end{align}
	By the dominated convergence theorem, $\phi(z) \in C^4([-1,-1])$  for $\nu > \frac{9}{2}$.
	A straightforward computation shows 
	\begin{align}\label{eq:v_eq}
		v(t,x) = \abs{t}^{\nu - 1} \psi (|x|/t)
	\end{align}
	where 
	$\psi(z) = \nu(\nu+2) \phi(z) - (2\nu + 2) z \phi'(z) + z^2 \phi''(z)$. The function $\phi(z)$ is an even function which implies $\psi \in C^2([-1,1])$ is even as well. Thus, there exists a constant $C > 0$ such that 
	\begin{align*}
		|\psi(z) - \psi(0)| \leq C |z|^2, \quad |z| \leq 1,
	\end{align*}
	which implies 
	\begin{align*}
		|v(t,x) - \psi(0)\abs{ t}^{\nu - 1}| \leq C \abs{t}^{\nu-3} |x|^2. 
	\end{align*}
	Since $\psi(0) = q p(\nu)$ and $u_L(t,r) = r v(t,r)$ we conclude 
	\begin{align*}
		|u_L(t,r) - qp \abs{t}^{\nu-1} r | \leq C \abs{t}^{\nu-3} r^3, 
	\end{align*}
	which proves \eqref{eq:linearleading}.
	
	To derive the estimates \eqref{eq:lineartest} and \eqref{eq:linearrest} we observe from \eqref{eq:v_eq} that $v(t,r)$ satisfies for all $r \leq \abs{t}$  
	\begin{align}
		|v(t,r)| &\leq C \abs{t}^{\nu-1}, \quad |\nabla_{t,r} v(t,r)| \leq C \abs{t}^{\nu-2}
	\end{align}
	which imply the desired bounds for $u^*_L(t,r) = r v(t,r)$. 
	\end{proof}

\begin{lemma}\label{l:nonlinear_app}
	Let $\nu > \frac{9}{2}$.  There exists a constant $C = C(\nu,q) > 0$ such that for all $r \leq |t|$, 
	\begin{align}
		|\nabla_{t,r} u^*(t,r) - \nabla_{t,r} u^*_L(t,r)| &\leq C r|t|^{3\nu-2}, \label{eq:desired_3} \\ 
		|u^*(t,r) - u^*_L(t,r)| &\leq C r^2|t|^{3\nu-2}. \label{eq:desired4}
	\end{align}
\end{lemma}

\begin{proof}
	For $r \leq |t|$, we have by finite speed of propagation and the well-posedness theory for \eqref{eq:wmk} (see \cite{ShSt00})
	\begin{align*}
		|\nabla^2_{t,r}& u^*(t,r) - \nabla^2_{t,r} u^*_L(t,r) | \\
		&\leq C \|\nabla^2_{t,r} u^*(t) - \nabla^2_{t,r} u^*_L(t) \|_{H(r \leq t)} \\
		&\leq C \Bigl ( \| \p_r^2 u_0\|_{H(r \leq 2t)} \| u_0 \|_{H(r \leq 2t)}^2 + 
		\| \p_r u_0 \|_{H(r \leq 2t)}^2 \| u_0 \|_{H(r \leq 2t)} \Bigr ) \\
		&\leq C |t|^{3\nu-2}. 
	\end{align*}
	We then obtain \eqref{eq:desired_3} and \eqref{eq:desired4} by the fundamental theorem of calculus applied to the previous estimate once and twice respectively.  
\end{proof}

By Lemma \ref{l:linear_app} and Lemma \ref{l:nonlinear_app}, we immediately obtain the following description of the radiation $\bs u^*$ inside the light cone. 

\begin{corollary}\label{l:ustar_app}
	Let $\nu > \frac{9}{2}$ and $p = p(\nu)$ be as in Lemma \ref{l:linear_app}.  There exists $C = C(\nu,q) > 0$ such that for all $r \leq  |t|$, 
	\begin{align}
		|u^*(t,r) - q p \abs{t}^{\nu - 1} r | &\leq C r^2 \abs{t}^{\nu - 2}, \\
		|\p_t u^*(t,r)| &\leq C r \abs{t}^{\nu-2}, \label{eq:nonlinearest2}\\
		|\p_r u^*(t,r)| &\leq C \abs{t}^{\nu-1} \label{eq:nonlinearrest2}, \\
		|u^*(t,r)| &\leq C r \abs{t}^{\nu-1}.
	\end{align}
\end{corollary}

\section{Modified Modulation Method} \label{s:mod} 

In this section, we begin our study of solutions $\bs u(t)$ to~\eqref{eq:wmk} which are close to
a superposition of a bubble and the regular part $\bs u^*(t)$, in the sense that there exists $\ti \lam(t)$ such that 
\EQ{
\|\bs u(t) - (\bs u^*(t) + \bs Q_{\ti \lambda(t)})\|_\cH^2  \ll 1 
}
for all $t \in J \subset (0,T_0]$.  We will also assume for $r \geq t$ 
\begin{align}\bs u(t,r) = (\pi, 0) + \bs u^*(t,r).
\end{align} 
We remark if $T_- = 0$, then this last assumption is necessary and sufficient for $\bs u^*$ to be the radiation field of $\bs u$.  

By the implicit function, we can find a modulation parameter $\lam(t) \in C^1(J)$ so that if $\bs g(t) \in \cl H$ is defined by
\begin{align}
\bs g(t) = (g(t), \dot g(t)) := \bs u(t) - (\bs Q_\lam(t) + \bs u^*(t)),
\end{align}
then $g(t)$ satisfies a suitable orthogonality condition (see Lemma \ref{l:modeq}). Roughly stated, our goal in this section is to show the modulation parameter $\lam(t)$ satisfies differential inequalities from above and below that are, to leading order in $t$, saturated by
\begin{align}
\lam(t) &= \frac{p|q|}{\nu^2(\nu+1)} \frac{t^{\nu+1}}{|\log t|}, \\
\bs g(t,r) &= (0, -\lam'(t) \Lambda Q_{\lam(t)}(r)), \quad r \leq t.  
\end{align} 
In particular, we find that the leading order term driving the lower differential inequality is the interaction between the bubble and the radiation $\bs u^*(t,r)$ inside the light cone $\{ r \leq t \}$. 
A key idea we use in our approach, 
that was used in the study of threshold wave maps in \cite{JL1, R19}, is to study a small modification of $\lam'(t) \log \frac{t}{\lam(t)}$ having a good monotonicity property.   

\subsection{Preliminary control of the modulation parameter} 

We first state the existence and uniqueness of a $C^1$ modulation parameter such that
\begin{align}
g(t) = u(t) - (Q_{\lam(t)} + u^*(t))
\end{align}
is orthogonal to the tangent of the family of harmonic maps. Let $\chi \in C^\infty(\bR^2)$ be radial such that $\chi(r) = 1$ for $r \leq \frac{1}{2}$ and $\chi(r) = 0$ if $r \geq 1$, and let 
$\cl Z(r) = \chi(r) \Lambda Q(r)$.  Then 
\begin{align}
\int_0^\infty \cl Z(r) \Lambda Q(r) \, r\ud r > 0. \label{eq:approx_tgt}
\end{align}
By standard arguments using the implicit function theorem (see for example~\cite[Proof of Lemma 2.5]{moi15} or \cite[proof of Lemma 3.1]{JL1}), we have the following. 

\label{ssec:mod-param}
\begin{lemma}[Modulation Lemma]
\label{l:modeq}
There exist $\eta_0>0$, $\lam_0>0$ and $C > 0$ with the following property:
Let $J\subset (0, T_0]$ be a time interval,
$\bs u(t)$ a solution to~\eqref{eq:wmk} defined on $J$,
and assume there exists $\bar \lam(t)$ such that 
\EQ{
\|\bs u(t) - (\bs u^*(t) + \bs Q_{\bar \lambda(t)})\|_\cH^2 \le \eta \le  \eta_0 , \quad \bar \lam(t) \le \lam_0 \qquad \forall t\in J.
}
 Then, there exists a unique $C^1(J)$ function $\lambda(t)$ so that, defining $\bs g(t) \in \cH$ by
\begin{equation}
\label{eq:gdef1} 
\bs g(t) = (g(t), \dot g(t)):= \bs u(t) - ( \bs Q_{\lambda(t)} + \bs u^*(t)) 
\end{equation}
we have, for each $t \in J$,
\begin{gather}  
\la{ {\mathcal Z}_{\uln{\lambda(t)}},  g(t) }\ra  = 0,  \label{eq:ola} \\
\| \bs g(t) \|_\cH^2 \leq C\eta 
. \label{eq:gdotgd} 
\end{gather} 
\end{lemma} 

By differentiating the orthogonality condition satisfied by $g(t)$, we obtain the following preliminary control of the modulation parameter $\lam(t)$. 

\begin{proposition}[Modulation Control Part 1]
\label{p:modp} 
Let $\bs u(t)$ be a solution to~\eqref{eq:wmk} on $J$ as in Lemma~\ref{l:modeq} and 
let $\lambda(t)$ be given by Lemma~\ref{l:modeq}.
Then for $t \in J$
\begin{align}
|\lambda'(t)| &\lesssim \|\dot g\|_{L^2}. \label{eq:la'}
\end{align}
\end{proposition}

\begin{proof}
The function $g(t)$ satisfies
\begin{align}
\partial_t g &= \dot g + \lambda'\Lambda Q_\uln\lambda.
\end{align}
Differentiating the orthogonality condition \eqref{eq:ola} we obtain 
\begin{align}
0 &= \frac{\ud}{\ud t} \la \cl Z_{\uln \lam}, g \ra = 
- \frac{\lam'}{\lam} \la \Lam_0 \cl Z_{\uln \lam}, g \ra 
+ \lam' \la \cl Z_{\uln \lam} , \Lam Q_{\uln \lam} \ra 
+ \la \cl Z_{\uln \lam} , \dot g \ra.
\end{align}
so 
\begin{align}
\lam' \Bigl (
 \la \cl Z_{\uln \lam} , \Lam Q_{\uln \lam} \ra - \frac{1}{\lam} \la \Lam_0 \cl Z_{\uln \lam}, g \ra 
\Bigr ) = - \la \cl Z_{\uln \lam}, \dot g \ra. 
\end{align}
We have the following easily verifiable bounds:
\begin{align}
\Bigl |
\frac{1}{\lam} \la \Lam_0 \cl Z_{\uln \lam}, g \ra
\Bigr | &\lesssim \| \cl Z \|_{L^1} \| g \|_{L^\infty} \lesssim 
\| g \|_H  \lesssim \eta_0, \\
\Big |
\la\cl Z_{\uln \lam} , \dot g \ra 
\Big | &\lesssim \| \dot g \|_{L^2}.
\end{align}
The previous bounds and \eqref{eq:approx_tgt} immediately imply \eqref{eq:la'} as long as $\eta_0$ is sufficiently small. 
\end{proof}

\subsection{Sharp control of the modulation parameter}

We now turn to obtaining an improved control of the modulation parameter $\lam(t)$. We first define an auxilary function $\zeta \in C^1(J)$ by 
\begin{align}
\zeta(t) &:= 4 \lam(t) \log \frac{t}{\lam(t)} - \int_0^{t} \Lambda Q_{\uln{\lam(t)}} g(t) \, r \ud r,  \label{eq:zetadef} 
\end{align}  
which serves as a better proxy for the dynamics than $\lam(t)$ but is still close to $\lam(t)$ (see Proposition \ref{p:modp2}). Since \eqref{eq:wmk} is second order in time, precise control of $\zeta(t)$ (and thus $\lam(t)$) is obtained by studying $\zeta''(t)$. It is not difficult to see that, up to formally acceptable error terms in $t$, 
\begin{align}
\zeta'(t) \approx -\int_0^t \Lambda Q_{\uln{\lam(t)}} \dot g(t) \, r \ud r.
\end{align}
However, it is difficult to prove the formally desired lower bound 
\begin{align}
 -\frac{\ud}{\ud t} \int_0^t \Lambda Q_{\uln{\lam(t)}} \dot g(t) \, r \ud r \geq  
 4p|q| t^{\nu-1}(1 + o(1)),
\end{align}
since several terms arising on the left hand side above have critical size and indefinite sign. To overcome this obstacle, we instead study a function $b(t)$ which is a (small) correction to $\zeta'(t)$ by a truncated virial functional. The introduction of truncated virial functional produces terms which are then large enough to cancel the bad terms which originally arose. The use of such an ad hoc correction was inspired by Rapha{\"e}l, Szeftel~\cite{RaSz11}, and was used crucially in this context in works of the authors~\cite{JJ-AJM, JL1, R19}. To define and study the properties of the function $b(t)$, we require the following lemmas. 

\begin{lemma} \emph{\cite[Lemma 4.6]{JJ-AJM}}
	\label{lem:fun-q}
	For each $c, R > 0$ there exists a function $$q(r) = q_{c, R}(r) \in C^{3,1}((0, \infty))$$ with the following properties:
	\begin{enumerate}[label=(P\arabic*)]
		\item $q(r) = \frac{1}{2} r^2$ for $r \leq R$, \label{enum:approx-q}
		\item there exists an absolute constant $\kappa > 0$  such that $q(r) \equiv \tx{const}$ for $r \geq \tilde R := \kappa e^{\kappa/c} R$, \label{enum:support-q}
		\item $|q'(r)| \lesssim r$ and $|q''(r)| \lesssim 1$ for all $r > 0$, with constants independent of $c, R$, \label{enum:gradlap-q}
		\item $q''(r) \geq -c$ and $\frac 1r q'(r) \geq -c$, for all $r > 0$, \label{enum:convex-ym}
		\item $(\frac{\ud^2}{\ud r^2} + \frac 1r \frac{\ud }{\ud r} r)^2 q(r) \leq c\cdot r^{-2}$, for all $r > 0$, \label{enum:bilapl-ym}
		\item $\big|r\big(\frac{q'(r)}{r}\big)'\big| \leq c$, for all $r > 0$. \label{enum:multip-petit-ym}
	\end{enumerate}
\end{lemma}

For each $\lambda > 0$ we define the operators $\A(\lambda)$ and $\A_0(\lambda)$ as follows:
\begin{align}
[\A(\lambda)g](r) &:= q'\big(\frac{r}{\lambda}\big)\cdot \p_r g(r), \label{eq:opA-wm} \\
[\A_0(\lambda)g](r) &:= \Big(\frac{1}{2\lambda}q''\big(\frac{r}{\lambda}\big) + \frac{1}{2r}q'\big(\frac{r}{\lambda}\big)\Big)g(r) + q'\big(\frac{r}{\lambda}\big)\cdot\p_r g(r). \label{eq:opA0-wm}
\end{align}
Since $q(r) = \frac{1}{2} r^2$ for $r \leq R$, $\A(\lam) g(r) = \frac{1}{\lam} \Lambda g(r)$ and $\A_0(\lam) g(r) = \frac{1}{\lam} \Lambda_0 g(r)$ for $r \leq  \lam R$.  One may intuitively think of $\A(\lam)$ and $\A_0(\lam)$ as truncations of $\frac{1}{\lam} \Lambda$ and $\frac{1}{\lam} \Lambda_0$ to $r \geq \lam R$ which have good boundedness properties.  The following lemma makes this precise.  In what follows, we denote 
\begin{align}
	X:= \left \{ g \in H \mid \frac{g}{r}, \p_r g \in H \right \}.
\end{align}

\begin{lemma} 
	\label{lem:op-A-wm}
	Let $c_0>0$. There exists $c>0$ small enough and $R, \tilde R>0$ large enough in Lemma~\ref{lem:fun-q} such that $\A(\lambda)$ and $\A_0(\lambda)$ defined in~\eqref{eq:opA-wm} and~\eqref{eq:opA0-wm} have the following properties:
	\begin{itemize}[leftmargin=0.5cm]
		\item The families $\{\A(\lambda): \lambda > 0\}$, $\{\A_0(\lam): \lam > 0\}$, $\{\lambda\partial_\lambda \A(\lambda): \lambda > 0\}$
		and $\{\lam\partial_\lam \A_0(\lambda): \lambda > 0\}$ are bounded in $\mathscr{L}(H; L^2)$, where bound depends only on the choice of the function $q(r)$,
		\item 
		For all $\lam > 0$ and $g_1, g_2 \in X$  there holds
		\begin{multline}  \label{eq:A-by-parts-wm}
		\Big| \Bigl \langle  \A(\lam)g_1 ,  \frac{1}{r^2}\big(f(g_1 + g_2) - f(g_1) - f'(g_1)g_2\big) \Big \rangle  \\ +\Big \la \A(\lambda)g_2 , \frac{1}{r^2}\big(f(g_1+g_2) - f(g_1) + g_2\big) \Bigr \ra \Big| 
		\leq \frac{c_0}{\lambda} \|g_2\|_H^2, 
		\end{multline}
		\item For all $g \in X$,   
		\begin{gather}
			\label{eq:A-pohozaev-wm}
			\Big \la \A_0(\lambda)g , \big(\partial_r^2 + \frac 1r\partial_r - \frac{1}{r^2}\big)g \Bigr \ra \leq \frac{c_0}{\lambda}\|g\|_{H}^2 - \frac{1}{\lambda}\int_0^{R\lambda}\Big((\partial_r g)^2 + \frac{1}{r^2}g^2\Big) r \ud r, 
		\end{gather}
		\item In addition we have the bounds,  
		\begin{gather}
		\label{eq:L-A-wm}
		\|\Lam Q_\uln\lam - \A(\lambda)Q_\lam\|_{L^\infty} \leq \frac{c_0}{\lambda},  \\
		\|\Lam_0 \Lambda Q_\uln\lambda - \A_0(\lam) \Lambda Q_\lambda\|_{L^2} \leq c_0, \label{eq:Al2} 
		\end{gather} 
		and there exist $T = T(c_0, \nu, q) > 0$ and $C = C(\nu,q) > 0$ such that if $t \leq T$ and $\{ r \leq \tilde R \lam \} \subset \{ r \leq t\}$, then for all $g \in H$, 
		\begin{multline}   \label{eq:approx-potential-wm}
		\bigg|\int_0^{\infty}\frac 12 \Big(q''\big(\frac{r}{\lam}\big) + \frac{\lambda}{r}q'\big(\frac{r}{\lambda}\big)\Big)\frac{1}{r^2}\big(f( Q_\lambda + u^* + g) - f(Q_\lam + u^*)- g\big)g \, r \ud r \\
		- \int_0^{\infty} \frac{1}{r^2}\big(f'(Q_\lam)- 1\big)|g|^2 \, r \ud r \bigg| \leq c_0 \|g\|_H^2 + 
		C \| g \|_H^3.  
		\end{multline} 
	\end{itemize}
\end{lemma}

\begin{proof}
For the proofs of all statements except \eqref{eq:approx-potential-wm} we refer the reader to \cite[Lemma 5.5]{JJ-AJM}.  By trigonometric identities we compute 
\begin{align}
f(Q_\lam + u^* + g) - f(Q_\lam + u^*) - g 
&=  -\sin (2 Q_\lam + 2 u^*) \sin^2 g \\&\quad - \bigl ( g - \frac{1}{2} \sin 2 g \bigr ) \cos (2 Q_\lam + 2 u^*) \\
&\quad - g\cos 2 Q_\lam (1-\cos 2u^*) - g \sin 2 Q_\lam \sin 2 u^* \\&\quad - (1 - \cos 2 Q_\lam)g, \\
f'(Q_\lam) - 1 &= \cos 2 Q_\lam -1.     
\end{align}
By Corollary \ref{l:ustar_app} and our assumption $\{ r \leq \tilde R \lam \} \subset \{r \leq t\}$, we have the pointwise estimate 
\begin{align}
|1 - \cos 2u^*| + |\sin 2 u^*| \lesssim t^{\nu},
\end{align}
which along with \ref{enum:gradlap-q} imply 
\begin{align} 
\begin{split}
\bigg|\int_0^{\infty}\frac 12 \Big(q''\big(\frac{r}{\lambda}\big) + \frac{\lambda}{r}q'\big(\frac{r}{\lambda}\big)\Big)\frac{1}{r^2}\big(f( Q_\lambda + u^* + g) - f(Q_\lambda + u^*)-g\big)g \, r \ud r \\
- \int_0^{\infty} \frac 12 \Big(q''\big(\frac{r}{\lambda}\big) + \frac{\lambda}{r}q'\big(\frac{r}{\lambda}\big)\Big)\frac{1}{r^2}( \cos 2 Q_\lam-1)|g|^2 r \ud r \bigg| \leq \frac{c_0}{2} \|g\|_H^2 + C \| g \|^3_H
\end{split}\label{eq:enum_est1}
\end{align}
as long as $T$ is sufficiently small. Now, for $r \leq \lam R$, we have by 
\ref{enum:approx-q}  
\begin{align} \frac 12 \Big(q''\big(\frac{r}{\lambda}\big) + \frac{\lambda}{r}q'\big(\frac{r}{\lambda}\big)\Big)\frac{1}{r^2}( \cos 2 Q_\lam-1) = \frac{1}{r^2} ( \cos 2Q_\lam-1). 
\end{align}
Thus, by \ref{enum:gradlap-q}
\begin{align}
\begin{split}
\bigg|\int_0^{\infty}\frac 12 \Big(q''\big(\frac{r}{\lambda}\big) + \frac{\lambda}{r}q'\big(\frac{r}{\lambda}\big)\Big)&\frac{1}{r^2}( \cos 2Q_\lam-1)|g|^2 r \ud r 
- \int_0^{\infty} \frac{1}{r^2}( \cos 2 Q_\lam-1)|g|^2 r \ud r\bigg| \\
&\lesssim \int_{R\lam}^\infty |1 - \cos 2Q_\lam| |g|^2 \frac{\ud r}{r} \\
&\lesssim \| g \|_H^2 \sup_{r \geq R} |1 - \cos 2Q| \\
&\leq \frac{c_0}{2} \| g \|_H^2
\end{split}\label{eq:enum_est2}
\end{align}
as long as $R$ is sufficiently large. Combining \eqref{eq:enum_est1} and \eqref{eq:enum_est2} yields the claim.  
\end{proof}

We now define a second auxiliary function $b(t)$ by   
\begin{align}
b(t) &:= - \int_0^{t} \Lambda Q_{\uln{\lam(t)}} \dot g(t) \, r \ud r  - \la \dot g(t), \A_0(\lambda(t)) g(t)\ra. \label{eq:bdef}
\end{align}
 As stated previously, we will  show below that we can think of $b(t)$ as a subtle monotonic correction to the derivative $\zeta'(t)$.
 We note the error term $\bs g(t)$ satisfies the differential equation
 \begin{align}
 \begin{split}
 \partial_t g &= \dot g + \lambda'\Lambda Q_\uln\lambda, \\
 \partial_t \dot g &= \Delta g - \frac{1}{r^2}\big(f(u^* + Q_\lambda + g) - f(u^*) - f(Q_\lambda)\big),
 \end{split}\label{eq:g_equations}
 \end{align}
 where $\Delta = \frac{1}{r} \p_r ( r \p_r \cdot )$ and $f(u) = \frac{1}{2} \sin 2u$. Moreover, the assumption $\bs u(t,r) = \bs u^*(t,r)$ for all $r \geq t$ implies 
 \begin{align}
 \bs g(t,r) = (\pi - Q_{\uln{\lam(t)}}(r), 0 ), \quad \forall r \geq t. 
 \end{align}

\begin{proposition}[Modulation Control Part 2]
\label{p:modp2} 
Assume the same hypothesis as in  Proposition~\ref{p:modp}. Let $\delta>0$
be arbitrary and let $\eta_0$ be as in Lemma~\ref{l:modeq}. 
Let $\zeta(t), b(t)$ be as in~\eqref{eq:zetadef}, \eqref{eq:bdef}. In addition, assume there exists a constant $\alpha > 0$ such that for all $t \in J$ 
\begin{align}\label{eq:lam_size_assumption1}
\frac{\lam(t)}{t} < \alpha
\end{align}
Then there exist $\eta_1 = \eta_1(\nu,q, \delta) < \eta_0$ and a constant $C_0 = C_0(\nu,q,\alpha) > 0$ such that
if $$\| \bs g(t) \|_{\HH}^2 + \alpha + \sup J \le \eta_1, \quad \forall t \in J,$$ then
\begin{align}
\left | \frac{\zeta(t)}{4 \lam(t)\log \frac{t}{\lam(t)}} - 1 \right | &\leq C_0 \frac{t}{\lam(t) \log \frac{t}{\lam(t)}} \| \bs g(t) \|_{\mathcal H} \label{eq:bound-on-l}, \\
|b(t)| &\leq 
 ( 4 + \delta )^{\frac{1}{2}} \left (\log \frac{t}{\lam(t)} \right )^{\frac{1}{2}} \| \dot g(t) \|_{L^2} + 
C_0 \| \bs g(t) \|_{\mathcal H}^2
\label{eq:b-bound}, \\
|\zeta'(t) - b(t) |&\le C_0  \left (
\| \dot g(t) \|_{L^2} + 
 \frac{\lam(t)}{t}\right ).  \label{eq:kala'}
\end{align}
In addition, $b(t)$ is locally Lipschitz and the derivative $b'(t)$ satisfies
\begin{align}
|b'(t)| &\leq C_0 \left ( t^{\nu-1} + \frac{\lam(t)}{t^2} + \frac{1}{\lam} \| \bs g(t) \|^2_{\cl H} \right ), \\
b'(t) &\ge (4p|q| - \delta) t^{\nu-1} - C_0 \frac{\lam(t)}{t^2} - \delta \frac{1}{\lam} \| \bs g(t) \|_{\cl H}^2. \label{eq:b'lb}  
\end{align} 
\end{proposition} 

\begin{proof}
To prove the first estimate, we note
	 \begin{align*}
\left | 
\int_0^t \Lambda Q_{\uln \lam} g \, r \ud r \right |\lesssim \lam \| g \|_{L^\infty(r \leq t)} 
\int_0^{t/\lam} |\Lambda Q| \, r\ud r \lesssim t \| g \|_{L^\infty(r \leq t)}
	 \end{align*}
	and conclude  
	 \begin{align*}
	 \frac{1}{4 \lam \log (t/\lam)}\left | 
	 \int_0^t \Lambda Q_{\uln \lam} g \, r \ud r \right |
	 \lesssim \frac{t}{\lam \log (t/\lam)} \| g \|_{H(r \leq t)}
	 \end{align*}
 which proves \eqref{eq:bound-on-l}.

To prove \eqref{eq:b-bound}, we first note  
\begin{align}
\left | 
\la \dot g , \mathcal A_0(\lam) g \ra 
\right | \leq C \| \dot g \|_{L^2} \| g \|_{H} \leq C \| \bs g \|_{\mathcal H}^2. \label{eq:b-bound1}
\end{align}
By Cauchy-Schwarz and a change of variables we have 
\begin{align}
\left | \int_0^t \Lambda Q_{\uln \lam} \dot g \, r \ud r \right |
\leq \| \Lambda Q \|_{L^2(r \leq t/\lam)} \| \dot g \|_{L^2(r \leq t)}. \label{eq:b-bound3}
\end{align} 
Since 
\begin{align}
\| \Lambda Q \|_{L^2(r \leq t/\lam)}^2 = 4 \log ( t/\lam ) + O(1), \label{eq:lamQL2}
\end{align} 
\eqref{eq:b-bound1} and \eqref{eq:b-bound3} imply \eqref{eq:b-bound} as long as $\eta_1$ is sufficiently small.

We now turn to \eqref{eq:kala'}.  We have 
\begin{align*}
\zeta' &= 4 \lam' \log \frac{t}{\lam} + 4 \frac{\lam}{t} - 4 \lam' - \frac{\ud}{\ud t} \int_0^t \Lambda Q_{\uln \lam}  g \, r \ud r \\
&= 4 \lam' \log \frac{t}{\lam} + 4 \frac{\lam}{t} - 4 \lam' -  \int_0^t  \Lambda Q_{\uln \lam}  \dot g \, r \ud r - \lam' \int_0^t |\Lambda Q_{\uln \lam} |^2 \, r \ud r  \\
&\,-\frac{\lam'}{\lam}
\int_0^t |\Lambda_0 \Lambda Q]_{\uln \lam} g \, r \ud r +
\Lambda Q_{\uln \lam} g \,r \big |_{r = t}. 
\end{align*}
By \eqref{eq:lamQL2} and the definition of $b$ it follows     
\begin{align}
|\zeta'(t) - b(t)| &\lesssim |\lam'|
\Bigl (
1 + |\la \dot g , \mathcal A_0(\lam) g \ra | 
 + \frac{1}{\lam} 
\int_0^t |[\Lambda_0 \Lambda Q]_{\uln \lam}| |g| \, r \ud r \Bigr ) \\
&\qquad+ \frac{\lam}{t} + 
\Bigl | \Lambda Q_{\uln \lam} g \,r \big |_{r = t} \Bigr |.
\end{align}
Since, for all $t \in J$, 
\begin{align}
\| \bs g(t) \|_{\cl H} &\lesssim 1, \\
|\lam'| &\lesssim \| \dot g \|_{L^2},
\end{align}
and $\int_0^{t} |\Lambda_0 \Lambda Q | \, r \ud r \lesssim 1$, we conclude  
\begin{align}
|\la \dot g , \mathcal A_0(\lam) g \ra | + \frac{1}{\lam} 
\int_0^t |[\Lambda_0 \Lambda Q]_{\uln \lam}| |g| \, r \ud r \lesssim 
\| \bs g \|_{\cl H}^2 + \| g \|_{L^\infty} \int_0^{t} |\Lambda_0 \Lambda Q | \, r \ud r \lesssim 1. 
\end{align}
Since $g(t,t) = \pi - Q_{\lam(t)}(t) = O(\lam(t) t^{-1})$, 
\begin{align*}
\left |
\Lambda Q_{\uln \lam} g\, r \Big |_{r = t} 
\right | \lesssim \frac{\lam}{t} 
\end{align*}
Thus,
\begin{align}
|\zeta' - b| \lesssim 
\| \dot g \|_{L^2}
 + \frac{\lam}{t}
\end{align}
as long as $\eta_1$ is sufficiently small. This proves \eqref{eq:kala'}.

We now turn to the delicate proof of \eqref{eq:b'lb}. To lessen notation, we denote 
\begin{align}
\la h_1, h_2 \ra_{loc} = \int_0^t h_1(r) h_2(r) \, r \ud r
\end{align}
We have 
\begin{align}
\begin{split}\label{eq:b'1}
b'(t) &= \frac{\lam'}{\lam} \la [\Lambda_0 \Lambda Q]_{\uln \lam} , \dot g \ra_{loc} - \la \Lambda Q_{\uln \lam} , \partial_t \dot g \ra_{loc} - \la \p_t \dot g , \A_0(\lam) g \ra \\
&\quad-\frac{\lam'}{\lam} \la \dot g , \lam \p_{\lam} \A_0(\lam) g \ra - 
\la \dot g , \A_0(\lam) \p_t g \ra - \Lambda Q_{\uln \lam} \dot g \big |_{r = t}  \\
&= \frac{\lam'}{\lam} \la [\Lambda_0 \Lambda Q]_{\uln \lam} , \dot g \ra_{loc} \\
&\quad- \left \la \Lambda Q_{\uln \lam} ,
\Delta g - \frac{1}{r^2} (f(u^* + Q_\lam + g) - f(u^*) - f(Q_\lam))
\right  \ra_{loc} \\
&\quad- \left \la \Delta g - \frac{1}{r^2} (f(u^* + Q_\lam + g) - f(u^*) - f(Q_\lam)) , \A_0(\lam) g \right \ra_{loc} \\
&\quad-\frac{\lam'}{\lam} \la \dot g , \lam \p_{\lam} \A_0(\lam) g \ra - 
\la \dot g , \A_0(\lam) \dot g \ra - 
\lam' \left \la \dot g , \A_0(\lam)  \Lambda Q_{\uln \lam} \right \ra. 
\end{split}
\end{align}
To pass from the first equality to the second, we used $\bs g(t,r) = \bs (\pi - Q_{\lam(t)}, 0)$ for all $r \geq t$ and \eqref{eq:g_equations}. 

Since $\dot g(t,r) = 0$ for $r \geq t$, we have 
\begin{align}
\left \la \dot g , \A_0(\lam)  \Lambda Q_{\uln \lam} \right \ra
= \left \la \dot g , \A_0(\lam)  \Lambda Q_{\uln \lam} \right \ra_{loc}
\end{align}
Thus by \eqref{eq:Al2} we have the first and sixth terms combined satisfy 
\begin{align}
\frac{|\lam'|}{\lam} \Big | \la [\Lambda_0 \Lambda Q]_{\uln \lam} - \cl A_0(\lam) \Lambda Q_\lam , \dot g \ra_{loc} \Bigr |
&\leq  \frac{|\lam'|}{\lam} \| \dot g \|_{L^2} \| [\Lambda_0 \Lambda Q]_{\uln \lam} - \cl A_0(\lam) Q_\lam \|_{L^2(r \leq t)} \\&\lesssim c_0 \lam^{-1} \| \dot g \|_{L^2}^2
\leq \frac{\delta}{100} \lam^{-1} \| \bs g \|_{\cl H}^2     
\end{align}
as long as $c_0$ is chosen small enough. 
Again by Lemma \ref{lem:op-A-wm} the fourth term appearing in $b'$ satisfies 
\begin{align*}
\frac{|\lam'|}{\lam} \left |
\la \dot g , \lam \p_{\lam} \A_0(\lam) g \ra
\right | \lesssim \frac{|\lam'|}{\lam} \| \bs g \|_{\cl H}^2 \leq 
\lam^{-1} \| \bs g \|_{\cl H}^3 \leq \frac{\delta}{100} \lam^{-1} \| \bs g \|_{\cl H}^2,
\end{align*} 
as long as $\eta_1$ is sufficiently small. 
Since $(\Delta - \frac{1}{r^2} f'(Q_\lam) ) \Lambda Q_{\uln \lam}  = 0$ we have by integration by parts  
\begin{align}
\la \Lambda Q_{\uln \lam} , \Delta g - \frac{1}{r^2} f'(Q_\lam) g \ra_{loc} 
&= r \Lambda Q_{\uln \lam} \p_r g \big |_{r = t} - r \p_r \Lambda Q_{\uln \lam} \, g \big |_{r = t}.
\end{align}
Since  
$g(t,r) = \pi - Q_{\lam(t)}(r)$ for all $r \geq t$, we conclude  
\begin{align}
\left | r
\Lambda Q_{\uln \lam} \p_r g \big |_{r = t} - r
\p_r \Lambda Q_{\uln \lam} \, g \big |_{r = t}
\right | &\lesssim t^{-1} |g(t,t)| + |(\p_r g)(t,t)| \\
&\lesssim  \lam t^{-2}
\end{align}
as long as $\eta_1$ is sufficiently small. 
We denote 
\begin{align*}
I(t) &:=  \Bigl \la \Lambda Q_{\uln \lam} ,
\frac{1}{r^2} (f(u^* + Q_\lam + g) - f(u^*) - f(Q_\lam) - f'(Q_\lam) g)
\Bigr  \ra_{loc} \\
&\quad- \Bigl \la \A_0(\lam) g , \Delta g - 
\frac{1}{r^2} (f(u^* + Q_\lam + g) - f(u^*) - f(Q_\lam)
\Bigr \ra_{loc}.
\end{align*}
Then up to this point, we have proved  
\begin{align}
\bigl | b'(t) - I(t) \bigr | \leq  C_0 \frac{\lam(t)}{t^2} + \frac{\delta}{50} \lam^{-1} \| \bs g(t) \|_{L^2}^2 
\end{align}
We now turn to the study of $I(t)$.  

We write 
\begin{align*}
I(t) &=  \Bigl \la \Lambda Q_{\uln \lam} ,
\frac{1}{r^2} ( f(Q_\lam + u^*) - f(Q_\lam) - f(u^*) ) \Bigr \ra_{loc} \\
&\quad +  \Bigl \la \Lambda Q_{\uln \lam} ,
\frac{1}{r^2} ( f'(Q_\lam + u^*)g - f'(Q_\lam)g ) \Bigr \ra_{loc} \\
&\quad + \Bigl \la \Lambda Q_{\uln \lam} ,
 \frac{1}{r^2} ( f(Q_\lam + u^* + g) - f(Q_\lam + u^*) - f'(Q_\lam + u^*)g ) \Bigr \ra_{loc}  \\
&\quad - \Bigl \la \A_0(\lam) g , \Delta g - 
\frac{1}{r^2} (f(u^* + Q_\lam + g) - f(u^*) - f(Q_\lam)
\Bigr \ra_{loc}.
\end{align*}

We claim the first term appearing in the expression for $I(t)$ contains the leading order: 
\begin{align}\label{eq:bder_leading}
\left | \Bigl \la \Lambda Q_{\uln \lam} ,
\frac{1}{r^2} ( f(Q_\lam + u^*) - f(Q_\lam) - f(u^*) ) \Bigr \ra_{loc}
+ 4pq t^{\nu-1} \right | \leq C_0 \lam t^{-2} + \frac{8 p |q| \delta}{100} t^{\nu-1}. 
\end{align}
By trigonometric identities and the fact that $\sin Q_\lam = \Lambda Q_\lam$ we have 
\begin{align}
(f(Q_\lam + u^*) - f(Q_\lam) - f(u^*)) = -\sin 2 Q_\lam \sin^2 u^* - \sin 2 u^*
| Q_\lam|^2. \label{eq:trig_exp}
\end{align}
The first term appearing in \eqref{eq:trig_exp} contributes to the $L^2$ pairing 
\begin{align}
\left |
\Bigl \la \Lambda Q_{\uln \lam} , \frac{\sin 2Q_\lam \sin^2 u^*}{r^2} \Big \ra_{loc}
\right |
&\lesssim \frac{1}{\lam} \int_0^t |\Lambda Q_{\lam}|^2 |u^*|^2 \frac{\ud r}{r} \\
&\lesssim \frac{1}{\lam} \int_0^t |\Lambda Q_{\lam}|^2 r^2 t^{2\nu-2} \, \frac{\ud r}{r} \\
&\lesssim t^{2\nu-1} \int_0^{t/\lam} |\Lambda Q|^2 \ud r \\
&\leq \frac{4 p|q| \delta}{100} t^{\nu-1} \label{eq:leading2}
\end{align}
as long as $\eta_1$ is sufficiently small. By Lemma \ref{l:ustar_app} we have for all $r \leq t$
\begin{align*}
-\sin 2 u^*(t,r) = -2u^*(t,r) + O(|u^*(t,r)|^3) = 
-2 q p t^{\nu-1} r + O(r^3 t^{\nu-3} + rt^{3\nu-1}).  
\end{align*}
Thus,  
\begin{align}
\Bigl | \Bigl \la 
\Lambda Q_{\uln \lam} , -\frac{\sin 2 u^* |\Lambda Q_\lam|^2}{r^2} 
\Bigr \ra_{loc} &+ 
2pq t^{\nu-1} \int_0^t |\Lambda Q_{\lam}|^3 \, \frac{\ud r}{\lam}
\Bigr | \\&\lesssim \frac{1}{\lam} t^{\nu-3} \int_0^t |\Lambda Q_{\lam}|^3 \, r^2 \ud r
+ t^{3\nu-1} \int_0^t |\Lambda Q_{\lam}|^3 \, \frac{\ud r}{\lam} \\
&\lesssim \lam t^{\nu-2} \int_0^{t/\lam} |\Lambda Q|^3 r \ud r + t^{3\nu-1} \\
&\leq C \lam t^{-2} + \frac{4p|q|}{100} t^{\nu-1} 
\end{align}
Since $\int_0^\infty |\Lambda Q|^3 r \ud r = 2$ and 
\begin{align}
\int_{t}^\infty |\Lambda Q_{\lam}|^3 \frac{\ud r}{\lam} = 
\int_{t/\lam}^\infty |\Lambda Q|^3 \, \ud r 
\lesssim 
\lam^2 t^{-2}
\end{align}
as long as $\eta_1$ is sufficiently small, we conclude  
\begin{align}
\left | \Bigl \la \Lambda Q_{\uln \lam} , -\frac{\sin 2 u^* |\Lambda Q_\lam|^2}{r^2} 
\Bigr \ra_{loc}  + 
4pq t^{\nu-1}
\right | \leq C_0 \lam t^{-2} + \frac{8 p |q| \delta}{100} t^{\nu-1}.  \label{eq:leading} 
\end{align}
By \eqref{eq:trig_exp}, \eqref{eq:leading2} and \eqref{eq:leading} we obtain \eqref{eq:bder_leading}. 

We now introduce some notation.  Until the end of the proof, we write $A \simeq B$ if $A = B + O(\lam t^{-2}) + o(1)(t^{\nu-1} + \lam^{-1} \| \bs g(t) \|_{\cl H}^2)$, where $o(1) \rar 0$ uniformly for $t \leq \eta_1$ as $\eta_1 \rar 0$.  Thus, up to this point in the argument, we have shown  
\begin{align} 
b'(t) + 4 q p t^{\nu-1} \simeq
& \Bigl \la \Lambda Q_{\uln \lam} ,
\frac{1}{r^2} ( f(Q_\lam + u^* + g) - f(Q_\lam + u^*) - f'(Q_\lam + u^*)g ) \Bigr \ra_{loc}  \label{eq:first_term} \\
&- \Bigl \la \A_0(\lam) g , \Delta g - 
\frac{1}{r^2} (f(u^* + Q_\lam + g) - f(u^*) - f(Q_\lam)
\Bigr \ra_{loc}. \label{eq:second_term}
\end{align}
We rewrite \eqref{eq:first_term} as 
\EQ{
\Bigl \la& \Lambda Q_{\uln \lam} \,,\, 
\frac{1}{r^2} ( f(Q_\lam + u^* + g) - f(Q_\lam + u^*) - f'(Q_\lam + u^*)g ) \Bigr \ra \\
&= - \Bigl \la 
\cl A(\lam) g \,,\, \frac{1}{r^2} (f(Q_\lam + u^* + g) - f(Q_\lam + u^*) + g) 
\Bigr \ra \\
&\quad + \Bigl \la 
\cl A(\lam) g \,,\, \frac{1}{r^2} (f(Q_\lam + u^* + g) - f(Q_\lam + u^*) + g) 
\Bigr \ra \\
 &\quad + \Bigl \la 
\cl A(\lam) (Q_\lam + u^*) \,,\, \frac{1}{r^2} (f(Q_\lam + u^* + g) - f(Q_\lam + u^*) - f'(Q_\lam + u^*) g ) 
\Bigr \ra \\
&\quad - \Bigl \la 
\cl A(\lam) u^* \,,\, \frac{1}{r^2} (f(Q_\lam + u^* + g) - f(Q_\lam + u^*)
- f'(Q_\lam + u^*) g )\Bigr \ra \\
&\quad + \Bigl \la 
\Lambda Q_\lam - \cl A(\lam) Q_\lam \,,\, \frac{1}{r^2} (f(Q_\lam + u^* + g) - f(Q_\lam + u^*) - f'(Q_\lam + u^*) g 
\Bigr \ra_{loc}. 
\label{eq:A_expansion}
}
Here we used $\la \cl A(\lam) f, h \ra = \la \cl A(\lam) f, h \ra_{loc}$ for all $f$ and $h$ as long as $\eta_1$ is sufficiently small. 
By \eqref{eq:A-by-parts-wm} with $g_1 = Q_\lam + u^*$ and $g_2 = g$, the second and third terms right of the equal sign in \eqref{eq:A_expansion} satisfy
\begin{gather}
\Bigl |
\Bigl \la 
\cl A(\lam) g \,,\, \frac{1}{r^2} (f(Q_\lam + u^* + g) - f(Q_\lam + u^*) + g) 
\Bigr \ra_{loc} \\
+ \Bigl \la 
\cl A(\lam) (Q_\lam + u^*) \,,\, \frac{1}{r^2} (f(Q_\lam + u^* + g) - f(Q_\lam + u^*) - f'(Q_\lam + u^*) g ) \Bigr | \\ \leq \frac{c_0}{\lam} \| g \|_H^2 
\ll \frac{1}{\lam} \| g \|_{H}^2, 
\end{gather}
as long as $c_0$ is sufficiently small. The pointwise estimate 
\begin{gather}
|f(Q_\lam + u^* + g) - f(Q_\lam + u^*) - f'(Q_\lam + u^*) g | \lesssim |g|^2 
\end{gather}
implies the second to last line of \eqref{eq:A_expansion} satisfies 
\begin{align}
\Bigl | 
\Bigl \la 
\cl A(\lam) u^* \,,\, \frac{1}{r^2} &(f(Q_\lam + u^* + g) - f(Q_\lam + u^*)
- f'(Q_\lam + u^*) g )\Bigr \ra \Bigr | \\&\lesssim \| \cl A(\lam) u^* \|_{L^\infty} 
\| g \|_{H}^2 \ll \frac{1}{\lam} \| g \|_H^2 
\end{align}
as long as $\eta_1$ is sufficiently small. 
In the last estimate, we used $\| \cl A(\lam) u^* \|_{L^\infty} \lesssim 
\| \p_r u^* \|_{L^\infty} \lesssim 1$ which follows easily from the definition of $\cl A(\lam)$ and Lemma \ref{lem:fun-q}.  Using \eqref{eq:L-A-wm}, we can estimate the last line of \eqref{eq:A_expansion} similarly by 
\begin{gather}
\Bigl |
\Bigl \la 
\Lambda Q_\lam - \cl A(\lam) Q_\lam \,,\, \frac{1}{r^2} (f(Q_\lam + u^* + g) - f(Q_\lam + u^*) - f'(Q_\lam + u^*) g 
\Bigr \ra_{loc}
\Bigr | \\ \lesssim \frac{c_0}{\lam} \| g \|_H^2 \ll \frac{1}{\lam} \| g \|_{H}^2
\end{gather} 
as long as $c_0$ is sufficiently small.  In summary, we have shown  
\begin{gather}
\Bigl \la \Lambda Q_{\uln \lam} \,,\,
\frac{1}{r^2} ( f(Q_\lam + u^* + g) - f(Q_\lam + u^*) - f'(Q_\lam + u^*)g ) \Bigr \ra \\
\simeq -  \Bigl \la \cl A(\lam) g \, , \, 
\frac{1}{r^2} (f(Q_\lam + u^* + g) - f(Q_\lam + u^*) + g)
\Bigr \ra. \label{eq:first_term_app} 
\end{gather} 
which by \eqref{eq:first_term} implies 
\begin{align}
\begin{split}\label{b'_third_last}
b' + 4 pq t^{\nu-1} &\simeq  
-  \Bigl \la \cl A(\lam) g \, , \, 
\frac{1}{r^2} (f(Q_\lam + u^* + g) - f(Q_\lam + u^*) + g)
\Bigr \ra  \\
&\quad - \Bigl \la \A_0(\lam) g , \Delta g - 
\frac{1}{r^2} (f(u^* + Q_\lam + g) - f(u^*) - f(Q_\lam)
\Bigr \ra_{loc}.
\end{split}
\end{align}

We now consider the second term appearing in \eqref{b'_third_last}.  Since $\cl A_0(\lam)g$ is supported in $\{ r \leq \tilde R \lam \} \subset \{ r \leq t\}$ if $\eta_1$ is sufficiently small, we may replace the local $L^2$ pairing by the full $L^2$ pairing. Adding and subtracting terms and using \eqref{eq:A-pohozaev-wm} we obtain 
\EQ{
- \Bigl \la \A_0(\lam) g &, \Delta g - 
\frac{1}{r^2} (f(u^* + Q_\lam + g) - f(u^*) - f(Q_\lam)
\Bigr \ra_{loc} \\
&= - \Bigl \la 
\cl A_0(\lam) g \, , \, 
\p_r^2 g + \frac{1}{r} \p_r g - \frac{1}{r^2} g \Bigr \ra \\
&\quad + \Big \la \cl A_0(\lam) g \, , \, 
\frac{1}{r^2} (f(Q_\lam + u^*) - f(Q_\lam) - f(u^*)) \Bigr \ra \\
&\quad + \Big \la \cl A_0(\lam) g \, , \, 
\frac{1}{r^2} (f(Q_\lam + u^*+g) - f(Q_\lam+u^*) - g) \Bigr \ra \\
&\geq -\frac{c_0}{\lam} \| g \|_H^2 + \frac{1}{\lam} \int_0^{R\lam}
\Bigl ( |\p_r g|^2 + \frac{1}{r^2} |g|^2 \Bigr )\, r \ud r \\
&\quad + \Big \la \cl A_0(\lam) g \, , \, 
\frac{1}{r^2} (f(Q_\lam + u^*) - f(Q_\lam) - f(u^*)) \Bigr \ra \\
&\quad + \Big \la \cl A_0(\lam) g \, , \, 
\frac{1}{r^2} (f(Q_\lam + u^*+g) - f(Q_\lam+u^*) - g) \Bigr \ra,
}
where $R$ is defined in Lemma \ref{lem:fun-q}. By \eqref{eq:trig_exp} we see  
\begin{align}
|f(Q_\lam) - f(Q_\lam) - f(u^*)| \lesssim |\Lambda Q_\lam|^2 |u^*| + 
|\Lambda Q_\lam| |u^*|^2. 
\end{align}
Since $\|\cl A_0(\lam) g\|_{L^2} \lesssim \| g \|_{H}$ and $\| r^{-1} u^* \|_{L^\infty(r \leq t)} \lesssim t^{\nu-1}$ (see Corollary \ref{l:ustar_app}), the second to last line in the previous expansion can be estimated via Cauchy-Schwarz 
\begin{align}
\Bigl |
\Bigl \la 
\cl A_0(\lam) g \, , \, 
\frac{1}{r^2} &\bigl ( 
f(Q_\lam + u^*) - f(Q_\lam) - f(u^*)
\bigr )
\Bigr \ra 
\Bigr |  \\&\lesssim \| g \|_{H} \Bigl ( 
\int_0^t \bigl ( r^{-2} |u^*|^2 |\Lambda Q_\lam|^4 + r^{-2} |u^*|^4 |\Lambda Q_\lam|^2 \bigr ) \frac{\ud r}{r} 
\Bigr )^{\frac 12} \\ &\lesssim \| g \|_{H} t^{\nu-1} 
\ll \frac{1}{\lam} \| g \|^2_H + t^{\nu-1} 
\end{align}  
as long as $\eta_1$ is sufficiently small. 
We conclude  
\begin{align}  
\begin{split}\label{eq:b_sec_last}
&-\Bigl \la \cl A(\lam) g \, , \, 
\frac{1}{r^2} (f(Q_\lam + u^* + g) - f(Q_\lam + u^*) - g)
\Bigr \ra  \\
&- \Bigl \la \A_0(\lam) g , \Delta g - 
\frac{1}{r^2} (f(u^* + Q_\lam + g) - f(u^*) - f(Q_\lam)
\Bigr \ra_{loc}  \\
&\geq \frac{1}{\lam} \int_0^{R\lam} \Bigl (|\p_r g|^2 + \frac{1}{r^2} |g|^2 \Bigr )\, r \ud r \\
&\quad + 
\Bigl \la 
[\cl A_0(\lam) - \cl A(\lam)]g \, , \, 
\frac{1}{r^2} (f(Q_\lam + u^* + g) - f(Q_\lam + u^*) - g)
\Big \ra \\&\quad + o(1) \bigl ( t^{\nu-1} + \frac{1}{\lam} \| g \|_H^2 \bigr )
\end{split}
\end{align}
The operator $\cl A_0(\lam) - \cl A(\lam)$ is given by multiplication by the function 
$
\frac{1}{2\lam} \bigl ( q''\bigl ( \frac{r}{\lam} \bigr ) + \frac{\lam}{r} q' \bigl ( \frac{r}{\lam} \bigr )\bigr ). 
$ By \eqref{eq:approx-potential-wm} it follows  
\begin{align}
\begin{split}\label{eq:b'_last} 
\Bigl \la 
[\cl A_0(\lam) - \cl A(\lam)]&g \, , \, 
\frac{1}{r^2} (f(Q_\lam + u^* + g) - f(Q_\lam + u^*) - g)
\Big \ra \\&= \frac{1}{\lam} \int_{0}^{\infty} \frac{1}{r^2} \bigl (
f'(Q_\lam) - 1
\bigr ) |g|^2 \, r \ud r \\& \quad+ (c_0 + \| g \|_H) \,O
\Bigl ( \frac{1}{\lam} \| g \|_H^2 
\Bigr )
\end{split}
\end{align}
where $c_0$ is from Lemma \ref{lem:op-A-wm}. The estimates \eqref{b'_third_last}, \eqref{eq:b_sec_last} and \eqref{eq:b'_last} imply  
\begin{align}
b'(t) + 4pq t^{\nu-1} &\geq \frac{1}{\lam} \int_0^{R\lam} \Bigl (|\p_r g|^2 + \frac{1}{r^2} |g|^2 \Bigr )\, r \ud r
+ \frac{1}{\lam} \int_{0}^{\infty} \frac{1}{r^2} \bigl (
f'(Q_\lam) - 1
\bigr ) |g|^2 \, r \ud r \\&\quad - o(1) \Bigl ( t^{\nu-1} + \frac{1}{\lam} \| g \|_H^2 \Bigr ) - C_0 \frac{\lam}{t^2}. 
\end{align} The orthogonality condition $\bigl \la \cl Z_{\uln \lam} \, , \, g \bigr \ra = 0$ implies 
the localized coercivity estimate, 
\begin{align}
	\frac{1}{\lambda}\int_0^{R\lambda}\Big(|\partial_r g|^2 + \frac{1}{r^2}|g|^2\Big)r \ud r  + \frac{1}{\lambda}\int_0^{\infty} \frac{1}{r^2}\big(f'(Q_\lambda)-1\big)|g|^2 \ud r   \ge - \frac{c_1}{\lam} \| g\|_H^2  
\end{align}
(see~\cite[Lemma 5.4, eq. (5.28)]{JJ-AJM} for the proof). The constant $c_1>0$ appearing above can be made small by choosing $R$ sufficiently large, and thus, 
\begin{align}
b'(t) + 4pq t^{\nu-1} \geq O(\lam t^{-2}) - o(1) \Bigl ( t^{\nu-1} + \frac{1}{\lam} \| g \|_H^2 \Bigr )
\end{align}
for $R$ sufficiently large and $\eta_1$ sufficiently small. This concludes the proof. 
\end{proof}

\section{Energy Estimate for the Error}\label{s:energy} 

In this section, we continue our study of wave maps $\bs u$ such that for all $t \in J \subset (0,T_0]$, 
\begin{align}
\bs u(t) = \bs Q_{\lam(t)} + \bs u^*(t) + \bs g(t), \label{eq:u_de}
\end{align}
where $\lam \in C^1(J)$ is the modulation parameter and 
\begin{align}
\bs g(t) = (g(t), \dot g(t)) := \bs u(t) - (\bs Q_{\lam(t)} + \bs u^*(t)), 
\end{align}
is the error term given by Lemma \ref{l:modeq}. 
 In the previous section we obtained differential inequalities for $\lam(t)$ in terms of $\lam(t), t, $ and the size of $\bs g(t)$ (see Proposition \ref{p:modp2}). The goal for this section is to derive an energy estimate for $\bs g(t)$ to close these differential inequalities. 
 
\subsection{Energy estimate for $\bs g(t)$}
The goal of this section is to prove the following energy estimate for $\bs g(t)$. 

\begin{proposition}\label{p:gestimate}
Assume the same hypothesis as in Lemma \ref{l:modeq}. In addition assume
\begin{gather}
\bs u(t,r) = (\pi,0) + \bs u^*(t,r), \quad \forall r \geq t, \\
\frac{\lam(t)}{t} \leq 1, \quad \forall t \in J. 
\end{gather}
There exist $\eta_2 = \eta_2(\nu,q) >  0$, $C = C(\nu,q) > 0$ and $c = c(\nu,q) > 0$ such that if $\sup J < \eta_2$ then for all $t_0, t \in J$ with $t_0 < t$,
\begin{align}\label{eq:gestimate}
c \| g(t) \|_H^2 + \| \dot g(t) \|_{L^2}^2 &\leq - 8pq \int_{t_0}^t \lam'(\tau) \tau^{\nu-1} \, \ud \tau + C \| \bs g(t_0) \|_{\cl H}^2 \\
&\quad + \iota_1(t_0) + \iota_1(t) + \int_{t_0}^t \iota_2(\tau) d \tau , 
\end{align}
where 
\begin{align}
\iota_1(t) &\leq  C \Bigl (
\lam(t)^2 t^{-2} + t^{2\nu-2} \lam(t)^2 \log \frac{t}{\lam(t)} + \lam(t)t^{3\nu-1}
\Bigr ),\\
\iota_2(t) &\leq 
C\Bigl (\lam(t) t^{2\nu-2} +  t^{\nu-1} \|\bs g(t) \|_{\cl H}^2 \Bigr ). \label{eq:iot2_improved}
\end{align}
\end{proposition}

For a finite energy wave map $\bs u = (u, \p_t u)$, consider the local energy contained inside the light cone at time $t$:
\begin{align*}
\cl E_{loc}(\bs u(t)) :=  \pi \int_0^t \Bigl (
|\p_t u(t,r)|^2 + |\p_r u(t,r)|^2 + \frac{\sin^2 u(t,r)}{r^2}
\Bigr ) \, r \ud r. 
\end{align*}
We will obtain Proposition \ref{p:gestimate} from an estimate for the quadratic part of the local energy
\begin{align} \label{eq:Idef} 
\cl Q(\bs g(t),t) &:= \cl E_{loc}(\bs u(t))) - \cl E_{loc}(\bs Q_{\lam(t)} + \bs u^*(t)) \\&\quad- 
\la D\cl E_{loc}(\bs Q_{\lam(t)} + \bs u^*(t)) , \bs g(t) \ra \\
&= 
\la D^2\cl E_{loc}(\bs Q_{\lam(t)} + \bs u^*(t))\bs g(t), \bs g(t) \ra + O(\| \bs g(t)\|_{\cl H(r\leq t)}^3).  
\end{align} 

\begin{proposition}\label{p:quadratic_estimate}
Assume the same hypotheses as in Proposition \ref{p:gestimate}. There exist $\eta_2 = \eta_2(\nu,q) > 0$ and $C = C(\nu,q) > 0$ such that if $\sup J < \eta_2$ then for all $t_0,t \in J$ with $t_0 < t$, 
\begin{align}\label{eq:qestimate}
\cl Q(\bs g(t), t) &= - 8\pi pq  \int_{t_0}^t \lam'(\tau) \tau^{\nu-1} \, \ud \tau + \cl Q(\bs g(t_0), t_0)\\
&\quad + \iota_1(t_0) - \iota_1(t) + \int_{t_0}^t \iota_2(\tau) d \tau , 
\end{align}
where 
\begin{align}
\iota_1(t) &\leq  C \Bigl (
\lam^2(t) t^{-2} + t^{2\nu-2} \lam^2(t) \log \frac{t}{\lam(t)} + \lam(t)t^{3\nu-1} +
\lam(t) t^{\nu-1} \| g(t) \|_{H}
\Bigr ),\\
\iota_2(t) &\leq 
C\Bigl ( t^{2\nu-2} \lam(t) + 
+ t^{\nu-1} \| \bs g(t) \|_{\cl H}^2  \Bigr ). \label{eq:iot2_improved2}
\end{align}
\end{proposition}

The proof of Proposition \ref{p:quadratic_estimate} will occupy the majority of this section. We now give a quick proof of Proposition \ref{p:gestimate} assuming Proposition \ref{p:quadratic_estimate}. 

\begin{proof}[Proof of Proposition \ref{p:gestimate}]
By definition 
\begin{align}
\frac{1}{\pi} \la D^2\cl E_{loc}(Q_{\lam(t)} + \bs u^*(t))\bs g(t), \bs g(t) \ra
&= \int_0^t |\dot g(t)|^2 \, r \ud r + \int_0^t |\p_r g(t)|^2 \, r \ud r \\
&\quad + \int_0^t \frac{f'(Q_{\lam(t)} + u^*(t))}{r^2} |g(t)|^2 \, r \ud r.  
\end{align}
Then it is easy to see for all $t \in J$, there holds
\begin{align}
\cl Q(\bs g(t), t) \lesssim \| \bs g(t) \|_{\cl H}^2 \label{eq:upperestimate_for_Q}.
\end{align}
Towards deriving a lower bound for $\cl Q (\bs g(t),t)$, we first note
\begin{align}
f'(Q_\lam + u^*) = \cos (2 Q_\lam + 2 u^*) =  \cos 2 Q_\lam 
- 2 \cos 2 Q_\lam \sin^2 u^* - \sin 2 Q_\lam \sin 2 u^*. 
\end{align}
For $r \leq t$, $|u^*(t,r)| \lesssim t^{\nu}$ which implies 
\begin{align}
\Bigl | \int_0^t \bigl ( 2\cos 2 Q_{\lam} \sin^2 u^* + \sin 2 Q_{\lam} \sin 2 u^* \bigr ) |g(t)|^2 \frac{\ud r}{r} \Bigr | \lesssim t^{\nu} \int_0^t \frac{|g(t)|^2}{r^2} \, r \ud r.  
\end{align}
Thus,
\begin{align}
\frac{1}{\pi} \la D^2\cl E_{loc}&(\bs Q_{\lam(t)} + \bs u^*(t))\bs g(t), \bs g(t) \ra
\\&= \int_0^t |\dot g(t)|^2 \, r \ud r + \int_0^t |\p_r g(t)|^2 \, r \ud r \\
&\quad + \int_0^t \frac{f'(Q_{\lam(t)})}{r^2} |g(t)|^2 \, r \ud r + 
O(t^\nu) \int_0^t \frac{|g(t)|^2}{r^2} \, r \ud r. 
\end{align} 
Now, since $\bs g(t,r) = (\pi - Q_{\lam(t)}(r), 0)$ for all $r \geq t$, it is simply to see 
\begin{align}
\| g(t) \|_{\cl H(r \geq t)}^2 \simeq \frac{\lam(t)^2}{t^2}, \label{eq:gestimate_outside_light_cone}.
\end{align}
By \eqref{eq:gestimate_outside_light_cone} and the orthogonality condition imposed on $g(t)$, we conclude there exists $c > 0$ such that  
\begin{align}
\frac{1}{\pi} \la D^2\cl E_{loc}&(\bs Q_{\lam(t)} + \bs u^*(t))\bs g(t), \bs g(t) \ra \\
&= \int_0^\infty |\dot g(t)|^2 \, r \ud r + \int_0^\infty |\p_r g(t)|^2 \, r \ud r
+ \int_0^\infty \frac{f'(Q_{\lam(t)})}{r^2} |g(t)|^2 \, r \ud r \\&\quad + 
O(t^\nu) \int_0^\infty \frac{|g(t)|^2}{r^2} \, r \ud r + O(\lam^2(t) t^{-2}) \\
&= \la D^2 \cl E(\bs Q_{\lam(t)}) \bs g(t), \bs g(t) \ra + 
O(t^\nu) \int_0^\infty \frac{|g(t)|^2}{r^2} \, r \ud r + O(\lam^2(t) t^{-2}) \\
&\geq \bigl ( 2 c - O(t^\nu) \bigr ) \| g(t) \|_{H}^2 + \| \dot g(t) \|_{L^2}^2 + O(\lam^2(t) t^{-2}). 
\end{align} 
Thus, for all $t \in J$, 
\begin{align}
\frac{3c}{2} \| g(t) \|_H^2 + \| \dot g(t) \|_{L^2}^2 \leq 
\frac{1}{\pi} \cl Q(\bs g(t), t) + 
O(\lam(t)^2 t^{-2}) \label{eq:Q_lower_estimate}
\end{align}
as long as $\eta_2$ is small enough. 
Inserting \eqref{eq:upperestimate_for_Q} and \eqref{eq:Q_lower_estimate} into \eqref{eq:Q_est} and using $\lam(t) t^{-1} \lesssim \| g(t) \|_{H}$ (see \eqref{eq:gestimate_outside_light_cone}) yields \eqref{eq:gestimate} as long as $\eta_2$ is small enough. 
\end{proof} 

\subsection{Proof of Proposition \ref{p:quadratic_estimate}}

To prove Proposition \ref{p:quadratic_estimate}, we will use the following local energy identity: if $\bs u$ is a finite energy solution to \eqref{eq:wmk} on $J \times \bbR^2$, then for any $t_0, t \in J$ with $t_0 < t$,
\begin{align}\label{eq:locenergy_ident}
\cl E_{loc}(\bs u(t)) = \cl E_{loc}(\bs u(t_0)) + \cl F(\bs u, t_0, t), 
\end{align}
where $\cl F$ is the flux 
\begin{align*}
\cl F(\bs u, t_0, t) = \pi \int_{t_0}^{t} \Bigl (
|(\p_t + \p_r) u(\rho,\rho)|^2 + \frac{\sin^2 u(\rho,\rho)}{\rho^2}
\Bigr ) \, \rho d\rho. 
\end{align*}
Indeed, \eqref{eq:locenergy_ident} is easily verified for smooth wave maps using \eqref{eq:wmk} and the divergence theorem.  The identity \eqref{eq:locenergy_ident} then holds for all finite energy wave maps by a mollification argument.

Consider now $\bs u$ satisfying \eqref{eq:u_de} and $\bs u(t,r) = (\pi,0) + \bs u^*(t,r)$ for all $r \geq t$. 
We perform a Taylor expansion and write 
\begin{align}
\begin{split}\label{eq:loc_taylor}
\cl E_{loc}&(\bs Q_{\lam(t)} + \bs u^*(t) + \bs g(t)) \\&= \cl E_{loc}(\bs Q_{\lam(t)} + \bs u^*(t)) + 
\la D\cl E_{loc}(\bs Q_{\lam(t)} + \bs u^*(t)) , \bs g(t) \ra + \cl Q(\bs g(t),t). 
\end{split}
\end{align}
From \eqref{eq:locenergy_ident}, we conclude 
\begin{align}
\cl E_{loc}&(\bs Q_{\lam(t)} + \bs u^*(t)) + 
\la D\cl E_{loc}(\bs Q_{\lam(t)} + \bs u^*(t)) , \bs g(t) \ra + \cl Q(\bs g(t),t) \\
&=\cl E_{loc}(\bs Q_{\lam(t_0)} + \bs u^*(t_0)) + 
\la D\cl E_{loc}(\bs Q_{\lam(t_0)} + \bs u^*(t_0)) , \bs g(t) \ra + \cl Q(\bs g(t_0),t_0) \\
&\quad + \cl F(\bs u, t_0,t). 
\end{align}
We now extract the leading order contributions for each of the first two terms appearing on the left and right sides of the previous equality sign.  

\begin{lemma}\label{l:locenergyQustar}
For all $t \in J$,
\begin{align}
\cl E_{loc}(\bs Q_{\lam(t)} + \bs u^*(t)) &= \cl E(\bs Q) + \cl E_{loc}(\bs u^*(t)) + 4 \pi \lam(t) u^*(t,t) t^{-1}
\\&\quad + O\Bigl ( 
\lam^2(t) t^{-2} + t^{2\nu-2} \lam^2(t) |\log(t/\lam)| + \lam t^{3\nu-1}
\Bigr )
\end{align}
\end{lemma}

\begin{proof}
We compute 
\begin{align*}
\cl E_{loc}&(\bs Q_{\lam(t)} + \bs u^*(t)) - \cl E(\bs Q) - \cl E_{loc}(\bs u^*(t)) \\
&= \pi \int_t^\infty 
\Bigl ( |\p_r Q_\lam|^2 + \frac{\sin^2 Q_\lam}{r^2} \Bigr ) \, r \ud r
+
2\pi \int_0^t \p_r u^*(t) \p_r Q_\lam \, r\ud r \\
&\quad+ \pi\int_0^t [ \sin^2 (Q_\lam + u^*(t)) - \sin^2 Q_\lam - \sin^2 u^*(t)  ] \frac{\ud r}{r}.  
\end{align*}
The first term is easily seen to be $O(\lam^2/t^2)$. 
Integrating by parts and using $\Delta Q_\lam = \frac{1}{2r^2} \sin 2 Q_\lam$ we obtain
\begin{align*}
2\int_0^t \p_r u^*(t) \p_r Q_\lam \, r\ud r &= -\int_0^t \sin (2 Q_\lam ) u^* \frac{\ud r}{r} + 2\p_r Q_{\lam(t)}(t) u^*(t,t) t.
\end{align*} 
Since $\p_r Q_{\lam(t)}(t) = 2 \lam / t^2 + O(\lam^3/t^4)$ and $|u^*(t,t)| \lesssim t^\nu$, we conclude $\cl E_{loc}(\bs Q_{\lam(t)} + \bs u^*(t)) = \cl E(\bs Q) + \cl E_{loc}(\bs u^*(t)) + 4 \pi \lam(t)u^*(t,t)/t + O(\lam^2/t^2) + \epsilon(t)$ with 
\begin{align*}
\epsilon(t) = \pi \int_0^t 
\left [ \sin^2 (Q_\lam + u^*(t)) - \sin^2 Q_\lam - \sin^2 u^*(t) - \sin(2Q_\lam) u^*(t) \right ] \frac{\ud r}{r}.  
\end{align*} 
Using trigonometric identities we can simplify the previous to
\begin{align*}
\epsilon(t) = -2\pi \int_0^t \sin^2 u^* \sin^2 Q_\lam \frac{\ud r}{r} + 
\frac{\pi}{2} \int_0^t \sin(2Q_\lam) [ \sin 2 u^* - 2u^*] \frac{\ud r}{r}. 
\end{align*}
Since $|u^*(t,r)| \lesssim rt^{\nu-1}$ and $\sin^2 Q_\lam = |\Lambda Q_\lam|^2$, the first integral appearing above satisfies 
\begin{align}
 \int_0^t \sin^2 u^* \sin^2 Q_\lam \frac{\ud r}{r}
 &\lesssim 
 t^{2\nu-2} \int_0^t |\Lambda Q_\lam|^2 r \ud r \\
 & \lesssim t^{2\nu-2} \lam^2 \int_0^{t/\lam} |\Lambda Q|^2 r \ud r \\
 &\lesssim t^{2\nu-2} \lam^2 \log (t/\lam),
\end{align}
and the second integral satisfies 
\begin{align*}
\Bigl | \int_0^t \sin (2 Q_\lam) [\sin 2u^* - 2u^*] \frac{\ud r}{r} \Bigr |
&\lesssim t^{3\nu-3} \int_0^{t} |\sin 2 Q_\lam| r^2 \ud r \\
&\lesssim t^{3\nu-3} \lam^3 \int_0^{t/\lam} |\sin 2Q| r^2 \ud r \\ 
&\lesssim t^{3\nu-1} \lam.  
\end{align*}
We conclude $|\epsilon(t)| \lesssim \lam t^{3\nu-1} + t^{2\nu - 2} \lam^2 |\log(t/\lam)|$ so 
\begin{align*}
\cl E_{loc}(\bs Q_{\lam} + \bs u^*(t)) &= \cl E(\bs Q) + \cl E_{loc}(\bs u^*(t)) 
+ 4\pi\lam(t) u^*(t,t)/t \\ 
&\quad+ O\Bigl ( 
\lam^2/t^2 + t^{2\nu-2} \lam^2 |\log(t/\lam)| + \lam t^{3\nu-1}
\Bigr )
\end{align*}
as desired. 
\end{proof}

\begin{lemma}\label{l:linearsplitting}
For all $t \in J$, 
\begin{align}
\la D\cl E_{loc}(\bs Q_{\lam(t)} + \bs u^*(t)) , \bs g(t) \ra &= 
\la D\cl E_{loc}(\bs u^*(t)), \bs g(t) \ra + O(\lam(t)^2 t^{-2}) \\&\quad + O \Bigl ( \|g(t) \|_{H}
\lam(t) t^{\nu-1} \Bigr )
\end{align}
\end{lemma}

\begin{proof}
We write 
\begin{align*}
\la D\cl E_{loc}&(\bs Q_{\lam(t)} + \bs u^*(t)) , \bs g(t) \ra \\&= 
2 \pi \int_0^t \Bigl ( \p_t u^*(t) \dot g + \p_r u^*(t) \p_r g(t) + \frac{\sin(2 u^*(t))}{2r^2} g(t) \Bigr ) r \ud r \\
&\quad + 2 \pi\int_0^t \Bigl ( \p_r Q_{\lam(t)} \p_r g(t) + \frac{\sin(2Q_{\lam(t)})}{2r^2} g(t) \Bigr ) r \ud r \\
&\quad + \pi\int_0^t \Bigl [
\sin(2 Q_{\lam(t)} + 2 u^*(t)) - \sin(2 Q_{\lam(t)}) - \sin(2u^*(t))
\Bigr ] g(t) \frac{\ud r}{r}. 
\end{align*}
The first term on the right-hand side above is exactly $\la D\cl E_{loc}(\bs u^*(t)), \bs g(t) \ra$. After integrating by parts and using $\Delta Q_\lam = \frac{\sin 2Q_\lam}{2r^2}$ we see the second term satisfies
\begin{align} 
2\int_0^t \Bigl ( \p_r Q_{\lam(t)} \p_r g(t) + \frac{\sin(2Q_{\lam(t)})}{2r^2} g(t) \Bigr ) r \ud r &= 2\p_r Q_{\lam(t)} g(t,t) t \\
&=2\p_r Q_{\lam(t)} (\pi - Q_{\lam(t)}(t)) t \\
&\lesssim (\lam(t)/t^2) (\lam(t)/t) t = \lam(t)^2/t^2.
\end{align}
Using trigonometric identities, we write 
\begin{align}
&\int_0^t \Bigl [
\sin(2 Q_{\lam(t)} + 2 u^*(t)) - \sin(2 Q_{\lam(t)}) - \sin(2u^*(t))
\Bigr ] g(t) \frac{\ud r}{r} \\&\quad= -2 \int_0^t [
\sin(2Q_{\lam(t)}) \sin^2 u^*(t) + \sin (2 u^*(t)) \sin^2 Q_{\lam(t)}
] g(t) \frac{\ud r}{r}. 
\end{align}
Since $\|g \|_{L^\infty(r \leq t)} \lesssim \|g \|_{H(r \leq t)}$, the first integral satisfies 
\begin{align*}
\Bigl |\int_0^t 
\sin(2Q_{\lam(t)}) \sin^2 u^*(t) g(t) \frac{\ud r}{r} \Bigr | &\lesssim
\| g \|_{H(r \leq t)} t^{2\nu-2} \int_0^t |\sin 2 Q_\lam| r \ud r \\
&\lesssim \| g \|_{H(r \leq t)} t^{2\nu-2} \lam^2 \int_0^{t/\lam} |\sin 2 Q| r \ud r  \\
&\lesssim \| g \|_{H(r \leq t)} t^{2\nu-1} \lam.
\end{align*}
For the second integral, we obtain 
\begin{align*}
\Bigl | \int_0^t
\sin (2 u^*(t)) \sin^2 Q_{\lam(t)}
 g(t) \frac{\ud r}{r}
\Bigr | &\lesssim \| g \|_{H(r \leq t)} t^{\nu-1}
\int_0^t |\Lambda Q_{\lam(t)}|^2 \ud r 
 \\&\lesssim \| g \|_{H(r \leq t)}
\lam(t) t^{\nu-1}.
\end{align*}
The lemma follows. 
\end{proof}

\begin{lemma}\label{l:linear_pairing_estimate}
For all $t_0,t \in J$, 
\begin{align}
\la D\cl E_{loc}(\bs u^*(t)), \bs g(t) \ra &= 8 \pi p q \int_{t_0}^t \lam'(\tau) \tau^{\nu-1} \, \ud \tau + 
\la D\cl E_{loc}(\bs u^*(t_0)) , \bs g(t_0) \ra \\&\quad- 4 \pi \lam(t) u^*(t,t) t^{-1} + 4\pi \lam(t_0) u^*(t_0,t_0) t_0^{-1} \\&\quad + 
\int_{t_0}^t 
\gamma(\tau)\, \ud \tau, 
\end{align}
where $\gamma(t)$ satisfies 
\begin{align}
\gamma(t) = 
O \Bigl (
&t^{2\nu-2} \lam(t) 
+ t^{\nu-1} \| \bs g(t) \|_{\cl H}^2 \Bigr ).
\end{align} 
\end{lemma}

\begin{proof}
Define
	\begin{align*}
	\beta(t) &:= \frac{1}{\pi} \la D\cl E_{loc}(\bs u^*(t)), \bs g(t) \ra \\
	&= 2 \int_0^t \Bigl ( \p_t u^*(t) \dot g(t) + \p_r u^*(t) \p_r g(t) + \frac{\sin(2 u^*(t))}{2r^2} g(t) \Bigr ) r \ud r.
	\end{align*}
We compute $\beta'(t)$. 
Since $\dot g(t,r) = 0$ and $g(t,r) = \pi - Q_{\lam(t)}(r)$ for $r \geq t$, we have 
\begin{align*}
\beta'(t) &= -2 \p_r u^*(t,t) \p_r Q_{\lam(t)}(t) t + \frac{\sin 2u^*(t,t)}{t} (\pi - Q_{\lam(t)}(t)) \\
&\quad + 2 \int_0^t \Bigl [ \p_t^2 u^* \dot g + \p_{rt}^2 u^* \p_r g + 
\frac{1}{r^2} f'(u^*) \p_t u^* g \Bigr ] r \ud r \\
&\quad + 2 \int_0^t\Bigl[ \p_t u^* \p_t \dot g  + \p_r u^* \p^2_{tr} g + 
\frac{1}{r^2}f(u^*) \p_t g\Bigr ] r \ud r  
\end{align*}
where again we have $f(\rho) := \frac{1}{2}\sin 2\rho$. 
Using the equations satisfied by $\bs u^*(t)$ and $\bs g(t)$ we can write the last two lines as 
\begin{align}
&2 \int_0^t \Bigl \{  \Delta u^* \dot g + \p_{rt}^2 u^* \p_r g +\frac{1}{r^2}f'(u^*) \p_t u^* g + \p_t u^* \Delta  g + \p_{r} u^* \p^2_{tr} g \\ &\quad \quad - \frac{1}{r^2}[f(u^* + Q_\lam + g) - f(u^*) - f(Q_\lam)]\p_t u^* + \frac{1}{r^2} f(u^*) \lam' \Lambda Q_{\uln{\lam}} \Bigr \} r \ud r. 
\end{align}
Integrating by parts and using $g(t,r) = \pi - Q_\lam(t)(r)$ for $r \geq t$, we have  
\begin{align*}
2 \int_0^t \p^2_{rt} u^* \p_r g \, r \ud r &= 
- 2 \p_t u^*(t,t) \p_r Q_{\lam(t)}(t) t - 2 \int_0^t \p_t u^* \Delta g r \ud r,\\
2 \int_0^t \p_{r} u^* \p^2_{tr} g \, r \ud r &= 
 2 \lam'(t) \p_r u^*(t,t) \Lambda Q_{\uln \lam(t)}(t) t - 2 \int_0^t \Delta u^* \dot g \, r \ud r\\
 &\quad -2 \lam' \int_0^t \Delta u^* \Lambda Q_{\uln \lam} r \ud r. 
\end{align*}
We conclude  
\begin{align*}
\beta'(t) &= -2 \p_r u^*(t,t) \p_r Q_{\lam(t)}(t) t + \frac{\sin 2u^*(t,t)}{t} (\pi - Q_{\lam(t)}(t)) \\
&\quad - 2 \p_t u^*(t,t) \p_r Q_{\lam(t)}(t) t +  2 \lam'(t) \p_r u^*(t,t) \Lambda Q_{\uln \lam(t)}(t) t \\
&\quad - 2 \lam'\int_0^t \Bigl (\Delta u^* - \frac{f(u^*)}{r^2}
 \Bigr ) \Lambda Q_{\uln{\lam}} r \ud r + \gamma(t),
\end{align*}
where 
\begin{align*}
\tilde \gamma(t) :=  2 \int_0^t \p_t u^* \Bigl [ 
f(u^* + Q_\lam + g) - f(u^*) - f(Q_\lam) - f'(u^*) g
\Bigr ] \frac{\ud r}{r}. 
\end{align*}

We now study the size of $\tilde \gamma(t)$.
Adding and subtracting $f(u^* + g)$ in the integrand, using $f(u^* + g) - f(u^*) - f'(u^*)g = O(|g|^2)$ and Corollary \ref{l:ustar_app} we see  
\begin{align*}
\tilde \gamma(t) = - 2\int_0^t \p_t u^* \Bigl [ f(u^* + Q_\lam + g) - f(u^* + g) - f(Q_\lam)\Bigr ] \frac{\ud r}{r} + O\Bigl (
t^{\nu-1} \| g \|_{H(r \leq t)}^2
\Bigr ).
\end{align*} 
Now 
\begin{gather}
 f(u^* + Q_\lam + g) - f(u^* + g) - f(Q_\lam) \\= -\frac{1}{2} \sin (2u^* + 2g) (1- \cos 2 Q_\lam) - \frac{1}{2}(1-\cos (2u^*+2g)) \sin 2 Q_\lam.
\end{gather}
Thus, 
\begin{align*}
\abs{\tilde \gamma(t)} &\lesssim \int_0^t |\p_t u^*| \Bigl [
(|u^*| +|g|)|\cos 2 Q_\lam - 1| + (|u^*|^2 + |g|^2) |\sin 2Q_\lam|
\Bigr ] \frac{\ud r}{r}. 
\end{align*}
By Corollary \ref{l:ustar_app}
\begin{align*}
\int_0^t |\p_t u^*| 
|u^*| |\cos 2 Q_\lam - 1| \frac{\ud r}{r} 
&\lesssim t^{2\nu-2} \int_0^t |\cos 2 Q_\lam - 1| \ud r \\
&\lesssim t^{2\nu-2} \lam,
\end{align*}
as well as 
\begin{align*}
\int_0^t |\p_t u^*| |g| |\cos 2Q_\lam - 1| \frac{\ud r}{r} 
&\lesssim t^{\nu-2} \lam \|g \|_{L^\infty(r \leq t)} \int_0^{t/\lam} |\cos 2Q - 1| \, \ud r \\
&\lesssim t^{\nu-2} \lam \|g \|_{H(r \leq t)}. 
\end{align*}
Via Corollary \ref{l:ustar_app} we also conclude   
\begin{align*}
\int_0^t |\p_t u^*| |u^*|^2 |\sin 2 Q_\lam| \frac{\ud r}{r} + 
\int_0^t |\p_t u^*| |g|^2 |\sin 2Q_\lam| \frac{\ud r}{r}
&\lesssim t^{3\nu}\lam + t^{\nu-1} \| g \|_{H(r \leq t)}^2. 
\end{align*}
In summary, we have proved the linear term 
\begin{align*}
\beta(t) &= \frac{1}{\pi} \la D\cl E_{loc}(\bs u^*(t)), \bs g(t) \ra \\
&= 2 \int_0^t \Bigl ( \p_t u^*(t) \dot g(t) + \p_r u^*(t) \p_r g(t) + \frac{\sin(2 u^*(t))}{2r^2} g(t) \Bigr ) r \ud r,
\end{align*}
satisfies 
\begin{align}
\beta'(t) &= -2 \p_r u^*(t,t) \p_r Q_{\lam(t)}(t) t + \frac{\sin 2u^*(t,t)}{t} (\pi - Q_{\lam(t)}(t)) \\
&\quad - 2 \p_t u^*(t,t) \p_r Q_{\lam(t)}(t) t +  2 \lam'(t) \p_r u^*(t,t) \Lambda Q_{\uln{\lam(t)}}(t) t \\
&\quad- 2 \lam'\int_0^t \Bigl ( \Delta u^* - \frac{f(u^*)}{r^2} \Bigr ) \Lambda Q_{\uln{\lam}} r \ud r + \tilde \gamma(t),
\end{align}
where  
\begin{align}
\tilde \gamma(t) &= O\Bigl ( 
t^{2\nu-2} \lam(t) +  t^{\nu-2} \lam(t)\| g \|_{H(r \leq t)} 
+ t^{\nu-1} \| g \|_{H(r \leq t)}^2
\Bigr ).
\end{align}

We note for all $r \geq 1$,
\begin{gather}\label{eq:Q_est}
\pi  - Q(r) = 2 r^{-1} + O(r^{-3}), \quad \p_r Q(r) = 2 r^{-2} + O(r^{-4}), \\
 \Lambda Q(r) = 2 r^{-1} + O(r^{-3}).
 \end{gather}
By Lemma \ref{l:nonlinear_app} and Lemma \ref{l:linear_app}
\begin{align}\label{eq:ustarest}
t^{-1} |u^*(t,t)| + |\p_t u^*(t,t)| + |\p_r u^*(t,t)| \lesssim t^{\nu-1}
\end{align} 
Thus, we conclude 
\begin{align*}
-2 \p_r u^*(t,t) \p_r Q_{\lam(t)}(t) t &= 
-4 \p_r u^*(t,t) \lam(t) t^{-1} + O(t^{\nu-4} \lam^3(t) ),\\
\frac{\sin 2 u^*(t,t)}{t} (\pi - Q_{\lam(t)}) &= 
4 u^*(t,t) \lam(t) t^{-2} + O( t^{3\nu - 2} \lam(t) + t^{\nu - 4} \lam(t)^3), \\
-2 \p_t u^*(t,t) \p_r Q_{\lam(t)}(t) t &= -4 \p_t u^*(t,t) \lam(t) t^{-1} + O(t^{\nu-4} \lam^3(t) ), \\
2 \lam'(t) \p_r u^*(t,t) \Lambda Q_{\uln \lam(t)}(t) t
&= 4 \lam'(t) \p_r u^*(t,t) + O (
\lam'(t) t^{\nu-3} \lam^2(t) ). 
\end{align*}
Now, integrating by parts and using $(\Delta - \frac{f'(Q_{\lam})}{r^2})\Lambda Q_\lam = 0$ we obtain  
\begin{align*}
-2 \lam' \int_0^t \Bigl (& 
\Delta u^* - \frac{f(u^*)}{r^2} 
\Bigr ) \Lambda Q_{\uln \lam} r \ud r \\
&= - 2\lam'(t) \p_r u^*(t,t) \Lambda Q_{ \uln \lam}(t) t + 2 \lam'(t) u^*(t,t) \p_r \Lambda Q_{\uln \lam}(t) t \\ &\quad - 2 \lam' \int_0^t u^* \Delta Q_{\uln \lam} r \ud r + 2 \lam' \int_0^t f(u^*) \Lambda Q_{\uln \lam} \frac{\ud r}{r} \\
&= - 2\lam'(t) \p_r u^*(t,t) \Lambda Q_{\uln \lam}(t) t + 2 \lam'(t) u^*(t,t) \p_r \Lambda Q_{\uln \lam}(t) t \\ &\quad 
-  \frac{2\lam'}{\lam} \int_0^t \Bigl ( u^* f'(Q_\lam) - f(u^*) \Bigr ) \Lambda Q_{\lam} \, \frac{\ud r}{r}.
\end{align*}
Since $1 - \cos 2 Q_\lam = 2 \sin^2 Q_\lam = 2 (\Lambda Q_\lam)^2$
\begin{align*}
-u^* f'(Q_{\lam}) + f(u^*) 
= 2  u^* \sin^2 Q_\lam + O(|u^*|^3)
= 2 u^* (\Lambda Q_\lam)^2 + O(|u^*|^3).  
\end{align*}
so 
\begin{align}
-\frac{2\lam'}{\lam} \int_0^t \Bigl ( u^* f'(Q_\lam) - f(u^*) \Bigr ) \Lambda Q_{\lam} \, \frac{\ud r}{r} &= 
\frac{4\lam'}{\lam} \int_0^t u^* (\Lambda Q_{\lam})^3 \frac{\ud r}{r} + O
\Bigl (
\lam' t^{3\nu-1}
\Bigr ). 
\end{align}
By Lemma \ref{l:nonlinear_app} and Lemma \ref{l:linear_app} we have for $r \leq t$
\begin{align}
u^*(t,r) = pq t^{\nu-1} r + O(r^3 t^{\nu-3} + r t^{3\nu-1})
\end{align}
Since $\int_0^t (\Lambda Q_\lam)^3 r^2 \ud r \lesssim \lam^3 |\log (\lam/t)|$
we conclude
\begin{align}
-\frac{2\lam'}{\lam} \int_0^t \Bigl ( u^* f'(Q_\lam) - f(u^*) \Bigr ) \Lambda Q_{\lam} \, \frac{\ud r}{r} 
&= 4\lam' \int_0^{t} pq t^{\nu-1} (\Lambda Q_\lam)^3 \frac{\ud r}{\lam} \\
&\qquad+  O
\Bigl ( t^{\nu-3} \lam' \lam^2 |\log (t/\lam)| + \lam' t^{3\nu-1}
\Bigr ) \\
&= 8 pq t^{\nu-1} \lam' \\
&\qquad+  O
\Bigl ( t^{\nu-3} \lam' \lam^2 |\log (t/\lam)| + \lam' t^{3\nu-1}
\Bigr )
\end{align}
Thus, 
\begin{align}
-2 \lam' \int_0^t \Bigl (&\Delta u^* - \frac{f(u^*)}{r^2} \Bigr ) \Lambda Q_{\uln \lam} \, r \ud r \\&= 
-2 \lam'(t) \p_r u^*(t,t) \Lambda Q_{\uln \lam}(t) t + 
2\lam'(t) u^*(t,t) \p_r \Lambda Q_{\uln \lam}(t) t\\
&\quad +8 pq t^{\nu-1} \lam'(t)  +  O
\Bigl (  t^{\nu-3} \lam' \lam^2 |\log (t/\lam)| + \lam'  t^{3\nu-1}
\Bigr ).
\end{align}
 Using \eqref{eq:Q_est}, \eqref{eq:ustarest} and Proposition \ref{p:modp} we conclude 
\begin{align}
-2 \lam' \int_0^t \Bigl (&\Delta u^* - \frac{f(u^*)}{r^2} \Bigr ) \Lambda Q_{\uln \lam} \, r \ud r \\&=
-4 \lam'(t) \p_r u^*(t,t) - 4 \lam'(t) u^*(t,t) t^{-1}
+ 8 pq t^{\nu-1} \lam'(t) \\
&\quad +  O
\Bigl ( t^{\nu-3} \lam^2 \log (t/\lam) \| \dot g \|_{L^2} +  t^{3\nu-1} \| \dot g \|_{L^2}
\Bigr ).   
\end{align}
In summary, we have proven  
\begin{align}
\beta'(t) &= -4 \lam(t) \p_r u^*(t,t) t^{-1} + 4 \lam(t) u^*(t,t) t^{-2} 
+ 4 \lam(t) \p_t u^*(t,t) t^{-1} \\&\quad - 4 \lam'(t) u^*(t,t) t^{-1} + 8 pq \lam'(t) t^{\nu-1} + \gamma(t) \\
&= -4 \frac{\ud }{\ud t} \Bigl ( \lam(t) u^*(t,t) t^{-1} \Bigr ) 
+ 8 pq \lam'(t) t^{\nu-1} + \gamma(t)
\end{align}
where  
\begin{align*}
\gamma(t) = 
O \Bigl (
&t^{2\nu-2} \lam(t) + t^{\nu-2}\lam(t)\|\bs g(t) \|_{\cl H}  
+ t^{\nu-1} \| \bs g(t) \|_{\cl H}^2 + t^{\nu-4}\lam^3(t) \Bigr ),
\end{align*}
Since $\lam(t) t^{-1} \lesssim \| \bs g(t) \|_{\cl H}^2$, the previous implies 
\begin{align*}
\gamma(t) = 
O \Bigl (
&t^{2\nu-2} \lam(t) + t^{\nu-1} \| \bs g(t) \|_{\cl H}^2 \Bigr ),
\end{align*}
The lemma follows upon integrating $\beta'$ from $t_0$ to $t$.  
\end{proof} 

The proof of Proposition \ref{p:quadratic_estimate} now follows from the previous three lemmas.

\begin{proof}[Proof of Proposition \ref{p:quadratic_estimate}] For $t \geq t_0$, we define 
\begin{align}
\iota_1(t) &:= \cl E_{loc}(\bs Q_{\lam(t)} + \bs u^*(t)) - \cl E(\bs Q) - \cl E_{loc}(\bs u^*(t))
- 4 \pi \lam(t) u^*(t,t) t^{-1} \\ &\quad + \la D\cl E_{loc}(\bs Q_{\lam(t)} + \bs u^*(t)) , \bs g(t) \ra - 
\la D\cl E_{loc}(\bs u^*(t)), \bs g(t) \ra, \\
\iota_2(t) &:= -\gamma(t),
\end{align}
where $\gamma(t)$ is as in the statement of Lemma \ref{l:linear_pairing_estimate}. 
By Lemma \ref{l:locenergyQustar}, Lemma \ref{l:linearsplitting} and Lemma \ref{l:linear_pairing_estimate} $\iota_1$ and $\iota_2$ satisfy the desired estimates. The Taylor expansion of the local energy inside of the light cone may then be expressed as 
\begin{align}
\begin{split}\label{eq:loc_taylor2}
\cl E_{loc}&(\bs Q_{\lam(t)} + \bs u^*(t) + \bs g(t)) \\&= \cl E_{loc}(\bs Q_{\lam(t)} + \bs u^*(t)) + 
\la D\cl E_{loc}(\bs Q_{\lam(t)} + \bs u^*(t)) , \bs g(t) \ra + \cl Q(\bs g(t),t) \\
&= \cl E(\bs Q) + \cl E_{loc}(\bs u^*(t))
+ 4 \lam(t) u^*(t,t) t^{-1} + \la D\cl E_{loc}(\bs u^*(t)) , \bs g(t) \ra \\ &\quad+ \iota_1(t) + \cl Q(\bs g(t), t).
\end{split}
\end{align}
Recall the local energy identities satisfied by $\bs u$ and $\bs u^*$:  
\begin{align}
\begin{split}
\cl E_{loc}(\bs u(t)) &=  \cl E_{loc}(\bs u(t_0)) + \cl F(\bs u, t_0,t), \\
\cl E_{loc}(\bs u^*(t)) &=  \cl E_{loc}(\bs u^*(t_0)) + \cl F(\bs u^*, t_0,t).
\end{split}
\label{eq:u_energy_ident2}
\end{align} 
Conservation of energy and \eqref{eq:u_energy_ident2} imply
\begin{align}
\begin{split}
\cl E_{ext}(\bs u(t_0)) &= \cl E_{\ext}(\bs u(t)) + \cl F(\bs u,t_0,t), \\
\cl E_{ext}(\bs u^*(t_0)) &= \cl E_{\ext}(\bs u^*(t)) + \cl F(\bs u^*,t_0,t), \end{split} \label{eq:exteriorenergy_identity}
\end{align}
where 
\begin{align*}
\cl E_{ext}(\bs u(t)) :=  \pi \int_t^\infty \Bigl (
|\p_t u(t,r)|^2 + |\p_r u(t,r)|^2 + \frac{\sin^2 u(t,r)}{r^2}
\Bigr ) \, r \ud r. 
\end{align*}
Since $\bs u(t,r) = (\pi,0) + \bs u^*(t,r)$ for all $r \geq t$, 
 $\cl E_{ext}(\bs u(t)) = \cl E_{ext}(\bs u^*(t))$ for all $t$. By \eqref{eq:exteriorenergy_identity}, we conclude 
\begin{align}
\cl F(\bs u, t_0, t) = \cl F(\bs u^*, t_0, t). 
\end{align} 
Then \eqref{eq:loc_taylor2} and \eqref{eq:u_energy_ident2} imply 
\begin{align}
\cl E(\bs Q) + \cl E_{loc}&(\bs u^*(t))
+ 4 \lam(t) u^*(t,t) t^{-1}  + 
\la D\cl E_{loc}(\bs u^*(t)) , \bs g(t) \ra + \iota_1(t) + \cl Q(\bs g(t), t) \\&= 
\cl E(\bs Q) + \cl E_{loc}(\bs u^*(t_0))
+ 4 \lam(t_0) u^*(t_0,t_0) t_0^{-1}  + 
\la D\cl E_{loc}(\bs u^*(t_0)) , \bs g(t_0) \ra \\&\quad + \iota_1(t_0) + \cl Q(\bs g(t_0), t_0) + \cl F(\bs u^*, t_0, t). 
\end{align}
Since 
\begin{align}
\cl E_{loc}(\bs u^*(t)) &= \cl E_{loc}(\bs u^*(t_0)) + \cl F(\bs u^*,t_0,t), \\
\la D\cl E_{loc}(\bs u^*(t)) , \bs g(t) \ra &= 
8p q \int_{t_0}^t \lam'(\tau) \tau^{\nu-1} \, \ud \tau + \la D\cl E_{loc}(\bs u^*(t_0)) , \bs g(t_0) \ra \\&\quad- 4\lam(t) u^*(t,t) t^{-1} + 4\lam(t_0) u^*(t_0,t_0) t_0^{-1}  - \int_{t_0}^t 
\iota_2(\tau) \ud \tau,
\end{align}
we conclude 
\begin{align}
\cl Q(\bs g(t), t) &= \cl Q(\bs g(t_0), t_0) - 8pq \int_{t_0}^t \lam'(\tau) \tau^{\nu-1} \ud \tau \\
&\quad+ \iota_1(t_0) - \iota_1(t) + \int_{t_0}^t \iota_2(\tau) d \tau , 
\end{align}
as desired. 
\end{proof}  

\section{Construction of the blow-up solution} \label{s:construction} 

In this section we prove part $(a)$ of Theorem \ref{t:main1}.  We
define 
\begin{align}
\lam_c(t) = \frac{p|q|}{\nu^2(\nu+1)} \frac{t^{\nu+1}}{|\log t|}.
\end{align}
Let $\chi \in C^\infty(\bR^2)$ be radial such that $\chi(r) = 1$ for $r \leq \frac{1}{2}$ and $\chi(r) = 0$ if $r \geq 1$.  As before, let $\bs u^*$ denote the unique solution to \eqref{eq:wmk} such that $\bs u^*(0,r) = \bs u^*_0(r) = 
\chi(r)(qr^\nu,0)$.
For $t > 0$, we define 
\begin{align}
\bs v(t) = \bs Q_{\lam_c(t)} + \bs u^*(t) + (1 - \chi_{t})(\pi - Q_{\lam_c(t)},0). 
\end{align}  
We note
\begin{align}
\bs v(t,r) = \bs Q_{\lam_c(t)}(r) + \bs u^*(t,r), \quad \forall r \leq \frac{t}{2},
\end{align}
and
\begin{align}
\bs v(t,r) = (\pi,0) + \bs u^*(t,r), \quad \forall r \geq t. 
\end{align}
The main tool we use to prove part $(a)$ of Theorem \ref{t:main1} is the following proposition.

\begin{proposition}\label{p:close_existence}
	For all $\epsilon > 0$ sufficiently small, there exists $T_1 = T_1(\nu, q, \epsilon) > 0$ with the following property.  For all $T \leq T_1$, and for all $t_0 < T$, the unique finite energy solution $\bs u$ to \eqref{eq:wmk} with initial data $\bs u(t_0) = \bs v(t_0)$ is defined on $[t_0,T]$ and satisfies 
	\begin{align}
	\sup_{t \in [t_0,T]} \left \| 
	\bs u(t) - (\bs Q_{\lam_c(t)} + \bs u^*(t) ) 
	\right \|_{\cl H} < \epsilon. 
	\end{align} 
\end{proposition}

We prove Proposition \ref{p:close_existence} via a bootstrap argument. We will first require a few simple facts.

\subsection{A few lemmas} 

The following two lemmas are simple consequences of Proposition \ref{p:modp2} and Proposition \ref{p:gestimate}.

\begin{lemma}
	\label{p:modp3} 
	Assume the same hypothesis as in  Proposition~\ref{p:modp}. Let $\delta>0$
	be arbitrary and let $\eta_0$ be as in Lemma~\ref{l:modeq}. 
	Let $\zeta(t), b(t)$ be as in~\eqref{eq:zetadef}, \eqref{eq:bdef}. In addition, assume there exist constants $\kappa_1, \kappa_2 > 0$ such that for all $t \in [t_0,T]$ 
	\begin{gather}\label{eq:lam_size_assumption}
	\begin{split}
	\frac{\kappa_1}{2019} \frac{t^{\nu+1}}{|\log t|} \leq \lam(t) \leq
	\kappa_1 \frac{t^{\nu+1}}{|\log t|}, \\
	\| \bs g(t) \|_{\cl H}^2 \leq \kappa_2 t^{2\nu} |\log t|^{-1}.
	\end{split}
	\end{gather}
	Then there exist $T_0 = T_0(\nu,q,\delta,\kappa_1,\kappa_2) > 0$ and $C_0 = C_0(\nu,q,\kappa_1,\kappa_2) > 0$ such that
	if $T \leq T_0$, then for all $t \in [t_0,T]$
	\begin{align}
	\left | \frac{\zeta(t)}{4 \lam(t)\log \frac{t}{\lam(t)}} - 1 \right | &\leq C_0 |\log t|^{-\frac 12} \label{eq:bound-on-l2}, \\
	|b(t)| &\leq 
	( 4\nu + \delta )^{\frac{1}{2}} \left |\log t \right |^{\frac{1}{2}} \| \dot g(t) \|_{L^2} + 
	\delta t^{\nu}, 
	\label{eq:b-bound2}, \\
	|\zeta'(t) - b(t) |&\le \delta t^{\nu}. \label{eq:kala'2}
	\end{align}
	In addition, $b(t)$ is locally Lipschitz and the derivative $b'(t)$ satisfies
	\begin{align}
	&|b'(t)| \leq C_0 t^{\nu-1}, \\
	&b'(t) \ge (-4pq - \delta) t^{\nu-1}. \label{eq:b'lb2}  
	\end{align} 
\end{lemma} 

\begin{proof}
	The assumption \eqref{eq:lam_size_assumption} 
	implies  
	\begin{align}
	\frac{1}{\kappa} \frac{|\log t|}{t^{\nu}}\leq  \frac{t}{\lam(t)} \leq 
	\frac{2019}{\kappa} \frac{|\log t|}{t^{\nu}}.
	\end{align}
	Thus, for all $t \leq T$, 
	\begin{align}
	\log \frac{t}{\lam(t)}  &= \nu |\log t| + O(\log |\log t|). \label{eq:logtoverlam}
	\end{align} 
	The assumption \eqref{eq:lam_size_assumption} and \eqref{eq:logtoverlam} then easily imply the following estimates: 
	\begin{align}
	\begin{split}\label{eq:modp3_est}
	\frac{t}{\lam(t) \log \frac{t}{\lam(t)}} \| \bs g(t) \|_{\cl H} 
	&\lesssim |\log t|^{-\frac 12}, \\
	\| \dot g(t) \|_{L^2} + \frac{\lam(t)}{t}
	&\lesssim t^{\nu} |\log t|^{-\frac 12}, \\
	\frac{\lam(t)}{t^2} &\lesssim t^{\nu-1} |\log t|^{- 1}, \\ 
	\frac{1}{\lam(t)} \| \bs g(t) \|_{\cl H}^2 &\lesssim t^{\nu-1}, 
	\end{split}
	\end{align}
	where the implied constants depend only on $\kappa_1, \kappa_2$. By choosing $T_0$ sufficiently small, \eqref{eq:lam_size_assumption} implies the hypotheses in Proposition \ref{p:modp2} are verified.  The conclusions of Proposition \ref{p:modp3} then follow from those in Proposition \ref{p:modp2}, the estimates \eqref{eq:logtoverlam} and \eqref{eq:modp3_est}. 
\end{proof}

\begin{lemma}\label{l:g_size_with_assumptions}
	Assume there exist constants $\kappa_1, \kappa_2 > 0$ such that for all $t \in [t_0,T]$ 
	\begin{gather}\label{eq:lam_size_assumption2}
	\begin{split}
	\frac{\kappa_1}{2019} \frac{t^{\nu+1}}{|\log t|} \leq \lam(t) \leq
	\kappa_1 \frac{t^{\nu+1}}{|\log t|}, \\
	\| \bs g(t) \|_{\cl H}^2 \leq \kappa_2 t^{2\nu} |\log t|^{-1}.
	\end{split}
	\end{gather}
	There exists $T_0 = T_0(\nu, q, \kappa_1, \kappa_2) > 0$ and $C_1 = C_1(\nu, q,\kappa_1, \kappa_2)$ such that if $T \leq T_0$ then for all $t \in [t_0,T]$, 
	\begin{align}\label{eq:g_estimate}
	c \| g(t) \|_{H}^2 + \| \dot g (t) \|_{L^2}^2
	&\leq -8pq \int_{t_0}^t \lam'(\tau) \tau^{\nu-1} \ud \tau \\ &\quad + C_1 
	\left ( \| \bs g(t_0) \|_{\cl H}^2 + t^{2\nu} |\log t|^{-2} \right ).   
	\end{align}
\end{lemma}

\begin{proof}
	Let $\iota_1$ and $\iota_2$ be as in the statement of Proposition \ref{p:gestimate}. Then using \eqref{eq:lam_size_assumption2} we have for all $t$ sufficiently small
	\begin{align}
	\lam(t)^2 t^{-2} &\lesssim t^{2\nu} |\log t|^{-2}, \\
	t^{2\nu-2} \lam^2(t) |\log t/(\lam(t))| &\lesssim t^{4\nu} |\log t|^{-1}, \\
	\lam(t) t^{3\nu-1} &\lesssim t^{4\nu}|\log t|^{-1},
	\end{align}
	which imply for all $t$ sufficiently small 
	\begin{align}\label{eq:iota_1_bound}
	|\iota_1(t)| \lesssim t^{2\nu} |\log t|^{-2}. 
	\end{align}
	Note the function $f(t) = t^{2\nu} |\log t|^{-2}$ is strictly increasing on $(0,T_0]$ as long as $T_0$ is sufficiently small. This implies for all $t_0 \leq t \leq T_0$ 
	\begin{align}\label{eq:increasing}
	\iota_1(t_0) \lesssim t_0^{2\nu} |\log t_0|^{-2} \lesssim t^{2\nu} |\log t|^{-2}. 
	\end{align}
	Using \eqref{eq:lam_size_assumption2} we have for all $t$ sufficiently small
	\begin{align}
	t^{2\nu-2} \lam(t) + t^{\nu-1} \| \bs g(t) \|_{\cl H}^2 \lesssim t^{3\nu-1} |\log t|^{-1}, 
	\end{align}
	which imply, for all $t$ sufficiently small, 
	\begin{align}\label{eq:iota2_bound}
	\int_{t_0}^t \iota_2(\tau) \ud \tau \lesssim 
	\int_0^t \tau^{3\nu-1} |\log \tau|^{-1} \ud \tau \lesssim 
	t^{3\nu} |\log t|^{-1} \lesssim t^{2\nu} |\log t|^{-2}. 
	\end{align}
	Inserting \eqref{eq:iota_1_bound}, \eqref{eq:increasing} and \eqref{eq:iota2_bound} into Proposition \ref{p:gestimate} finishes the proof.
\end{proof}

The following two lemmas are needed to estimate the size of $\bs v(t_0) - (\bs Q_{\lam_c(t_0)} + \bs u^*(t_0))$ and the size of $Q_{\lam_1} - Q_{\lam_2}$. 

\begin{lemma}\label{l:initialg_size}
For all $t > 0$ sufficiently small, we have 
\begin{align}
\left \|
(1 - \chi_{t})\left (\pi - Q_{\lam_c(t)},0 \right )
\right \|_{\cl H}^2 \lesssim t^{2\nu} |\log t|^{-2}. 
\end{align}
\end{lemma}

\begin{proof}
For $t$ sufficiently small, we have after a change of variable
\begin{align}
\left \|
(1 - \chi_{t})\left (\pi - Q_{\lam_c(t)},0 \right )
\right \|_{\cl H}^2 
&\lesssim \int_{t/2\lam_c(t)}^\infty \Bigl (|\p_r Q|^2 + \frac{|Q - \pi|^2}{r^2} \Big ) \, r \ud r \\ 
&\lesssim \int_{t/2\lam_c(t)}^\infty r^{-4} \, r \ud r \\
&\lesssim \lam_c(t)^2 t^{-2} \lesssim t^{2\nu} |\log t|^{-2}. 
\end{align}
\end{proof}

\begin{lemma}\label{l:difference_est}
Let $\lam_1, \lam_2 > 0$.  Then 
\begin{align}\label{eq:difference_est}
\left \| Q_{\lam_1} - Q_{\lam_2} \right \|_{H} \lesssim  \left |
\log \frac{\lam_1}{\lam_2}
\right |. 
\end{align}
\end{lemma}

\begin{proof}
By the scaling invariance of the estimate we can set $\lam_2 = 1$.  Then, making the change of variables 
\begin{align}
x = \log r, \quad x_0 = \log \lam_1,  \quad f(x) = Q(e^x), 
\end{align}  
we see \eqref{eq:difference_est} is equivalent to showing 
\begin{align}\label{eq:equiv_difference_est}
\left \| f(x - x_0) - f(x) \right \|_{H^1(\bR)} \lesssim |x_0|.  
\end{align}
By the fundamental theorem of calculus
\begin{align}
f(x - x_0) - f(x) = -x_0 \int_0^1 f'(x - t x_0)\, dt.
\end{align}
Since $f'(x) = \Lambda Q(e^x), f''(x) = \Lambda^2 Q(e^x)$, we obtain 
\begin{align}
\| f(x - x_0) - f(x) \|_{H^1(\bR)} 
\lesssim |x_0| \| f' \|_{H^1(\bR)} 
\lesssim |x_0| ( \| \Lambda Q \|_H + \| \Lambda^2 Q \|_H ) \lesssim |x_0|
\end{align}
as desired. 
\end{proof}

\subsection{Proof of Proposition \ref{p:close_existence}}
The proof proceeds via a bootstrap argument.  Let $\epsilon > 0$ be sufficiently small, $T_1$ to be chosen later, and $t_0 < T \leq T_1$. Let $\bs u$ be the unique finite energy solution to \eqref{eq:wmk} with initial data 
$\bs u(t_0) = \bs v^*(t_0)$.  By Lemma \ref{l:initialg_size}
\begin{align}
\left \| \bs u(t_0) - (\bs u^*(t_0) + \bs Q_{\lam_c(t_0)}) \right \|_{\cl H}
\lesssim t_0^{2\nu} |\log t_0|^{-2}, 
\end{align}
and as long as $T_1$ is sufficiently small, 
\begin{align}
\left \la 
u(t_0) - (u^*(t_0) + Q_{\lam_c(t_0)}), \cl Z_{\lam_c(t_0)}  
\right \ra
= \left \la 
(1 - \chi_{t_0})(\pi - Q_{\lam_c(t_0)}), \cl Z_{\lam_c(t_0)}  
\right \ra = 0. 
\end{align}
Thus, there exist a time interval containing $t_0$ and a $C^1$ modulation parameter $\lam(t)$ such that  
\begin{align}
\bs u(t) = \bs Q_{\lam(t)} + \bs u^*(t) + \bs g(t)
\end{align}
with
\begin{align}\label{eq:bootstrap_init}
\lam(t_0) = \lam_c(t_0), \quad \bs g(t_0) = (1 - \chi_{t_0}) 
(\pi - Q_{\lam_c(t_0)}, 0). 
\end{align}

Let $T' \in (t_0,T]$ be the largest time such that for $t \in [t_0,T']$, $\bs u(t)$ is defined and
\begin{gather}
(1 - \epsilon^2) \frac{p|q|}{\nu^2(\nu+1)} \frac{t^{\nu+1}}{|\log t|} 
\leq  \lam(t) \leq (1 + \epsilon) \frac{p|q|}{\nu^2(\nu+1)} \frac{t^{\nu+1}}{|\log t|}, \label{eq:bootstrap_1} \\
c \| g(t) \|_H^2 + \| \dot g(t) \|_{L^2}^2 
\leq \Bigl ( 1 + 2 \epsilon \Bigr ) \frac{4p^2q^2}{\nu^3} \frac{t^{2\nu}}{|\log t|}, \label{eq:bootstrap_2} 
\end{gather}
where $c$ is the constant from Lemma \ref{l:g_size_with_assumptions}. By the preceding discussion, \eqref{eq:bootstrap_init} and Lemma \ref{l:initialg_size}, such a $T'$ exists as long as $T_1$ is sufficiently small. Moreover, by Corollary A.4 of \cite{JJ-AJM} and as long as $T_1$ is sufficiently small, $\bs u(t)$ is defined beyond $T'$.

We first show the bootstrap assumptions imply an improvement to \eqref{eq:bootstrap_2} as long as $\epsilon$ and $T_1$ are sufficiently small.  Let $T_1$ be small enough so Lemma \ref{l:g_size_with_assumptions} applies.  
Then since $t \mapsto t^{2\nu} |\log t|^{-2}$ is increasing and $\| \bs g(t_0) \|_{\cl H}^2 \lesssim t_0^{2\nu} |\log t_0|^{-2}$, it follows there exists $\alpha_0 = \alpha_0(\nu,q) > 0$ such that 
\begin{align}\label{eq:gbound_from_boot}
c \| g(t) \|_{H}^2 + \| \dot g(t) \|_{L^2}^2 \leq 8p|q| \int_{t_0}^t \lam'(\tau) \tau^{\nu-1} \ud \tau + \alpha_0 t^{2\nu} |\log t|^{-2}.
\end{align}
We note since $\frac{\ud }{\ud t} t^{2\nu} |\log t|^{-1} = 2\nu t^{2\nu-1} |\log t|^{-1} + t^{2\nu-1} |\log t|^{-2}$, it follows 
\begin{align}
\int_{t_0}^t \tau^{2\nu-1} |\log \tau|^{-1} \, \ud \tau = 
\frac{1}{2\nu} t^{2\nu} |\log t|^{-1} - \frac{1}{2\nu} t_0^{2\nu} |\log t_0|^{-1}
+ O(t^{2\nu} |\log t|^{-2})
\end{align}
Thus, by integration by parts and \eqref{eq:bootstrap_1} we conclude that as long as $\epsilon$ and $T_1$ are sufficiently small
\begin{align}
 \int_{t_0}^t \lam'(\tau) \tau^{\nu-1} \ud \tau &= 
\lam(t) t^{\nu-1} - \lam(t_0) t_0^{\nu-1} - (\nu-1)\int_{t_0}^t \lam(\tau) \tau^{\nu-2} \ud \tau \\
&\leq \frac{p|q|}{\nu^2(\nu+1)} \Bigl (
(1 + \epsilon) t^{2\nu} |\log t|^{-1} - t_0^{2\nu} |\log t_0|^{-1}
\\ &\quad-(\nu-1)(1-\epsilon^2) \int_{t_0}^t \tau^{2\nu-1} |\log \tau|^{-1} d \tau
\Bigr )  + Ct^{2\nu} |\log t|^{-2} \\
&\leq \frac{p|q|}{2\nu^3}\Bigl (1 + \frac{2\nu}{\nu+1} \epsilon + 
\frac{\nu-1}{\nu+1} \epsilon^2 \Bigr ) t^{2\nu} |\log t|^{-1} \\&\quad + \frac{p|q|}{\nu^2(\nu+1)}\Bigl (-1 + \frac{\nu-1}{2\nu}(1-\epsilon^2) \Bigr ) t_0^{2\nu}|\log t_0|^{-1}  + Ct^{2\nu} |\log t|^{-2}\\
&\leq \frac{p|q|}{2 \nu^3} \Bigl ( 1 + (2 + \epsilon) \frac{\nu}{\nu+1} \epsilon \Bigr ) t^{2\nu} |\log t|^{-1}.   \label{eq:lam_integral}
\end{align}
Combining \eqref{eq:gbound_from_boot} and \eqref{eq:lam_integral} we conclude that as long as $T_1$ is sufficiently small, we have for all $t \in [t_0,T']$   
\begin{align} \label{eq:bootstrap2_improvement}
c \| g(t) \|_H^2 + \| \dot g \|_{L^2}^2 \leq 
\Bigl ( 1 + (1 + \epsilon) \frac{2\nu}{\nu+1} \epsilon \Bigr ) \frac{4p^2q^2}{\nu^3} 
t^{2\nu} |\log t|^{-1}. 
\end{align}
As long as $\epsilon$ is small enough so $$(1+\epsilon)\frac{2\nu}{\nu+1} < 2,$$ \eqref{eq:bootstrap2_improvement} is an improvement of \eqref{eq:bootstrap_2}. 
We now turn to obtaining an improvement on \eqref{eq:bootstrap_1}.

Let $\zeta(t)$ and $b(t)$ be defined as in \eqref{eq:zetadef} and \eqref{eq:bdef}.  Let $T_1$ be small enough so Lemma \ref{p:modp3} applies with $\delta = \epsilon^3$. We conclude from \eqref{eq:b-bound2}, \eqref{eq:b'lb2} and \eqref{eq:bootstrap2_improvement} that as long as $T_1$ is small, there exists $\alpha_1 = \alpha_1(\nu,q) > 0$ such that for all $t \in [t_0,T']$
\begin{align}
\frac{4p|q|}{\nu}(1 - \alpha_1 \epsilon^3) t^{\nu} \leq b(t) \leq (1 + \alpha_1 \epsilon^3)^{\frac{1}{2}} \Bigl (
1 + (1+\epsilon) \frac{2\nu}{\nu+1} \epsilon 
\Bigr )^{\frac 12}\frac{4p|q|}{\nu} t^{\nu} 
\end{align} 
By Taylor's theorem, $(1 + x)^{\frac 12} = 1 + \frac{1}{2} x + O(x^2)$.  Thus, the previous implies that as long as $\epsilon$ is sufficiently small, there exists a constant $\alpha_2 = \alpha_2(\nu,q)$ such that 
\begin{align}
\frac{4p|q|}{\nu}(1 - \alpha_2 \epsilon^3) t^{\nu} \leq b(t) \leq \Bigl (
1 + (1+\alpha_2 \epsilon) \frac{\nu}{\nu+1} \epsilon 
\Bigr )\frac{4p|q|}{\nu} t^{\nu} 
\end{align} 
 By \eqref{eq:kala'2}, \eqref{eq:bootstrap_2} and integration, the previous implies that there exists $\alpha_3 = \alpha_3(q,\nu)$ such that 
\begin{gather}
\frac{4p|q|}{\nu(\nu+1)}(1 - \alpha_3 \epsilon^3) (t^{\nu+1} - t_0^{\nu+1} ) \leq \zeta(t) - \zeta(t_0) \\ \leq \Bigl (
1 + (1+\alpha_3 \epsilon) \frac{\nu}{\nu+1} \epsilon 
\Bigr )\frac{4p|q|}{\nu(\nu+1)} (t^{\nu+1} - t_0^{\nu+1}) \label{eq:alpha2_bound}
\end{gather}
By \eqref{eq:bound-on-l} and \eqref{eq:bootstrap_1} (see \eqref{eq:logtoverlam}) we have 
\begin{align}
\zeta(t) = 4 \nu \lam(t) |\log t| \bigl (1 + O(|\log t|^{-1} \log |\log t|) \bigr ) \label{eq:zeta_relation_to_t}
\end{align}
In particular, 
\begin{align}\label{eq:zeta_t_0}
\zeta(t_0) = \frac{4p|q|}{\nu(\nu+1)} t_0^{\nu+1} \bigl ( 1 + O(|\log t_0|^{-1} \log |\log t_0| ) \bigr )
\end{align}
Inserting \eqref{eq:zeta_relation_to_t} and \eqref{eq:zeta_t_0} into \eqref{eq:alpha2_bound}, we conclude there exists $\alpha_4 = \alpha_4(\nu,q)$ such that as long as $T_1$ is sufficiently small we have 
\begin{align}
\frac{4p|q|}{\nu(\nu+1)}(1 - \alpha_4 \epsilon^3) \frac{t^{\nu+1}}{|\log t|} \leq \lam(t) \leq \Bigl (
1 + (1+\alpha_4 \epsilon) \frac{\nu}{\nu+1} \epsilon 
\Bigr )\frac{4p|q|}{\nu(\nu+1)} \frac{t^{\nu+1}}{|\log t|} 
\label{eq:bootstrap1_improvement}
\end{align}
As long as $\epsilon$ is sufficiently small so 
\begin{align}
 \alpha_4 \epsilon < 1, \quad (1 + \alpha_4 \epsilon) \frac{\nu}{\nu+1} < 1, 
\end{align} 
\eqref{eq:bootstrap1_improvement} is an improvement on the bootstrap assumption \eqref{eq:bootstrap_1}.  Thus, we have proven that as long as $\epsilon$ is sufficiently small and $T_1$ is sufficiently small depending on $\epsilon$, the bootstrap assumptions \eqref{eq:bootstrap_1} and \eqref{eq:bootstrap_2} can be improved on $[t_0,T']$.  We conclude $T' = T$.  In particular, 
\begin{align}
\sup_{t \in [t_0,T]} \left \| \bs u(t) - (\bs u^*(t) + \bs Q_{\lam(t)}) \right \|_{\cl H} \lesssim t^\nu |\log t|^{-1/2}, \label{eq:boundondiff}
\end{align} 
with 
\begin{align}
(1 - \epsilon^2) \lam_c(t)
\leq  \lam(t) \leq (1 + \epsilon) \lam_c(t), \quad \forall t \in [t_0,T]. \label{eq:lam_lamc}
\end{align}
Combining \eqref{eq:boundondiff}, \eqref{eq:lam_lamc} and Lemma \ref{l:difference_est} finishes the proof. 
\qed

\subsection{Proof of part $(a)$ of Theorem \ref{t:main1}}
We prove part $(a)$ of Theorem \ref{t:main1} using Proposition \ref{p:close_existence} and a general scheme for constructing multi-soliton and singular solutions to dispersive equations introduced by Martel \cite{Martel05} and Merle \cite{Merle90}. 
Let $\epsilon > 0$ be small enough so Proposition \ref{p:close_existence} is valid.  Then there exist $T > 0$ and a sequence of times $T > t_0 > t_1 > \ldots$ with $t_n \rar 0$ such that the unique finite energy solution $\bs u_n$ to \eqref{eq:wmk} with initial data $\bs u_n(t_n) = \bs v(t_n)$ is defined on $[t_n,T]$ and satisfies, for all $n \geq 0$ and for all $0 \leq m \leq n$, 
\begin{align}
\sup_{t \in [t_m,t_{m-1}]} \left \| 
\bs u_n(t) - (\bs Q_{\lam_c(t)} + \bs u^*(t) ) 
\right \|_{\cl H} < 2^{-m} \epsilon. \label{eq:exist_1}
\end{align}
Here we set $t_{-1} := T$.  In particular, we have 
\begin{align}
\sup_{t \in [t_n,T]} \left \| 
\bs u_n(t) - (\bs Q_{\lam_c(t)} + \bs u^*(t) ) 
\right \|_{\cl H} < \epsilon. \label{eq:exist_2}
\end{align} 
By Corollary A.6 of \cite{JJ-AJM} we can conclude from \eqref{eq:exist_2} (after shrinking $\epsilon$ and extracting subsequences if necessary) there exists a finite energy solution $\bs u_c$ to \eqref{eq:wmk} defined on the time interval $(0,T]$ such that $\bs u_n(t) \rightharpoonup_n \bs u_c(t)$ for all $t \in (0,T]$. By \eqref{eq:exist_1} and weak convergence, we conclude for all $m \geq 0$, 
\begin{align}
\sup_{t \in [t_m,t_{m-1}]} \left \| 
\bs u_c(t) - (\bs Q_{\lam_c(t)} + \bs u^*(t) ) 
\right \|_{\cl H} \leq 2^{-m}\epsilon. \label{eq:exist_3}
\end{align}
Thus, $\bs u_c$ is the desired blow-up solution. 
\qed 

%
%
\section{Classification of bubbling dynamics determined by the radiation} \label{s:classification} 

In this section we prove  part $(b)$ of Theorem~\ref{t:main2}. We assume there exists a continuous time parameter $\bar \lambda(t)$ such that 
\EQ{  \label{eq:ckls} 
 \bs u(t)  &= \bs Q_{\bar \lam(t)} + \bs u^*(t) + \bs o_{\HH}(1) \mas t  \to 0,  \\
 \frac{\bar \lam(t)}{t}  &\rar 0 \mas t \to 0, 
}
where $\bs u^*(t) \in \HH$ is the unique finite energy wave map with initial data 
$\bs u^*(0,r) = \bs u^*_0(r) = \chi(r)(qr^\nu,0)$. Here, as always, 
$\chi \in C^\infty(\bR^2)$ is radial such that $\chi(r) = 1$ for $r \leq \frac{1}{2}$ and $\chi(r) = 0$ if $r \geq 1$.
By the local Cauchy theory for \eqref{eq:wmk} and finite speed of propagation we have 
\EQ{ \label{eq:u=u*+pi} 
\bs u(t, r) = ( \pi, 0) + \bs u^*(t, r), \quad \forall r \ge t. 
}
It follows from~\eqref{eq:ckls} that 
\EQ{ \label{eq:dto0} 
\|  \bs u(t) -( \bs Q_{\bar \lam(t)} + \bs u^*(t)) \|_{\HH} \to 0 \mas t \to 0. 
}
Therefore, we can find $T_0>0$ small enough such that the hypothesis of Lemma~\ref{l:modeq} are satisfied on the time interval 
\EQ{\label{eq:Jdef} 
J  = (0, T_0].
} 
Let $\lam(t)$ and $\bs g(t)$ be given by Lemma~\ref{l:modeq} so that for all $t \in J$, $\bar \lam(t) \simeq \lam(t)$ and  
\EQ{ \label{eq:lamg} 
\bs u(t) &=  \bs Q_{\lam(t)} + \bs u^*(t) + \bs g(t),  \\
 0 & = \ang{ \calZ_{\U \lam(t)} \mid g(t)}.
}
 From~\eqref{eq:ckls},~\eqref{eq:dto0} and~\eqref{eq:gdotgd} we see
\EQ{ \label{eq:gto0} 
\| \bs g(t) \|_{\HH} \to 0 \mas t \to 0, \quad \frac{\lam(t)}{t} \to 0 \mas t \to 0. 
}
We remark that \eqref{eq:ckls}, \eqref{eq:lamg} and \eqref{eq:gto0} imply 
\begin{align}
\frac{\bar \lam(t)}{\lam(t)} \rar 1, \quad \mbox{as } t \rar 0.
\end{align}
Thus, we may assume $\bar \lam(t) = \lam(t)$ to begin with.  
By Proposition~\ref{p:modp}, 
\EQ{ \label{eq:lam'dotg} 
\abs{\lam'(t)} \lesssim \| \dot g(t) \|_{L^2}.
}
Finally, from~\eqref{eq:u=u*+pi} we see 
\EQ{ \label{eq:g-ext} 
\bs g(t, r) =  (\pi - Q_{\lam(t)}(r),0), \quad \forall r \ge t,
}
and thus we have the bound 
\EQ{ \label{eq:gext-bound} 
 \| \bs g(t) \|_{\HH(r \ge t)} \simeq \frac{\lam^2(t)}{t^2}.
}

\subsection{Focusing nature of the radiation}
We first prove the interaction of the bubble and the radiation must be attractive for blow-up to occur. More precisely, we prove the following. 

\begin{proposition}\label{p:sign}
	Let $\nu > \frac{9}{2}$ and $q \in \bbR$, and let $\bs u$ satisfy \eqref{eq:ckls}.  
	Then
	\begin{align}
	q < 0. 
	\end{align}
\end{proposition}

Towards a contradiction, we assume $q >  0$ (if $q = 0$, then \eqref{eq:ckls} and the variational characterization of $\bs Q$ imply $\exists \lam_0 > 0$ such that $\bs u = \bs Q_{\lam_0}$). We require the following lemmas. 

\begin{lemma}
	Let $\nu > \frac{9}{2}$ and $q > 0$, and let $\bs u$ satisfy \eqref{eq:ckls}.
 There exist constants $C = C(\bs u), c = c(\bs u) > 0$ such that as $t \rar 0$, 
	\begin{align}
	c \| \bs g(t) \|_{\cl H}^2 +  \lam(t) t^{\nu-1}
	&\leq [(\nu-1)+o(1)] \int_0^t \lam(\tau) \tau^{\nu-2} \, \ud \tau + 
	C \frac{\lam(t)^2}{t^2}. \label{eq:glam_apriori}
	\end{align} 
\end{lemma}

\begin{proof}
		Let $\iota_1$ be as in the statement of Proposition \ref{p:gestimate}.
	By \eqref{eq:gto0}, \eqref{eq:gext-bound}
	it follows that 
	\begin{align}
	|\iota_1(t)| \lesssim \frac{\lam(t)^2}{t^2} + \lam(t) t^{2\nu-1}. \label{eq:iota1_estimatess2}
	\end{align}
	Thus, $\lim_{t_0 \rar 0} \iota(t_0) = 0$.
	By Proposition \ref{p:gestimate} and integration by parts 
	\begin{align}
	c \| \bs g(t) \|_{\cl H}^2 + 8pq \lam(t) t^{\nu-1} 
	&\leq (\nu-1) 8pq \int_0^t \lam(\tau) \tau^{\nu-2} \, \ud \tau 
	\\&\quad+ 
	C \frac{\lam(t)^2}{t^2} + C \int_0^t \| \bs g(\tau) \|_{\cl H}^2 \tau^{\nu-1} \, \ud \tau \\
	&\quad  + C \lam(t) t^{2\nu-1} + 
	C \int_0^t \lam(\tau) \tau^{2\nu-2} \, \ud \tau. 
	\end{align}
	Thus, 
	\begin{align}
	c \| \bs g(t) \|_{\cl H}^2 + \lam(t) t^{\nu-1} 
	&\leq [(\nu-1) + o(1)] \int_0^t \lam(\tau) \tau^{\nu-2} \, \ud \tau 
	\\&\quad+ 
	C \frac{\lam(t)^2}{t^2} + C \int_0^t \| \bs g(\tau) \|_{\cl H}^2 \tau^{\nu-1} \, \ud \tau.  
	\end{align}
	By Gronwall's inequality the previous estimate immediately implies \eqref{eq:glam_apriori}. 
\end{proof}

\begin{lemma}\label{p:new_lam_upper}
Let $\nu > \frac{9}{2}$ and $q > 0$, and let $\bs u$ satisfy \eqref{eq:ckls}.
Then
	\begin{align}
	\lam(t) = o(t^{\nu+1}), \quad \mbox{as } t \rar 0.   
	\end{align}
\end{lemma}
\begin{proof}
We claim there exists $C = C(\bs u) > 0$ such that the following holds: for all $R > 0$ sufficiently large, there exists $T_0 = T_0(R) > 0$ such that for all $t \leq T_0$
\begin{align}
\lam(t) \leq C (\log R)^{-1} t^{\nu+1}. \label{eq:lam_logR}
\end{align}
The above claim then immediately implies $\lam(t) = o(t^{\nu+1})$ as desired. 

To prove \eqref{eq:lam_logR}, we first note   
$\p_t g = \dot g + \lambda' \Lambda Q_{\uln \lam}$ implies  
\begin{align}
\frac{\ud }{\ud t} \int_0^{R\lam}\Lambda Q_{\uln{\lam}} g \, r \ud r 
&=
\int_0^{R\lam} \Lambda Q_{\uln{\lam}} \dot g  \, r \ud r
+ \lam' \int_0^{R\lam} |\Lambda Q_{\uln \lam} |^2 \, r \ud r \\  
&\quad-\frac{\lam'}{\lam}
\int_0^{R\lam} |\Lambda_0 \Lambda Q]_{\uln \lam} g \, r \ud r +
R \lam' \Lambda Q_{\uln \lam} g \,r \big |_{r = R\lam}. 
\end{align}
Since $|\lam'| \lesssim \| \dot g \|_{L^2}$, $\Lambda_0 \Lambda Q \in L^1(\bbR^2)$ and $\Lambda Q(R) = O(R^{-1})$, we have 
\begin{align}
\frac{|\lam'|}{\lam}
\Bigl | \int_0^{R\lam} |\Lambda_0 \Lambda Q]_{\uln \lam} g \, r \ud r \Bigr |
\lesssim \| \dot g \|_{L^2} \| g \|_{L^\infty} \lesssim \| \bs g \|_{\cl H}^2,  
\end{align}
as well as 
\begin{align}
\Bigl | \lam' R\Lambda Q_{\uln \lam} g r \big |_{r = R\lam} \Bigr |
\lesssim  \| \dot g \|_{L^2} R \lam^{-1} R^{-1} \| g \|_{L^\infty} R \lam 
\lesssim R \| \bs g \|_{\cl H}^2. 
\end{align}
The previous two estimates imply 
\begin{align} \Bigl |
-\frac{\lam'}{\lam}
\int_0^{R\lam} |\Lambda_0 \Lambda Q]_{\uln \lam} g \, r \ud r +
R \lam' \Lambda Q_{\uln \lam} g \,r \big |_{r = R\lam} \Bigr |
\lesssim R \| \bs g \|^2_{\cl H}
\end{align} 
whence 
\begin{align}
\lam' \| \Lambda Q \|_{L^2(r \leq R)}^2
= -\int_0^{R\lam} \Lambda Q_{\uln \lam} \dot g \, r \ud r 
+ \frac{\ud }{\ud t} \int_0^{R\lam} \Lambda Q_{\uln \lam} g \, r \ud r
+ O(R \| \bs g \|_{\cl H}^2) \label{eq:lam_logR2}
\end{align}
Integrating in time and using Minkowski's inequality we conclude 
\begin{align}
\lam(t) \log R &\lesssim (\log R)^{\frac 12} \int_0^t  \| \bs g(\tau) \|_{\cl H} \, \ud \tau + \lam(t) R \| \bs g(t) \|_{\cl H} + R \int_0^t \| \bs g(\tau) \|_{\cl H}^2 \, \ud \tau. 
\end{align}
Since $\lim_{t \rar 0} \| \bs g(t) \|_{\cl H} = 0$, the second term on the right side above absorbs into the left side and the third term absorbs into the first term (for all $t$ sufficiently small depending on $R$). Thus, for all $t$ sufficiently small,  
\begin{align}
\lam(t) &\lesssim (\log R)^{-\frac 12} \int_0^t  \| \bs g(\tau) \|_{\cl H} \, \ud \tau. \label{eq:lam_logR3}
\end{align}
By \eqref{eq:glam_apriori} we conclude
\begin{align}
\lam(t) &\lesssim (\log R)^{-\frac 12} \int_0^t \Bigl 
[
\int_0^\tau \lam(s) s^{\nu-2} \ud s 
\Bigr ]^{\frac 12} \, \ud \tau + (\log R)^{-\frac 12} \int_0^t \frac{\lam(\tau)}{\tau} \, \ud \tau. 
\end{align} 
Define
$$\alpha(t) := \sup_{\tau \in (0, t]} \lam(\tau) = o(t).$$ Then the previous estimate implies there exists a constant $C > 0$ (uniform in $R$) such that  
\begin{align}
\alpha(t)^{\frac 12} \leq C (\log R)^{-\frac 12} t^{\frac{\nu+1}{2}} 
+ C (\log R)^{-\frac 12} \int_0^t \frac{\alpha(\tau)^{\frac 12}}{\tau} \, \ud \tau. \label{eq:lam_logR4}
\end{align}
Let $R$ be large enough so 
\begin{align}
C(\log R)^{-\frac 12} < \frac{1}{4} \min(\nu+1, 1). \label{eq:R_condition}
\end{align}
Then \eqref{eq:lam_logR4} implies 
\begin{align}
\frac{\ud}{\ud t} \left( t^{-C (\log R)^{-\frac 12}} \int_0^t \frac{\alpha(\tau)^{\frac 12}}{\tau} \, \ud \tau  \right)
\leq C (\log R)^{-\frac 12} t^{\frac{\nu+1}{2} - C (\log R)^{-\frac 12} - 1}. \label{eq:lam_logR5}
\end{align}
We note \eqref{eq:R_condition} and the fact $\alpha(t) = o(t)$ imply  
\begin{align}
\lim_{t \rar 0}  t^{-C (\log R)^{-\frac 12}} \int_0^t \frac{\alpha(\tau)^{\frac 12}}{\tau} \, \ud \tau = 0. 
\end{align}
Integrating \eqref{eq:lam_logR5} then yields 
\begin{align}
\int_0^t \frac{\alpha(\tau)^{\frac 12}}{\tau} \, \ud \tau
\leq C (\log R)^{-\frac 12} t^{\frac{\nu+1}{2}}. \label{eq:lam_logRest}
\end{align}
Inserting \eqref{eq:lam_logRest} into \eqref{eq:lam_logR4} implies that for all $t$ sufficiently small, 
\begin{align}
\alpha(t)^{\frac 12} \leq C (\log R)^{- \frac 12} t^{\frac{\nu+1}{2}}
\end{align}
which yields \eqref{eq:lam_logR}. 
\end{proof}

We now give the short proof of Proposition \ref{p:sign}.

\begin{proof}[Proof of Proposition \ref{p:sign}]
	Assume $q > 0$. By \eqref{eq:glam_apriori} and Proposition \ref{p:new_lam_upper} we conclude 
	\begin{align}
	c \| \bs g(t) \|_{\cl H}^2 + \lam(t) t^{\nu-1}
	&\leq [(\nu-1)+o(1)] \int_0^t \lam(\tau) \tau^{\nu-2} \, \ud \tau \label{eq:glam_apriori2}, 
	\end{align}
	as well as 
	\begin{align}
	\| \bs g(t) \|_{\cl H}^2 \lesssim t^{2\nu} |\log t|^{-1}.
	\end{align}
	Define 
	\begin{align}
	f(t) := c \| \bs g(t) \|_{\cl H}^2 t^{-\nu} + \lam(t) t^{-1}. 
	\end{align}
	Then $f(t) > 0$, $\lim_{t \rar 0} f(t) = 0$ and by \eqref{eq:glam_apriori2} 
	\begin{align}
	f(t) < \nu t^{-\nu} \int_0^t f(\tau) \tau^{\nu-1} \, \ud \tau. \label{eq:stopping_time}
	\end{align}
	Let $T > 0$ be a maximal time for $f$, i.e. 
	\begin{align}
	\forall t \in (0,T], \quad f(t) \leq f(T). 
	\end{align}
	Then \eqref{eq:stopping_time} implies 
	\begin{align}
	f(T) < \nu T^{-\nu} \int_0^T f(\tau) \tau^{\nu-1} \, \ud \tau 
	\leq f(T) T^{-\nu} \int_0^T \nu \tau^{\nu-1} \, \ud \tau = f(T),
	\end{align}
	which is a contradiction.  Thus, $q < 0$. 
\end{proof} 

\subsection{An upper bound for the modulation parameter} 

By Proposition \ref{p:sign} we can assume for the remainder of this work that $q < 0$.
In this section we prove the following preliminary bound on the modulation parameter $\lam(t)$. 

\begin{proposition} \label{t:lamupper} 
Let $\nu > \frac{9}{2}$ and $q < 0$, and let $\bs u$ satisfy \eqref{eq:ckls}.
Let $J = (0, T_0]$ be as in~\eqref{eq:Jdef}, and 
let $\lam(t)$ and $\bs g(t)$ be as in~\eqref{eq:lamg}.
Then, there exists a constant $A = A(\bs u)> 0$ such that  
\EQ{ \label{eq:lamupper} 
\lam(t) \le A \frac{t^{\nu+1}}{\abs{\log{t}}}, \quad \forall t \in J. 
}
\end{proposition} 

Throughout this section, we assume the hypotheses from Proposition \ref{t:lamupper}. We first use Proposition \ref{p:gestimate} to obtain the following bound for $\bs g(t)$.

\begin{lemma}\label{l:g_apriori_bd}
	There exists $C = C(\bs u) > 0$ such that
	\begin{align}
	{\| \bs g(t) \|_{\cl H}^2} \leq C \Bigl (\lam(t) t^{\nu-1}
	+ \frac{\lam(t)^2}{t^2} \Bigr ), \quad \forall t \in J
	. \label{eq:g_apriori_bd}
	\end{align}
\end{lemma}

\begin{proof}
	Let $\iota_1$ be as in the statement of Proposition \ref{p:gestimate}.
	By \eqref{eq:gto0}, \eqref{eq:gext-bound}
	it follows  
	\begin{align}
	|\iota_1(t)| \lesssim \frac{\lam(t)^2}{t^2} + \lam(t) t^{2\nu-1}. \label{eq:iota1_estimatess}
	\end{align}
	Thus, $\lim_{t_0 \rar 0} \iota(t_0) = 0$.
	By Proposition \ref{p:gestimate} and \eqref{eq:iota1_estimatess} and there exist constants $C,c > 0$ such that for all $t$ sufficiently small 
	\begin{align}
	c \| \bs g(t) \|^2_{\cl H} &\leq 8p|q| \int_0^t \lam'(\tau) \tau^{\nu-1} \, \ud \tau + 
	C \left ( \lam(t) t^{2\nu-1} + \frac{\lam(t)^2}{t^2} \right ) \\ &\qquad+ C \int_0^t \bigl (\lam(\tau) \tau^{2\nu-2} + \tau^{\nu-1} \| \bs g(\tau) \|_{\cl H}^2 \bigr ) \, \ud \tau.
	\end{align}
	After integrating by parts, the previous expression implies for all $t$ sufficiently small 
	\begin{align}
	c \| \bs g(t) \|_{\cl H}^2 + 8p|q|&(\nu-1) \int_0^t \lam(\tau) \tau^{\nu-2} \ud \tau \\&\leq 
	8p|q| \lam(t) t^{\nu-1} +  C \left ( \lam(t) t^{2\nu-1} + \frac{\lam^2}{t^2} \right ) \\ &\qquad+ C \int_0^t \bigl (\lam(\tau) \tau^{2\nu-2} + \tau^{\nu-1} \| \bs g(\tau) \|_{\cl H}^2 \bigr ) \ud \tau \\
	&\leq C \lam(t) t^{\nu-1} + C \frac{\lam(t)^2}{t^2} + 
	\frac{8p|q|(\nu-1)}{2} \int_0^t \lam(\tau) \tau^{\nu-2} \ud \tau \\&\qquad+ 
	C \int_0^t \tau^{\nu-1} \| \bs g(t) \|_{\cl H}^2 \ud \tau.  
	\end{align}
	We conclude there exists a constant $C > 0$ such that 
	\begin{align}
	\| \bs g(t) \|_{\cl H}^2 + \int_0^t \lam(\tau) \tau^{\nu-2} \ud \tau 
	\leq C \Bigl ( \lam(t) t^{\nu-1} + \frac{\lam(t)^2}{t^2} + 
	\int_0^t \tau^{\nu-1} \| \bs g(t) \|_{\cl H}^2 \ud \tau
	\Bigr ).
	\end{align}
	By Gronwall's inequality and \eqref{eq:gto0}
	\begin{align}
	\| \bs g(t) \|_{\cl H}^2 + \int_0^t \lam(\tau) \tau^{\nu-2} \ud \tau
	&\leq C \lam(t) t^{\nu-1} + C \frac{\lam(t)^2}{t^2} \\&\quad+ 
	C \int_0^t \bigl (\lam(\tau) \tau^{2\nu-2} + \lam(\tau)^2 \tau^{\nu-3} \bigr ) \ud \tau \\
	&\leq C \Bigl (\lam(t) t^{\nu-1} + \frac{\lam(t)^2}{t^2} \Bigr ) + 
	o(1) \int_0^t \lam(\tau) \tau^{\nu-2} \ud \tau.
	\end{align}
	Absorbing the third term on the right side into the left side above yields
	\begin{align}
	\| \bs g(t) \|_{\cl H}^2 \leq C \Bigl (\lam(t) t^{\nu-1} + \frac{\lam(t)^2}{t^2} \Bigr ),
	\end{align}
	which finishes the proof.
\end{proof}  

Before turning to the proof of Proposition \ref{t:lamupper}, we recall the two auxiliary functions $\zeta(t)$ and $b(t)$ defined in~\eqref{eq:zetadef} and~\eqref{eq:bdef}:  
\begin{align}
\zeta(t) &= 4 \lam(t) \log \frac{t}{\lam(t)} - \int_0^{t} \Lambda Q_{\uln{\lam(t)}} g(t) \, r \ud r, \\
b(t) &= - \int_0^{t} \Lambda Q_{\uln{\lam(t)}} \dot g(t) \, r \ud r  - \la \dot g(t), \A_0(\lambda(t)) g(t)\ra. 
\end{align}
For the convenience of the reader we also recall several of the estimates proved in Proposition~\ref{p:modp2} for these functions: 
\begin{align}
\left | \frac{\zeta(t)}{4 \lam(t)\log \frac{t}{\lam(t)}} - 1 \right | &\leq C_0 \frac{t}{\lam(t) \log \frac{t}{\lam(t)}} \| \bs g(t) \|_{\mathcal H} \label{eq:zlam}, \\
|b(t)| &\leq 
( 4 + o(1) )^{\frac{1}{2}} \left (\log \frac{t}{\lam(t)} \right )^{\frac{1}{2}} \| \dot g(t) \|_{L^2} + 
C_0 \| \bs g(t) \|_{\mathcal H}^2
\label{eq:bupper}, \\
|\zeta'(t) - b(t) |&\le C_0  \left (
\| \dot g(t) \|_{L^2} + 
\frac{\lam(t)}{t}\right ),  \label{eq:z'-b} \\
b'(t) &\ge (4p|q| - o(1)) t^{\nu-1} - C_0 \frac{\lam(t)}{t^2} - \delta \frac{1}{\lam} \| \bs g(t) \|_{\cl H}^2. \label{eq:b'lblast}  
\end{align}

\begin{proof}[Proof of Proposition \ref{t:lamupper}]
Combining~\eqref{eq:z'-b} and~\eqref{eq:bupper} and using the fact $ \| \bs g(t) \|_{\HH} \lesssim 1$ we conclude
\EQ{
\abs{\zeta'(t)}  \lesssim \Big( \log \frac{t}{\lam(t)}\Big)^{\frac{1}{2}}   \| \bs g(t) \|_{\HH}+ \frac{\lam(t)}{t}. 
}
Integrating the above yields, 
\EQ{
 \zeta(t)  \lesssim  \int_{0}^t   \Big( \log \frac{\tau}{\lam(\tau)}\Big)^{\frac{1}{2}}  \| \bs g(\tau) \|_{\HH}\, \ud \tau  +  \int_{0}^t  \frac{\lam(\tau)}{\tau}  \, \ud \tau. 
}
By~\eqref{eq:zlam} 
\EQ{
 \lambda(t) \log \frac{t}{\lam(t)} \lesssim \int_{0}^t \Big( \log \frac{\tau}{\lam(\tau)}\Big)^{\frac{1}{2}}  \| \bs g(\tau) \|_{\HH}\, \ud \tau  +   \int_{0}^t  \frac{\lam(\tau)}{\tau}  \, \ud \tau+ t\| \bs g(t) \|_{\HH}.
}
Plugging in \eqref{eq:g_apriori_bd} to the previous estimate we obtain 
\EQ{
 \lambda(t) \log \frac{t}{\lam(t)} &\lesssim  \int_{0}^t \lam^{\frac{1}{2}}(\tau)  \Big( \log \frac{\tau}{\lam(\tau)}\Big)^{\frac{1}{2}} \tau^{\frac{\nu-1}{2}} \, \ud \tau  \\
 &\quad +  \int_{0}^t \lam(\tau)  \Big( \log \frac{\tau}{\lam(\tau)}\Big) \frac{1}{\tau \Big(\log \frac{\tau}{\lam(\tau)} \Big)^{\frac{1}{2}}}\, \ud \tau  \\ 
  &\quad +  \int_{0}^t \lam(\tau)  \Big( \log \frac{\tau}{\lam(\tau)}\Big) \frac{1}{\tau \Big(\log \frac{\tau}{\lam(\tau)} \Big)}\, \ud \tau  \\ 
 &\quad +  \lam(t)^{\frac{1}{2}} t^{\frac{\nu+1}{2}}   +   \lam(t).
}
Since $\log \frac{t}{\lam(t)}  \to \infty$ as $t \to 0$ the last term on the right-hand side can be absorbed into the left-hand side and the third term on the right absorbs into the second term. The previous estimate reduces to 
\EQ{ \label{eq:lamlogt/lam} 
 \lambda(t) \log \frac{t}{\lam(t)} &\lesssim  \int_{0}^t \lam^{\frac{1}{2}}(\tau)  \Big( \log \frac{\tau}{\lam(\tau)}\Big)^{\frac{1}{2}} \tau^{\frac{\nu-1}{2}} \, \ud \tau  \\
 &\quad +  \int_{0}^t \lam(\tau)  \Big( \log \frac{\tau}{\lam(\tau)}\Big) \frac{1}{\tau (\log \frac{\tau}{\lam(\tau)} )^{\frac{1}{2}}}\, \ud \tau  \\ 
 &\quad +  \lam(t)^{\frac{1}{2}} t^{\frac{\nu+1}{2}}.   
}
Now, define 
\EQ{
\al(t) := \sup_{\tau \in (0, t]} \lambda(\tau)  \log \frac{\tau}{\lam(\tau)}. 
}
Since $\lam(t) = o(t)$, $\alpha(t) > 0$ for all $t$ sufficiently small and  
\EQ{
\al(t) \to 0 \mas t \to 0.
}
From~\eqref{eq:lamlogt/lam} we obtain 
\begin{align}
\alpha(t) \lesssim \alpha(t)^{\frac 12} t^{\frac{\nu+1}{2}} + 
\alpha(t)^{\frac 12} \int_0^t \alpha(\tau)^{\frac 12} \frac{1}{\tau \bigl (\log \frac{\tau}{\lam(\tau)} \bigr )^{\frac 12}} \, \ud \tau. 
\end{align}
Dividing through by $\al(t)^{\frac{1}{2}}$ gives  
\EQ{ \label{eq:al12}
\al(t)^{\frac{1}{2}} & \le   C t^{\frac{\nu+1}{2}}   + C\int_{0}^t \al(\tau)^{\frac{1}{2}} \frac{1}{\tau \bigl (\log\frac{\tau}{\lam(\tau)}\bigr )^{\frac{1}{2}}}\, \ud \tau 
}
for some universal constant $C>0$. Next, we estimate the last term on the right above. Let $\de>0$ be a small constant such that 
\EQ{ \label{eq:1/8} 
\frac{C}{(\log \frac{\tau}{\lam(\tau)} )^{\frac{1}{2}}}  \le  \frac{1}{8}  \mif t < \de,
}
and define additional auxiliary functions 
\EQ{
\gamma(t)&:= -\int^\de_t \frac{C}{ \tau \bigl (\log \frac{\tau}{\lam(\tau)} \bigr )^{\frac{1}{2}}} \, \ud \tau , \quad \gamma'(t) = \frac{C}{ t \bigl (\log \frac{t}{\lam(t)} \bigr )^{\frac{1}{2}}} \\
\te(t)&:= \exp(-\gamma(t))\int_{0}^t \al(\tau)^{\frac{1}{2}}\gamma'(\tau)\, \ud \tau.  
}
Then~\eqref{eq:al12} reads 
\EQ{ \label{eq:al121} 
\al(t)^{\frac{1}{2}} - \int_{0}^t \al(\tau)^{\frac{1}{2}}\gamma'(\tau)\, \ud \tau \le C t^{\frac{\nu+1}{2}}.
}
We first claim $\te(t) \to 0$ as $t \to 0$. To see this, note  
\EQ{
\lam(t) \log \frac{t}{\lam(t)} = t  \frac{\lam(t)}{t} \log \frac{t}{\lam(t)} = o(1) t
}
which implies  
\EQ{
\al(\tau)^{\frac{1}{2}}\gamma'(\tau) \lesssim t^{-\frac{1}{2}}. 
}
Note also by~\eqref{eq:1/8} we have 
\EQ{
-\gamma(t)  \le \frac{1}{8} \int^\de_t \frac{1}{ \tau} \, \ud \tau \le \frac{1}{8} \log \frac{1}{t} = \log t^{-\frac{1}{8}}
}
and therefore 
\EQ{ \label{eq:expgamma} 
\exp( -\gamma(t)) \lesssim t^{-\frac{1}{8}}.
}
Combining these estimates gives 
\EQ{
\te(t) \lesssim t^{-\frac{1}{8}}\int_0^t \tau^{-\frac{1}{2}} \, \ud \tau  \lesssim t^{\frac{3}{8}} \to 0 \mas t \to 0 
}
as claimed. 
Now, using~\eqref{eq:al12} we obtain 
\EQ{
\te'(t)  &= -\gamma'(t) \exp(-\gamma(t))\int_{0}^t \al(\tau)^{\frac{1}{2}}\gamma'(\tau)\, \ud \tau  + \al(t)^{\frac{1}{2}}\exp(-\gamma(t))\gamma'(t) \\
& =\exp(-\gamma(t))\gamma'(t) \Big( \al(t)^{\frac{1}{2}} - \int_{0}^t \al(\tau)^{\frac{1}{2}}\gamma'(\tau)\, \ud \tau \Big) \\
& \le C\exp(-\gamma(t))\gamma'(t) t^{\frac{\nu+1}{2}}, 
}
so integrating from $0$ to $t$, using integration by parts and \eqref{eq:expgamma} we obtain 
\begin{align}
\theta(t) &\leq C \int_0^t \exp(-\gamma(\tau)) \gamma'(\tau) \tau^{\frac{\nu+1}{2}} \, \ud \tau \\
&= C \Bigl (
-\exp(-\gamma(t)) t^{\frac{\nu+1}{2}} + \frac{\nu+1}{2} \int_0^t 
\exp(-\gamma(\tau)) \tau^{\frac{\nu-1}{2}} \, \ud \tau. 
\Bigr ) \\
&\leq C \int_0^t 
\exp(-\gamma(\tau)) \tau^{\frac{\nu-1}{2}} \, \ud \tau.
\end{align}
Since $\log \frac{t}{\lam(t)} \rar 0$, we observe 
\begin{align}
\frac{\ud}{\ud t} \left ( 
\exp(-\gamma(t)) t^{\frac{\nu+1}{2}}
\right ) &= \frac{\nu+1}{2} \exp(-\gamma(t)) t^{\frac{\nu-1}{2}} \\
&\quad - \gamma'(t) \exp(-\gamma(t)) \max(t^{\frac{\nu+1}{2}}, t^{\frac{\mu+1}{2}}) \\
&\geq \frac{\nu+1}{4} \exp(-\gamma(t)) t^{\frac{\nu-1}{2}}
\end{align}
for all $t$ sufficiently small. Thus, $\int_0^t 
\exp(-\gamma(\tau)) \tau^{\frac{\nu-1}{2}} \, \ud \tau \lesssim \exp(-\gamma(t)) t^{\frac{\nu+1}{2}}$ so 
\begin{align}
\theta(t) \lesssim \exp(-\gamma(t)) t^{\frac{\nu+1}{2}}, \label{eq:gron_fin_est}
\end{align}
for all $t$ sufficiently small. From the definition of $\theta(t)$, we conclude for all $t$ sufficiently small, 
\begin{align}
\int_0^t \alpha(\tau)^{\frac 12} \frac{1}{\tau \bigl (
	\log \frac{\tau}{\lam(\tau)}
	\bigr )^{\frac 12}} \, \ud \tau \lesssim t^{\frac{\nu+1}{2}}. \label{eq:estimate_for_second_term_gron}
\end{align}
By the definition of $\alpha(t)$, \eqref{eq:al12} and \eqref{eq:estimate_for_second_term_gron}, we conclude for all $t$, 
\begin{align}
\lam(t) \log \frac{t}{\lam(t)} \lesssim t^{\nu+1}. 
\end{align}
Dividing the previous by $t$ and taking a logarithm yields 
\begin{align}
|\log t| \lesssim \log \frac{t}{\lam(t)}, 
\end{align}
whence it follows
\begin{align}
\lam(t) \lesssim t^{\nu+1} |\log t|^{-1}
\end{align}
which finishes the proof. 
\end{proof} 

\subsection{Determination of the blow-up rate} 

In this section we prove the remaining statements  in Theorem \ref{t:main2} (b). More precisely, we prove the following. 

\begin{proposition}\label{p:scaling_dynamics}
Let $\nu > \frac{9}{2}$ and $q < 0$, and let $\bs u$ satisfy \eqref{eq:ckls}.
Let $\lam(t)$ and $\bs g(t)$ be as in~\eqref{eq:lamg}.
Then 
\begin{align}
\lam(t) = \left (
\frac{p |q|}{\nu^2(\nu+1)} + o(1)
\right ) \frac{t^{\nu+1}}{|\log t|}. 
\end{align}
\end{proposition}

\begin{proof}
We first prove 
\begin{align}
\lam(t) \geq \left (\frac{p|q|}{\nu^2(\nu+1)} - o(1)\right ) \frac{t^{\nu+1}}{|\log t|}, \quad \mbox{as } t \rar 0. \label{eq:lam_lower}
\end{align}
Proposition \ref{t:lamupper} and Lemma \ref{l:g_apriori_bd} imply there exists $C = C(\bs u)$ such that 
	\begin{align}
	\frac{\| \bs g(t) \|_{\cl H}^2}{\lam(t)} \leq C t^{\nu-1}. \label{eq:g_apriori_bd2}
	\end{align}
By \eqref{eq:b'lblast}, \eqref{eq:lamupper} and \eqref{eq:g_apriori_bd2} 
\begin{align}
b'(t) &\geq (4p|q|- o(1)) t^{\nu-1} - C \frac{\lam(t)}{t^2} - o(1) \frac{\| \bs g(t) \|_{\cl H}^2}{\lam(t)} \\
&\geq (4p|q| - o(1)) t^{\nu-1}, \quad \mbox{as } t \rar 0. 
\end{align}
By \eqref{eq:bupper}, \eqref{eq:g_apriori_bd2} and \eqref{eq:gto0} 
\begin{align}
|b(t)| \lesssim t^{\frac 12} \frac{\| \bs g(t) \|}{\lam(t)^{\frac 12}}
+ \| \bs g(t) \|_{\cl H}^2 \rar 0,
\end{align}
as $t \rar 0$.
Thus, $b(t) \geq \left (\frac{4p|q|}{\nu} - o(1)\right ) t^{\nu}$ which implies 
\begin{align}
\int_0^t b(\tau) \, \ud \tau \geq \left (\frac{4p|q|}{\nu(\nu+1)} - o(1)\right ) t^{\nu+1}. \label{eq:intb_lower}
\end{align}
By \eqref{eq:z'-b} and \eqref{eq:zlam} we conclude  
\begin{align}
4\lam(t) \log \frac{t}{\lam(t)} &\geq \int_0^t b(\tau) \, \ud \tau -
C_0 \int_0^t \Bigl ( \| \dot g(\tau) \|_{L^2} + \frac{\lam(\tau)}{\tau} \Bigr ) \, \ud \tau \\
&\qquad- C_0 t \| \bs g(t) \|_{\cl H}. \label{eq:lower_almost}
\end{align}
By \eqref{eq:g_apriori_bd2} and \eqref{eq:lamupper},
\begin{align}
 \int_0^t \Bigl ( \| \dot g(\tau) \|_{L^2} + \frac{\lam(\tau)}{\tau} \Bigr ) \, \ud \tau
+ t \| \bs g(t) \|_{\cl H} \lesssim \frac{t^{\nu+1}}{|\log t|^{\frac 12}}. \label{eq:error_estimates}
\end{align}
Inserting \eqref{eq:intb_lower} and \eqref{eq:error_estimates} into \eqref{eq:lower_almost} implies  
\begin{align}
\lam(t) \log \frac{t}{\lam(t)} \geq \left (\frac{p|q|}{\nu(\nu+1)} - o(1)\right ) t^{\nu+1}. \label{eq:lower_almostalmost}
\end{align}
Dividing both sides of \eqref{eq:lower_almostalmost} by $t$ and taking logarithms implies  
\begin{align}
\log \frac{t}{\lam(t)} \leq  (\nu + o(1)) |\log t|. \label{eq:logestimate}
\end{align}
Inserting \eqref{eq:logestimate} into \eqref{eq:lower_almostalmost} yields 
\begin{align}
\lam(t) \geq \left (\frac{p|q|}{\nu^2(\nu+1)} - o(1)\right ) \frac{t^{\nu+1}}{|\log t|}
\end{align}
which finishes the proof of \eqref{eq:lam_lower}.

Now we prove the sharp upper bound,  
\begin{align}
\lam(t) \leq (1 + o(1))\frac{p|q|}{\nu^2(\nu+1)} \frac{t^{\nu+1}}{|\log t|}, \quad \mbox{as } t \rar 0. \label{eq:sharp_upper_lam}
\end{align}
The estimates \eqref{eq:lamupper} and \eqref{eq:g_apriori_bd2} imply  
\begin{align}
\| \bs g(t) \|_{\cl H}^2 \lesssim t^{2\nu} |\log t|^{-1}
\end{align}
where the implied constant is uniform in $t$. This estimate will be used frequently in what follows. 

Our first step towards proving \eqref{eq:sharp_upper_lam} is to establish the following bounds: 
\begin{gather}
\left |
\int_0^t \Lambda Q_{\uln{\lam(t)}} g(t) \, r \ud r
\right | \lesssim \frac{t^{\nu+1}}{|\log t|^{\frac 12}}, \label{eq:final_lemma2} \\
\left |4\nu |\log t| \lambda'(t) - \frac{\ud}{\ud t} \int_0^t \Lambda Q_{\uln{\lam(t)}} g(t) \, r \ud r 
- b(t) \right | \lesssim \frac{t^\nu}{|\log t|^{\frac 14}}. \label{eq:final_lemma1}
\end{gather}
Indeed, the estimate \eqref{eq:final_lemma2} is quite simple:   
\begin{align}
\left |
\int_0^t \Lambda Q_{\uln{\lam}} g(t) \, r \ud r
\right |
&\lesssim \| g(t) \|_{L^\infty} \int_0^t |\Lambda Q_{\uln \lam}| \, r \ud r \\ 
 &\lesssim \| g(t) \|_{L^\infty} \lam \int_0^{t/\lam} |\Lambda Q| \, r \ud r \\
&\lesssim t \| g(t) \|_{H} \lesssim \frac{t^{\nu+1}}{|\log t|^{\frac 12}}.
\end{align}
We now turn to \eqref{eq:final_lemma1}.  By \eqref{eq:lamupper} and \eqref{eq:lam_lower}, we have 
\begin{align}
\log \frac{t}{\lam(t)} = \nu |\log t|(1 + O(|\log t|^{-1} \log |\log t| )), \quad \mbox{as } t \rar 0. 
\end{align}
Since $\p_t g = \dot g + \lam' \Lambda Q_{\uln \lam}$,
\begin{align}
\frac{\ud}{\ud t} \int_0^t \Lambda Q_{\uln{\lam(t)}} g(t) \, r \ud r 
&=
\int_0^t \Lambda Q_{\uln{\lam}} \dot g  \, r \ud r
  + \lam' \int_0^t |\Lambda Q_{\uln \lam} |^2 \, r \ud r \\  
&\quad-\frac{\lam'}{\lam}
\int_0^t |\Lambda_0 \Lambda Q]_{\uln \lam} g \, r \ud r +
\Lambda Q_{\uln \lam} g \,r \big |_{r = t} \\
&= -b(t) + 4\nu |\log t| \lam' + O(\lam' \log |\log t|) \\
&\quad- \la \dot g\, , \, \cl A(\lam) g \ra + 
\frac{\lam'}{\lam}\int_0^t |\Lambda_0 \Lambda Q]_{\uln \lam} g \, r \ud r +
\Lambda Q_{\uln \lam} g \,r \big |_{r = t}.
\end{align}
The two estimates 
\begin{gather}
\| \bs g(t) \|_{\cl H} \lesssim t^{\nu} |\log t|^{-\frac 12}, \quad 
\int_0^{\infty} |\Lambda_0 \Lambda Q | \, r \ud r \lesssim 1,
\end{gather}
imply 
\begin{align}
|\la \dot g , \mathcal A_0(\lam) g \ra | + \frac{1}{\lam} 
\int_0^t |[\Lambda_0 \Lambda Q]_{\uln \lam}| |g| \, r \ud r \lesssim 
\| \bs g \|_{\cl H}^2 + \| g \|_{L^\infty} \int_0^{t} |\Lambda_0 \Lambda Q | \, r \ud r \lesssim 1. 
\end{align}
Since $g(t,t) = \pi - Q_{\lam(t)}(t) = O(\lam(t) t^{-1})$, 
\begin{align*}
\left |
\Lambda Q_{\uln \lam} g\, r \Big |_{r = t} 
\right | \lesssim \frac{\lam}{t} \lesssim t^{\nu}|\log t|^{-1}. 
\end{align*}
Since $|\lam'| \lesssim \| \dot g\|_{L^2} \lesssim t^{\nu} |\log t|^{-\frac 12}$, we conclude 
\begin{align}
\left |4\nu |\log t| \lambda'(t) - \frac{\ud}{\ud t} \int_0^t \Lambda Q_{\uln{\lam(t)}} g(t) \, r \ud r 
- b(t) \right | \lesssim 
|\lam'| \log |\log t| \lesssim t^{\nu} |\log t|^{-\frac 14}.
\end{align}
which finishes the proof of \eqref{eq:final_lemma1}. 

In the second step of proving \eqref{eq:sharp_upper_lam}, we establish the bound 
	\begin{align}
	c &\| g(t) \|_H^2 + \| \dot g(t) \|_{L^2}^2 \\&\leq (1 + o(1))\frac{2p|q|}{\nu}  \int_0^t b(\tau) \tau^{\nu-1} |\log \tau|^{-1} \, \ud \tau, \quad \mbox{as } t \rar 0, 
	\label{eq:b_lemma}
	\end{align}
	where $c$ is the constant appearing in Proposition \ref{p:gestimate}. 
Indeed, by Proposition \ref{p:gestimate}, \eqref{eq:lamupper} and \eqref{eq:final_lemma1},
\begin{align}
c \| g \|_H^2 + 
\| \dot g\|^2_{L^2} &\leq 8 p|q| \int_0^t \lam'(\tau) \tau^{\nu-1} \, \ud \tau 
+ C t^{2\nu} |\log t|^{-2} \\
&\leq \frac{2p|q|}{\nu} \int_0^t b(\tau) \tau^{\nu-1} |\log \tau|^{-1} \ud \tau \\
&\quad 
+ 8p|q| \int_0^t \frac{\ud}{\ud \tau} \Bigl (\int_0^\tau \Lambda Q_{\uln{\lam(\tau)}} g(\tau) \, r \ud r \Bigr ) \tau^{\nu-1} |\log \tau|^{-1} \, \ud \tau \\
&\quad + C \int_0^t \tau^{2\nu-1} |\log \tau|^{-\frac 54} \, \ud \tau + C t^{2\nu} |\log t|^{-2} \\
&\leq \frac{2p|q|}{\nu} \int_0^t b(\tau) \tau^{\nu-1} |\log \tau|^{-1} \ud \tau \\
&\quad 
+ 8p|q| \int_0^t \frac{\ud}{\ud \tau} \Bigl (\int_0^\tau \Lambda Q_{\uln{\lam(\tau)}} g(\tau) \, r \ud r \Bigr ) \tau^{\nu-1} |\log \tau|^{-1} \, \ud \tau \\
&\quad + C t^{2\nu} |\log t|^{-\frac 54}
\end{align}
By integration by parts and \eqref{eq:final_lemma2},
\begin{align}
\int_0^t & \frac{\ud}{\ud \tau} \Bigl (\int_0^\tau \Lambda Q_{\uln{\lam(\tau)}} g(\tau) \, r \ud r \Bigr ) \tau^{\nu-1} |\log \tau|^{-1} \, \ud \tau \\
&= t^{\nu-1} |\log t|^{-1} \int_0^t \Lambda Q_{\uln{\lam(t)}} g(t) \, r \ud r  \\&\quad-
\int_0^t \int_0^\tau \Lambda Q_{\uln{\lam(\tau)}} g(\tau) \frac{\ud}{\ud \tau}\bigl( \tau^{\nu-1} |\log \tau|^{-1} \bigr )\, r \ud r  \, \ud \tau \\
&\lesssim t^{2\nu} |\log t|^{-\frac 32} + 
\int_0^t \tau^{2\nu-1} |\log \tau|^{-\frac 32} \, \ud \tau \\
&\lesssim t^{2\nu} |\log t|^{-\frac 32}. 
\end{align}
We recall $b(t)$ satisfies the lower bound $b(t) \gtrsim t^{\nu}$ so 
\begin{align}
t^{2\nu} |\log t|^{-\frac 54} \leq o(1) t^{2\nu} |\log t|^{-1} 
\leq o(1) \frac{2p|q|}{\nu} \int_0^t b(\tau) \tau^{\nu-1} \ud \tau.  
\end{align}
We conclude 
\begin{align}
c \| g \|_H^2 + \| \dot g \|^2_{L^2} &\leq \frac{2p|q|}{\nu} \int_0^t b(\tau) \tau^{\nu-1} |\log \tau|^{-1} \ud \tau + C t^{2\nu} |\log t|^{-\frac 54} \\
&\leq (1 + o(1)) \frac{2p|q|}{\nu} \int_0^t b(\tau) \tau^{\nu-1} |\log \tau|^{-1} \, \ud \tau
\end{align}
as desired. 

From the first two steps, the proof of the sharp upper bound \eqref{eq:sharp_upper_lam} follows easily. 
By \eqref{eq:bupper} and \eqref{eq:b_lemma},
\begin{align}
b(t) &\leq (1 + o(1)) (4\nu)^{\frac 12} |\log t|^{\frac 12} \| \dot g \|_{L^2} \\
&\leq (1 + o(1)) (8 p|q|)^{\frac 12} |\log t|^{\frac 12}
\Bigl (\int_0^t b(\tau) \tau^{\nu-1} |\log \tau|^{-1} \, \ud \tau \Bigr )^{\frac 12}
\label{eq:b_intb_estimate}.
\end{align}
Let 
\begin{align}
f(t) = \int_0^t b(\tau) \tau^{\nu-1} |\log \tau|^{-1} \, \ud \tau. 
\end{align}
Then $f$ is positive for all $t$ sufficiently small and $\lim_{t \rar 0} f(t) = 0$. The previous estimate implies 
\begin{align}
f'(t) \leq (1 + o(1)) (8 p|q| )^{\frac 12} t^{\nu-1} |\log t|^{-\frac 12} f(t)^{\frac 12}. 
\end{align}
Thus, 
\begin{gather}
\frac{\ud}{\ud t} [f(t)]^{\frac 12} \leq (1 + o(1)) (2p|q|)^{\frac 12} t^{\nu-1} |\log t|^{-\frac 12},
\end{gather}
and we conclude $f(t)^{\frac 12} \leq (1 + o(1)) (2p|q|)^{\frac 12} \nu^{-1} t^\nu |\log t|^{-\frac 12}$, i.e. 
\begin{align}
\Bigl (\int_0^t b(\tau) \tau^{\nu-1} |\log \tau|^{-1} \, \ud \tau \Bigr )^{\frac 12}
\leq (1 + o(1)) (2p|q|)^{\frac 12} \nu^{-1} t^\nu |\log t|^{-\frac 12} \label{eq:int_b_estimate}. 
\end{align}
Plugging \eqref{eq:int_b_estimate} into \eqref{eq:b_intb_estimate} we obtain 
\begin{align}
b(t) \leq (1 + o(1)) \frac{4p|q|}{\nu} t^{\nu}. \label{eq:optimal_b_estimate}
\end{align}
Then  
\begin{align}
(1 + o(1)) 4\nu |\log t| \lam(t) &= 4 \lam(t) \log \frac{t}{\lam(t)} \\
&\leq \zeta(t) + C t \| \bs g(t) \|_{\cl H} \\
&\leq \int_0^t b(\tau) \, \ud \tau + C
\int_0^t \Bigl (\| \dot g(\tau) \|_{L^2} + \frac{\lam(\tau)}{\tau} \Bigr ) \ud \tau
 + C t \| \bs g(t) \|_{\cl H} \\
&\leq (1 + o(1)) \frac{4p|q|}{\nu(\nu+1)} t^{\nu+1} 
+ C \int_0^t \tau^{\nu} |\log \tau|^{-\frac 12} \, \ud \tau \\&\quad+ C t^{\nu+1} |\log t|^{-\frac 12} \\
&\leq (1 + o(1)) \frac{4p|q|}{\nu(\nu+1)} t^{\nu+1} + C t^{\nu+1} |\log t|^{-\frac 12} \\
&\leq (1 + o(1)) \frac{4p|q|}{\nu(\nu+1)} t^{\nu+1}. 
\end{align}
We conclude
\begin{align}
\lam(t) \leq (1 + o(1)) \frac{p|q|}{\nu^2(\nu+1)} t^{\nu+1} |\log t|^{-1},
\end{align}
which finishes the proof of Proposition~\ref{p:scaling_dynamics} and thus also of Theorem~\ref{t:main1}. 
\end{proof}

\subsection{Faster decay of $g(t)$}

We conclude this work with a simple corollary of Theorem \ref{t:main2}, which shows that the $L^2$ norm of $\dot g(t)$ carries the leading order in $\| \bs g(t) \|_{\HH}$. 

\begin{corollary}
Assume the same hypotheses as Theorem \ref{t:main2}, and let 
\begin{align}
\bs g(t) = (g(t), \dot g(t)) := \bs u(t) - (\bs Q_{\lam(t)} + \bs u^*(t)), 
\end{align}
where $\lam(t) \in C^1(J)$ is the modulation parameter defined on a time interval $J = (0,T_0)$ (so \eqref{eq:lamg} and \eqref{eq:gto0} hold). Define
$\bar b \in C(J)$ and $\dot w \in C(J; L^2(\bbR^2))$ via 
\begin{align}
\bar b(t) &:= -\| \Lambda Q \|_{L^2(r \leq t/\lam(t))}^{-2} \int_0^t \Lambda Q_{\uln{\lam(t)}} \dot g(t) \, r \ud r, \label{eq:def_btilde} \\
\dot w(t) &:= \dot g(t) + \bar b(t) \chi_{[r \leq t]} \Lambda Q_{\uln{\lam(t)}}, \label{eq:def_wdot} 
\end{align}
where $\chi_{[r \leq t]}$ is the sharp cutoff 
\begin{align}
\chi_{[r \leq t]}(r) = 
\begin{cases}
1 &\mbox{ if } r \leq t, \\
0 &\mbox{ if } r > t. 
\end{cases}
\end{align}
Then, as $t \rar 0$, 
\begin{gather}
\bar b(t) = \Bigl ( 
\frac{p |q|}{\nu^2} + o(1)
\Bigr ) \frac{t^\nu}{|\log t|}, \label{eq:barb_leading}\\
\| g(t) \|_H^2 + \| \dot w(t) \|_{L^2}^2 = o\Bigl ( \frac{t^{2\nu}}{|\log t|}\Bigr ), \label{eq:gdotw_norms} \\
\| \dot g(t) \|_{L^2}^2 = \Bigl (
\frac{4p^2 q^2}{\nu^3} + o(1)
\Bigr ) \frac{t^{2\nu}}{|\log t|}. \label{eq:gdot_leading_order}
\end{gather}
\end{corollary}

\begin{proof}
By the definition of $\bar b$ and $\dot w$, $\la \dot w, \Lambda Q_{\uln \lam} \ra = 0$. Thus,   
\begin{align}
\| \dot w \|_{L^2}^2 + (\bar b)^2 \| \Lambda Q \|_{L^2(r \leq t/\lam)}^2 = 
\| \dot g \|_{L^2}^2. \label{eq:norm_of_wdot_norm_g}
\end{align}
By the definition of the auxiliary function $b(t)$, Lemma \ref{l:g_apriori_bd} and Theorem \ref{t:main2},
\begin{align}
b(t) &= - \int_0^t \Lambda Q_{\uln \lam} \dot g \, r \ud r + \la \dot g , \cl A_0(\lam) g \ra \\
&= \| \Lambda Q \|_{r \leq t/\lam}^2 \bar b + O(\| \bs g(t) \|_{\cl H}^2) \\
&= (4\nu - o(1)) |\log t| \bar b + O(t^{2\nu} |\log t|^{-1}). 
\end{align}
By \eqref{eq:bupper} and \eqref{eq:b'lblast},
\begin{align}
b(t) = \Bigl ( \frac{4p|q|}{\nu} + o(1) \Bigr ) t^\nu \label{eq:sharp_b}
\end{align}
and thus, 
\begin{align}
\bar b(t) = \Bigl ( \frac{p|q|}{\nu^2} + o(1) \Bigr ) t^\nu |\log t|^{-1},  \label{eq:sharp_barb}
\end{align}
which establishes \eqref{eq:barb_leading}. 
In particular, we conclude 
\begin{align}
(\bar b)^2 \| \Lambda Q \|_{L^2(r \leq t/\lam)}^2 = \Bigl (
\frac{4 p^2 q^2}{\nu^3} + o(1) 
\Bigr ) t^{2\nu} |\log t|^{-1}. \label{eq:sharp_barb2}
\end{align}
Inserting \eqref{eq:norm_of_wdot_norm_g} into \eqref{eq:b_lemma} and using \eqref{eq:sharp_b} and \eqref{eq:sharp_barb2} yield 
\begin{align}
c \| g \|_H^2 + \| \dot w \|^2_{L^2} 
&\leq (1 + o(1)) \frac{2p|q|}{\nu} \int_0^t b(\tau) \tau^{\nu-1} |\log \tau|^{-1} \, \ud \tau 
\\&\quad - (1 + o(1)) \frac{4 p^2 q^2}{\nu^3} t^{2\nu} |\log t|^{-1} \\
&= (1 + o(1)) \frac{8p^2 q^2}{\nu^2}\int_0^t \tau^{2\nu-1} |\log \tau|^{-1} \, \ud \tau 
\\&\quad - (1 + o(1)) \frac{4 p^2 q^2}{\nu^3} t^{2\nu} |\log t|^{-1} \\
&= 
o \Bigl (t^{2\nu} |\log t|^{-1} \Bigr )
\end{align}
which establishes \eqref{eq:gdotw_norms}. Finally, we obtain \eqref{eq:gdot_leading_order} from inserting \eqref{eq:sharp_barb2} and the previous bound into \eqref{eq:norm_of_wdot_norm_g}.  
\end{proof}

\appendix

\section{Formal derivation of the blow-up rate} \label{a:formal} 
In this section we demonstrate how to formally derive the blow-up rate for a given radiation field $\bs u_0^*$. 
The idea is to separate variables and search for a solution with an expansion in a small parameter $b(t)$, 
\EQ{
\bs u(t) = \bs u^*(t) + \bs Q_{\lam(t)} + b(t) \bs U^{(1)}_{\lam(t)} + b^2(t) \bs U^{(2)}_{\lam(t)} + \dots 
}
where for a pair $\bs U = (U, \dot U)$ the rescaling $\bs U_\lam$ is defined by $\bs U_{\lam} = (U_\lam, \dot U_{\U \lam})$, i.e., the invariant scaling in $\dot H^1 \times L^2$.  And above as usual $\bs u^*(t)$ denotes the wave map evolution of $\bs u^*_0$. We now use the equation~\eqref{eq:wmk} to make informed guesses for the dynamical parameters $(\lam(t), b(t))$ and the profiles $\bs U^{(j)}$. 

First, differentiating the first few terms in the expansion for $u(t)$ we see that  
\EQ{
\p_t u(t)  = \p_t u^*(t)  - \lam'(t) \Lam Q_{\U {\lam(t)}} + \dots 
}
which leads to the guess $b(t) = - \lam'(t)$ and $\bs U^{(1)} = (0, \Lam Q)$. 

Next, we write $\p_t^2 u$ two ways. On the one hand, differentiating in time the terms in the second slot of the expansion, and using $b(t) = \lam'(t)$ yields, 
\EQ{
\p_t^2 u(t) =  \p_t^2 u^*(t) + b'(t) \Lam Q_{\U \lam(t)} +  \frac{b(t)^2 }{\lam(t)} \Lam_0 \Lam Q_{\U{\lam(t)}}  + \dots 
}
On the other hand, if $u(t)$ is to solve~\eqref{eq:wmk} then, using that $\bs u^*(t)$ also solves~\eqref{eq:wmk} we have 
\EQ{
\p_t^2 u &= \Delta u - \frac{1}{r^2} f(u) \\
& = \De( u^* + Q_{\lam} + b^2 U^{(1)}_{\lam} + \dots ) - \frac{1}{r^2} f( u^* + Q_{\lam} + b^2 U^{(1)}_{\lam} + \dots) \\
& = \p_t^2 u^* + b^2 \De U^{(1)}_{\lam}  - \frac{1}{r^2} \Big( f  (u^* + Q_{\lam} + b^2 U^{(1)}_{\lam}) - f(u^*) - f(Q_\lam)  \Big) + \dots  \\
& = \p_t^2 u^* - b^2 \LL_\lam U^{(1)}_{\lam}  \\
&\quad   - \frac{1}{r^2} \Big( f  (u^* + Q_{\lam} + b^2 U^{(1)}_{\lam}) - f(u^*) - f(Q_\lam)  - f'(Q_\lam)b^2U^{(1)}_{\lam}  \Big) + \dots   \\
& = \p_t^2 u^* - b^2 \LL_\lam U^{(1)}_{\lam}   + 2 r^{-2} u^* (\Lambda Q_\lam)^2 + \dots 
}
where in the last line we have used the approximation, 
\EQ{
f(Q_\lam + u^* + b^2 U^{(1)}_\lam) - f(Q_\lam) - f(u^*) - f'(Q_\lam) b^2 U^{(1)}_\lam = -2 u^* (\Lambda Q_\lam)^2 + \dots 
}
which is formally valid in the small-$r$ regime and if we continue ignoring terms that are smaller than $b^2$. 

Setting the right-hand sides of the previous two expression equal to each other, and ignoring lower order terms we arrive at the equation, 
\EQ{ \label{eq:U1}
b^2(t) \LL_\lam U^{(1)}_{\lam} = - b'(t) \Lam Q_{\U \lam(t)} -  \frac{b(t)^2 }{\lam(t)} \Lam_0 \Lam Q_{\U{\lam(t)}}  + 2 r^{-2} u^*(t, \cdot) (\Lambda Q_{\lam(t)})^2 
}
which we will use to choose $b(t)$ and $U^{(1)}$ for given $u^*(t, r)$. 

\subsection{The case $\bs u^*(r) = (q r^{\nu} + o(r^{\nu}), 0)$} \label{s:rnu} 
In this section we find the formal rate for radiation $\bs u^*$ as in Theorem~\ref{t:main1}. In this case we know from Corollary~\ref{l:ustar_app}  that the wave map evolution $\bs u^*(t)$ can be approximated  inside the light cone $r \le t$ by the value that the $4d$ linear evolution $\bs u_L^*(t, r)/r$ takes at the origin at time $t$, i.e., 
\EQ{
u^*(t, r) = p q rt^{\nu-1} + \dots 
}
where $p = p(q, \nu) >0$ is an explicit constant. Replacing $u^*(t, r)$ by the above we reduce~\eqref{eq:U1} to 
\EQ{
b^2(t) \LL_\lam U^{(1)}_{\lam} = - b'(t) \Lam Q_{\U \lam(t)} -  \frac{b(t)^2 }{\lam(t)} \Lam_0 \Lam Q_{\U{\lam(t)}}  + 2pq t^{\nu-1} r^{-1} (\Lambda Q_{\lam(t)})^2 
}
Noting that $\LL_\lam U_{\lam}^{(1)}  =  \frac{1}{\lam} (\LL U^{(1)})_{\U \lam} $ and that $ r^{-1} (\Lambda Q_{\lam(t)})^2  = [r^{-1} (\Lambda Q)^2]_{\U \lam}$ we rescale above and arrive at 
\EQ{
  \LL U^{(1)} = - \Lam_0 \Lam Q -  \frac{\lam}{b^2}  \Big( b'  \Lam Q- 2pq t^{\nu-1}r^{-1} (\Lambda Q)^2 \Big) 
}
A (formal) invocation of the Fredholm alternative, tells us that solving for $U^{(1)}$ above on say the interval $r  \le t/ \lam$, requires that the right-hand side be perpendicular to the kernel of $\LL$ on this interval. This leads to the condition 
\EQ{
 \left(\int_0^{t/ \lam} \abs{\Lam Q}^2 \, r \, \ud r \right) b' &= 2p q t^{\nu-1}\left ( \int_0^{t/\lam}  (\Lambda Q)^3) dr \right ) + \frac{b^2}{\lam} \left (  \int_0^{t/\lam} (\Lambda_0 \Lambda Q ) \Lambda Q r \ud r \right ), 
}
from which we deduce that to leading order, 
\EQ{
4 \log(t/ \lam) b'(t) = 4p qt^{\nu-1} + \dots 
}
We arrive at the formal system, 
\EQ{
 \lam'(t) &= - b(t) \\
 b'(t) & = pq \frac{t^{\nu-1}}{\log( t/\lam)} 
}
In the case $q <0$ we deduce the rate, 
\EQ{
\lam(t) = -\frac{pq}{\nu^2(\nu+1)} \frac{t^{\nu+1}}{|\log t|}(1 + o(1)), \quad \mbox{as } t \rar 0. 
}
which appears in Theorem~\ref{t:main1}.  
\subsection{Other rates} 
Using the same formal argument as above we can find the formal rates associated to different behaviors of the radiation as $r  \to 0$, for which the rigorous methods developed in this paper apply.   In particular, we can recover the pure polynomial blow-up rates $\lam(t) \simeq t^{\nu+1}$ discovered by Krieger, Schlag, and Tataru~\cite{KST}. Indeed, consider radiation $\bs u^*(r)$ such that 
\EQ{
 u_0^*(r) =  - r^\nu \abs{\log r} + o( r^\nu \abs{\log r} ) \mas r \to 0, \quad \dot u_0^*(r) = 0
}
As long is $\nu>\frac{9}{2}$ the same argument given in Section~\ref{s:radiation} shows that the wave map evolution $u^*(t)$ of this data is well approximated by 
\EQ{
u^*(t, r) =-  \ti p  r t^{\nu-1} \abs{ \log t} + \dots 
}
for some explicit constant $\ti p>0$. Arguing exactly as in Section~\ref{s:rnu} we obtain the formal system 
\EQ{
 \lam'(t) &= - b(t) \\
 b'(t) & = -\ti p t^{\nu-1} \frac{\abs{\log t} }{\log( \lam/ t)}
}
from which one deduces the rate 
\EQ{
\lam(t) \simeq t^{\nu+1}
}
Other rates can also be considered. For example, one immediate extension is to consider radiation of the form 
\EQ{
 u_0^*(r) =  - r^\nu \abs{\log r}^\mu , \quad \dot u_0^*(r) = 0
}
from which we find the formal rate
\EQ{
\lam(t) \simeq t^{\nu+1} \abs{\log t}^{\mu -1} 
}

\section{Radiation of the form $\bs u_0^* = (0, q r^{\nu-1} + o(r^{\nu-1})$)}  \label{a:u*dot} 
In this appendix we briefly discuss the case where the radiation takes the form \eqref{eq:u*dot}, and in particular the derivation of the constant~\eqref{eq:tipdef}. In fact, the only different aspect of the proof lies in the analysis of the linear flow in Lemma~\ref{l:linear_app}. 

Consider radiation of the form, 
\EQ{
\bs u_0^*(r) := (u_0^*(r), \dot u_0^*(r)) = \chi(r) ( 0, qr^{\nu-1}).
}
Let $u^*_L = u^*_L(t,r)$ be the solution to the linear equation 
\begin{align}
\begin{split}
		&\partial_t^2 u_L - \partial_r^2 u_L -\frac{1}{r}\partial_r u_L + \frac{1}{r^2} u_L = 0, \quad (t,r) \in \bbR \times (0,\infty), \\
		&\bs u^*_L(0) = \bs u^*_0. 
\end{split}	\label{eq:LW2}
\end{align}

Since the initial data takes the form $ \bs u_0^* = (0, g)$ we take special note here in the distinction between the forward and backwards evolution (oddness in time), as a crucial difference in sign emerges; see Theorem~\ref{t:main1} and Remark~\ref{r:q}. 

\begin{lemma}\label{l:lin-dot} 
	Let $\nu > \frac{9}{2}$.  There exists $C = C(\nu,q) > 0$ such that for all $r \leq  |t|$,
	\begin{align}
		|u_L^*(t,r) - q p  \, t \abs{t}^{\nu - 2}  r | &\leq C r^3  \abs{t}^{\nu - 3},  \label{eq:linearleading1} \\
		|\p_t u^*_L(t,r)| &\leq C r \abs{t}^{\nu-2}, \label{eq:lineartest1}\\
		|\p_r u^*_L(t,r)| &\leq C \abs{t}^{\nu-1}, \label{eq:linearrest1}
	\end{align}
	where 
	\begin{align*}
p(\nu) = 
\frac{\nu+1}{2} \int_0^1 \rho^{\nu+1} (1 - \rho^2)^{-1/2} \, \ud \rho = 
\frac{(\nu + 1)\sqrt{\pi}\Gamma \Bigl ( \frac{2+\nu}{2}\Bigr )}{4 \Gamma \Bigl ( \frac{3+\nu}{2} \Bigr )}.
	\end{align*}
\end{lemma}

\begin{proof}
	For $(t,x) \in \bbR^{1+4}$, define $v(t,x) = u^*_L(t,|x|)/|x|$.  Then $\Box_{\mathbb R^{1+4}} v = 0$ and $\bs v(0) = \frac{1}{r} \bs u_0^*$. For $|x| \leq |t|$, we can express $v(t,x)$ using Kirchoff's formula and a change of variables by 
	\EQ{ \label{eq:kirchoff1}
		v(t,x) &= \frac{1}{8}   \frac{1}{t} \p_t  
		\left (
		t^{3} \abs{t}^{\nu-2} \fint_{|y| \leq 1} q \left | y + \frac{|x|}{t} e_1 \right |^{\nu-2} (1 - |y|^2)^{-1/2} \, \ud y   
		\right ) \\
		&=\left ( \frac{1}{t} \p_t \right ) \left ( t^{3} \abs{t}^{\nu-2} \phi(|x|/t) \right ), 
	}
	where $e_1 = (1,0,0,0)$ and 
	\begin{align}\label{eq:phidef1} 
		 \phi(z) = \frac{1}{8} \fint_{|y| \leq 1} q \left | y + z e_1 \right |^{\nu-2} (1 - |y|^2)^{-1/2} \, \ud y, \quad z\in [-1,1]. 
	\end{align}
	By the dominated convergence theorem, $ \phi(z) \in C^3([-1,-1])$  for $\nu > \frac{9}{2}$.
	A straightforward computation shows 
	\begin{align}\label{eq:v_eq2}
		v(t,x) = t \abs{t}^{\nu - 2} \psi (|x|/t)
	\end{align}
	where 
	$\psi(z) = (\nu+1) \phi(z) -  z  \phi'(z)$. The function $\phi(z)$ is an even function which implies $\psi \in C^2([-1,1])$ is even as well. Thus, there exists a constant $C > 0$ such that 
	\begin{align*}
		| \psi(z) -  \psi(0)| \leq C |z|^2, \quad |z| \leq 1,
	\end{align*}
	which implies 
	\begin{align*}
		|v(t,x) - \psi(0)t  \abs{t}^{\nu - 2}| \leq C \abs{t}^{\nu-3} |x|^2. 
	\end{align*}
	Since $\psi(0) = q p(\nu)$ and $u_L(t,r) = r v(t,r)$ we conclude 
	\begin{align*}
		|u_L(t,r) - q p(\nu) t \abs{t}^{\nu-2} r | \leq C \abs{t}^{\nu-3} r^3, 
	\end{align*}
	which proves \eqref{eq:linearleading1}.
	
	To derive the estimates \eqref{eq:lineartest1} and \eqref{eq:linearrest1} we observe from \eqref{eq:v_eq2} that $v(t,r)$ satisfies for all $r \leq \abs{t}$,  
	\begin{align}
		|v(t,r)| &\leq C\abs{ t}^{\nu-1}, \quad |\nabla_{t,r} v(t,r)| \leq C \abs{t}^{\nu-2}
	\end{align}
	which imply the desired bounds for $u^*_L(t,r) = r v(t,r)$. 
	\end{proof}

With Lemma~\ref{l:lin-dot} in hand, the analogs of Lemma~\ref{l:nonlinear_app} and Corollary~\ref{l:ustar_app} follow as in Section~\ref{s:radiation}. And thus the rest of the proof of Theorem~\ref{t:main1} for radiation as in~\eqref{eq:u*dot} follows in an identical fashion as the detailed arguments carried out in the rest of the paper for radiation as in~\eqref{eq:u*}.

\bibliographystyle{plain}
\bibliography{researchbib}

\bigskip
\centerline{\scshape Jacek Jendrej}
\smallskip
{\footnotesize
 \centerline{CNRS and LAGA, Universit\'e Paris 13}
\centerline{99 av Jean-Baptiste Cl\'ement, 93430 Villetaneuse, France}
\centerline{\email{jendrej@math.univ-paris13.fr}}
} 
\medskip 
\centerline{\scshape Andrew Lawrie}
\smallskip
{\footnotesize
 \centerline{Department of Mathematics, Massachusetts Institute of Technology}
\centerline{77 Massachusetts Ave, 2-267, Cambridge, MA 02139, U.S.A.}
\centerline{\email{alawrie@mit.edu}}
} 
\medskip 
\centerline{\scshape Casey Rodriguez}
\smallskip
{\footnotesize
 \centerline{Department of Mathematics, Massachusetts Institute of Technology}
\centerline{77 Massachusetts Ave, 2-246B, Cambridge, MA 02139, U.S.A.}
\centerline{\email{caseyrod@mit.edu}}
} 

\end{document}